\documentclass[aop]{imsart}

\RequirePackage{amsthm,amsmath,amsfonts,amssymb}
\RequirePackage[numbers]{natbib}

\usepackage{amsmath,amssymb,amsthm}
\usepackage{fancybox,ascmac}
\usepackage{bbm}
\usepackage[dvips]{graphicx}
\usepackage{color}

\usepackage{comment}
\usepackage[utf8]{inputenc}

\startlocaldefs
\numberwithin{equation}{section}
\theoremstyle{plain}

\newtheorem{thm}{Theorem}
\newtheorem{lem}{Lemma}
\newtheorem{cor}{Corollary}

\newtheorem{clm}{Claim}

\newtheorem{prop}{Proposition}\makeatletter
\theoremstyle{remark}
\newtheorem{dfn}{Definition}[section]
\newtheorem{rem}[dfn]{Remark}

\newcommand{\red}[1]{{\color{red}#1}}
\newcommand{\blue}[1]{{\color{blue}#1}}

\newcommand{\ca}{{\rm Cap}}
\newcommand{\Z}{{\mathbb Z}}

\newcommand{\N}{{\mathbb N}}

\newcommand{\D}{{\mathcal D}}

\endlocaldefs

\begin{document}

\begin{frontmatter}
\title{Moderate Deviations for the Capacity of the Random Walk range in dimension four}
\runtitle{Deviations for the Capacity of the Range of a Random Walk}

\begin{aug}
\author[A]{\fnms{Arka}~\snm{Adhikari}\ead[label=e1]{arkaa@stanford.edu}},
\author[B]{\fnms{Izumi}~\snm{Okada}\ead[label=e2]{iokada@math.s.chiba-u.ac.jp}}
\address[A]{Department of Mathematics, Stanford University, Stanford, CA, USA \printead[presep={,\ }]{e1}}

\address[B]{Department of Mathematics and Informatics, Faculty of Science, Chiba University, Chiba 263- 8522, Japan\printead[presep={,\ }]{e2}}
\end{aug}

\begin{abstract}
In this paper, we find a natural four dimensional analog of the moderate deviation results for the capacity of the random walk, which corresponds to Bass, Chen and Rosen \cite{BCR} concerning the volume of the random walk range for $d=2$. We find that the deviation statistics of the capacity of the random walk can be related to the following constant of generalized Gagliardo-Nirenberg inequalities,
\begin{equation*} \label{eq:maxineq}
\inf_{f: \|\nabla f\|_{L^2}<\infty} \frac{\|f\|^{1/2}_{L^2} \|\nabla f\|^{1/2}_{L^2}}{ [\int_{(\mathbb{R}^4)^2} f^2(x) G(x-y) f^2(y) \text{d}x \text{d}y]^{1/4}}.
\end{equation*}
\end{abstract}



\end{frontmatter}
\tableofcontents

\section{Introduction}
In this paper, we study the moderate deviation results for the capacity of the random walk for $d=4$. 
Given an arbitrary set $A$ in $\mathbb{Z}^d$, the capacity of $A$ is defined as follows:
let $\tau_A$ denote the first positive hitting time of a finite set $A$
by a simple random walk $(\mathcal{S}_m)_{m \ge 0}$ on $\mathbb{Z}^d$ and
recall that the corresponding (Newtonian) capacity 
is given for $d \ge 3$, by
\begin{align*}
\ca (A)
:= \sum_{x \in A} 
P^x(\tau_A =\infty)
=\lim_{\|z\|\to \infty}\frac{P^z(\tau_A <\infty)}{G_D(z)}.  
\end{align*}
Here, $G_D$ is the Green's function for the random walk on the lattice. 
$\| \cdot \|$ denotes the  Euclidean distance.

There has been much significant interest in studying the capacity of the range of random walk in $d$-dimensions. As revealed in many other works, understanding the capacity of the range of the random walk relates to questions regarding the volume of a random walk or the intersection of two random walks. This, in turn, has a multitude of applications in various fields. For instance, random walk intersection estimates appear in the study  of quantum field theories \cite{Scy}, conformal field theories \cite{Dup}, and in the study of the self-avoiding walk \cite{Brydge}. {For a more} detailed discussion, one can see the references in \cite{As3}.

In this direction, there are many works {in the mathematical literature} studying the capacity. 
Let $\mathcal{S}[1,n]:=\{\mathcal{S}_1,\ldots,\mathcal{S}_n\}$. 
Jain and Orey \cite{JO68} proved a strong law of large numbers, that is, almost surely, 
$$
\lim_{n \to \infty} \frac{\ca(\mathcal{S}[1,n])}{n} = \alpha_d,  \quad \text{for }d\ge 3
$$
for some constant $\alpha_d$ depending on the dimension. 
If one defines Brownian capacity as,
$$
\text{Cap}_B(D)
:= \bigg(\inf\left\{  \iint G(x-y) \mu(\text{d}x) \mu(\text{d}y): \mu(D) =1\right\} \bigg)^{-1},
$$
and $G$ is the Green's function for the Brownian motion,
then, when $d=3$, Chang \cite{Ch17} has shown that 
$$
\frac{\ca(\mathcal{S}[1,n])}{\sqrt{n}} \stackrel{\D}{\Longrightarrow} \frac{1}{3\sqrt{3}}\text{Cap}_B(B[0,1]).
$$
Here, $B[0,1]$ is the image of the Brownian motion from time $0$ to $1$.

{In addition,} the paper \cite{As1} provides lower and upper bounds for the large deviation of the capacity of the range of a random walk in various dimensions, though without obtaining the optimal constant. The works \cite{As3,As5} also established a law of large numbers and a central limit theorem for the capacity of the range of a random walk in $\mathbb{Z}^4$.  
As a consequence of these results, one conjectures a curious link between the behavior of the capacity in $d$ dimensions and the self-intersection of random walks in $d-2$ dimensions. 

One can observe some of these links when looking at Central Limit Theorem type behavior for the volume of the range of a random walk in two dimensions and the capacity of a walk in four dimensions.
For example, Le Gall, J-F. \cite{LG86} showed that for $d=2$, 
\begin{align*}
    \frac{(\log n)^2}{n} \{\text{Vol}(\mathcal{S}[1, n])- \mathbb{E}\text{Vol}(\mathcal{S}[1, n])\}  \stackrel{\D}{\Longrightarrow} -\pi^2 \gamma([0, 1] ),
\end{align*}
where it is formally defined by understanding a corresponding quantity for Brownian motions: 
\begin{align*}
    \gamma([0, 1])
    :=\int_0^1 \int_0^1 \delta_0(B_s-B_r) \text{d}s\text{d}r
    - \mathbb{E}\left[\int_0^1 \int_0^1 \delta_0(B_s-B_r) \text{d}s\text{d}r\right]. 
\end{align*}
By looking at the form of this equation, it is quite similar to the result of  
Asselah et al. for the central limit theorem behavior of the capacity of a random walk in four dimension. \cite{As5} showed that for $d=4$, 
\begin{align*}
    \frac{(\log n)^2}{n} \{\ca(\mathcal{S}[1, n])- \mathbb{E}\ca(\mathcal{S}[1, n])\}  \stackrel{\D}{\Longrightarrow} -\frac{\pi^2}{4} \gamma_G([0, 1] ),
\end{align*}
where it is also formally defined by looking at a corresponding quantity for Brownian motions: 
\begin{align*}
    \gamma_G([0, 1])
    :=\int_0^1 \int_0^1 G(B_s-B_r) \text{d}s\text{d}r
    - \mathbb{E}\left[\int_0^1 \int_0^1 G(B_s-B_r) \text{d}s\text{d}r\right]. 
\end{align*}

However, as of yet, no deeper mechanism found to explain these parallels. 

More recently, Dembo and the second author \cite{DO} found such a parallel when they wanted  to understand  the more detailed 
 question of a law of iterated logarithms for the capacity.
In four dimensions, the main result of \cite{DO} was the following. 
Then, the following estimates were shown, almost surely, 
\begin{align*}
    &\limsup_{n \to \infty} \frac{\ca(\mathcal{S}[1,n]) - \mathbb{E}[\ca(\mathcal{S}[1,n])]}{\frac{\pi^2}{8} \frac{n \log(\log(\log n))}{(\log n)^2}} = 1, \\
    &\liminf_{n \to \infty} \frac{\ca(\mathcal{S}[1,n]) - \mathbb{E}[\ca(\mathcal{S}[1,n])]}{c_* \frac{n \log(\log n)}{(\log n)^2}} =-1,
\end{align*}
for some constant $c_*>0$. 
Via subadditivity arguments, the upper tail of the law of iterated logarithms can reduce to the computation of an explicit limit. 

By contrast, the constant associated with the lower tail of the large deviation is a far more delicate question. In \cite{DO}, it was only shown that the $\liminf$ exists; the value of the constant depends on quite precise large deviation statistics of the capacity.  However, rather than being merely a technical question, the exact value of the constant can reveal deep connections to other fields. 

 Indeed, much like how Chen et al \cite{Chen05, BCR} showed that the precise value of the large deviation constant for the intersection of random walks was related to the Gagliardo-Nirenberg inequality, we demonstrate here that the constant for the large deviation of the lower tail of the capacity of the random walk range is related to the generalized Gagliardo-Nirenberg inequality. This generalized Gagliardo-Nirenberg inequality was key in the study of the polaron and many other physical processes of interest \cite{FY15, MS}. 
 If we look at  \cite[Theorem 2.3]{FY15}, this inequality is derived from the Hardy-Littlewood-Sobolev inequality and is used to study the Hartree equation.  
Hence, we find a new relationship between the capacity of the random walk and the field of analysis. 
Furthermore, the value of the large deviation constant for the capacity of the random walk range should give great information on the corresponding large deviation statistics of the capacity of the Wiener sausage.

\subsection{Main results}
In our main result, we find that the moderate deviation of $\ca(\mathcal{S}[1,n])$ for $d=4$ is related to best constant of the generalized Gagliardo-Nirenberg inequality (see \cite[(6)]{FY15}).  
 Namely, it is the smallest constant $\tilde{\kappa}(4,2)$ such that the following inequality should hold among $g$ with $\|\nabla g\|_{L^2}<\infty$: 
\begin{align*}
\left[ \int_{(\mathbb{R}^4)^2} g^2(x) G(x-y)g^2(y) \text{d}x \text{d}y \right]^{1/4} \le \tilde{\kappa}(4,2) \|g\|^{1/2}_{L^2} \|\nabla g\|^{1/2}_{L^2},
\end{align*}
where $G(x-y)=2^{-1}\pi^{-2}\|x-y\|^{-2}$ for $d=4$. 

\begin{thm}\label{m1}
Assume $b_n \to \infty$ and $b_n=O(\log \log n)$. 
For $d=4$ and $\lambda>0$, 
\begin{align*}
\lim_{n\to \infty}
\frac{1}{b_n} \log 
\mathbb{P}\left( \ca(\mathcal{S}[1,n]) - \mathbb{E}[\ca(\mathcal{S}[1,n])] \le -\frac{\lambda n}{(\log n)^2}b_n \right)=-I_4(\lambda),
\end{align*}
where 
\begin{align*}
I_4(\lambda)
=\frac{2}{\pi^4} \tilde{\kappa}(4,2)^{-4} \lambda.
\end{align*}
\end{thm}

\begin{cor}\label{c1}
For $d=4$, almost surely,
\begin{equation*}
    \liminf_{n \to \infty} 
   \frac{(\log n)^2}{n \log\log n}
   \bigg(\ca(\mathcal{S}[1,n]) - \mathbb{E}[\ca(\mathcal{S}[1,n])]\bigg)
   =-\frac{\pi^4}{2}\tilde{\kappa}(4,2)^{4}.
\end{equation*}
\end{cor}

\begin{rem}
    We conjecture that the optimal order of $b_n$ is $o(\log n)$ assuming that it is the same as that of the volume of random walks in $d=2$. This conjecture is inspired by the discussion of the rate function for the moderate deviation of the mutual intersection of two random walks in dimension 2 as discussed in \cite[Section 7.2]{Chenbook}. However, just as in the case of the volume, obtaining the optimal scale of large deviations is challenging, and we chose the scale $b_n = O(\log \log n)$ as this is sufficient to obtain the exact constant for the law of the iterated logarithm. However, we do expect that with some technical improvements, our methods can get much closer to the optimal scale of large deviations.   
\end{rem}

\subsection{Strategy}

As mentioned before, to find the exact value of the constant associated with the lower tail of the law of the iterated logarithms, one would need to first prove a form of the large deviation principle. 
To do this, one would need to have control over exponential moments of the quantity in question. Now, one can find some control over such moments in the works of \cite{DO}. However, if one exactly wants the constant, then these estimates have to be optimal. Even with the rather technical bounds of \cite{DO}, there were still multiple times when one could not precisely track the exponential factor associated with the high moments. While this is perfectly fine for proving that some law of iterated logarithm holds, it is  impossible to deduce anything about the value of the lower tail of the law of the iterated logarithm.

Inspired by the connection between the capacity and the self-intersection, one might try to see if there are any parallels one can draw from the proof of the large deviation principle for the self-intersection in 2-dimensions. Indeed, Bass, Chen, Kumagai and Rosen \cite{BCR,BK02} were able to establish an exact form for the constant associated with the large deviation principle for the self-intersection of random walks.

As observed in \cite{BCR,BK02}, a vital tool in both these analyses is a splitting formula. The self-intersection of a random walk can be written as the sum of two self-intersections of the first and second half of the walks and the mutual intersection of the first and second half. The large deviation behavior when $d=2$ is largely determined by this mutual intersection. For the capacity, one can perform a similar splitting with the quantity $\chi$ {like} in the work \cite{As3}. 

For two arbitrary sets $A$ and $B$, $\chi$ is defined as,
\begin{equation}\label{eq:defchi}
\begin{aligned}
   \chi(A,B):= & \sum_{y \in A} \sum_{z \in B} \mathbb{P}(R'_{y} \cap (A \cup B) =\emptyset) G_D(y-z) \mathbb{P}(R'_z \cap B =\emptyset) \\
   +& \sum_{y \in A} \sum_{z \in B} \mathbb{P}(R'_{y} \cap  A =\emptyset) G_D(y-z) \mathbb{P}(R'_z \cap (A \cup B) =\emptyset),
\end{aligned}
\end{equation}
where $R'_y$ is the range of an infinite random walk range after time $1$ starting at the point $y$ at time $0$.  
To show the result, we will substitute  two independent simple random walk ranges until time $n$, $A+\mathcal{S}^1$ and $B=\mathcal{S}^2$ (which are also independent of $R'_y$). 
The large deviation behavior should also be determined by this `mutual capacity', $\chi$.
However, after this step, if one tries to imitate the strategy of Bass, Chen, and Rosen \cite{BCR} to analyze $\chi$, fundamental difficulties arise at the very beginning that prevent one from proceeding forward. 

First of all, observe that each line of $\chi$, due to the probability term $\mathbb{P}(R'_y \cap (\mathcal{S}^1 \cup \mathcal{S}^2) =\emptyset)$, is asymmetric in $A$ and $B$. Furthermore, the same probability term couples the first and second parts of the random walk. In general, many formulas that one would like to apply to compute moments, such as the Feynman-Kac formula for lower bounds on the asymptotic moments, would first require one to separate the two halves of the random walk from each other. Usually, such a separation can be justified by applying the Cauchy-Schwartz inequality, and, as in the works of \cite{Chenbook} for the cross term occurring when studying the moderate deviations of the range of a random walk, one will not incur too much loss by performing this procedure.  This is no longer the case when one deals with an asymmetric cross-term like $\chi$. Indeed, the key first step in trying to determine the exact constant for the moderate deviations would be to try to identify a symmetric main term contribution for $\chi$.

The first guess that one might have would be to show that the terms $\mathbb{P}(R'_y \cap (\mathcal{S}^1 \cup \mathcal{S}^2)=\emptyset)$ could be replaced by the expected value $(1+o(1))\frac{\pi^2}{8 \log n}$. This replacement was performed in the papers \cite{As5,DO} in order to establish a CLT and a LIL, respectively. However, the moment estimates required to prove such results are insufficiently strong to demonstrate a large deviation principle or determine an exact constant. Indeed, the paper \cite{As1} remarked that it is possible that in the large deviation regime, it would be more effective for the random walk to reorganize itself into configurations such that $\mathbb{P}(R'_y \cap (S^1 \cup S^2)=\emptyset, 0 \not\in S^1)$ is far away from its expected value of $(1+o(1))\frac{\pi^2}{8 \log n}$.

Indeed, since we cannot replace these probability terms with their expectation, we have to determine the main and error terms via manipulations that preserve the structure of these probability terms. Indeed, our main term can be guessed to be of the form,
$$
\sum_{y \in \mathcal{S}^1} \sum_{z \in \mathcal{S}^2} \mathbb{P}(R'_{y} \cap \mathcal{S}^1=\emptyset) G_D(y-z) \mathbb{P}(R'_z \cap \mathcal{S}^2 =\emptyset).
$$
By decomposing $G_D = \tilde{G}_D * \tilde{G}_D$, the convolutional square root of $G_D$, we see that we indeed have a decomposition that could split the two sets $\mathcal{S}^1$ and $\mathcal{S}^2$ from each other. Namely, the quantity above can be written as,
$$
\sum_{a \in \mathbb{R}^4} \sum_{y \in \mathcal{S}^1} \mathbb{P}(R'_{y} \cap \mathcal{S}^1=\emptyset) \tilde{G}_D(y-a) \sum_{z \in \mathcal{S}^2} \mathbb{P}(R'_z \cap \mathcal{S}^2 =\emptyset) \tilde{G}_D(z-a). 
$$
This term will indeed be symmetric, and one has more tools for computing the exact value of the asymptotic moments. The full analysis of this term is given in section \ref{sec:TL}. 
This main term will lead to the corresponding error term,
\begin{equation*}
\sum_{x^1 \in \mathcal{S}^1} \sum_{x^2 \in \mathcal{S}^2} \mathbb{P}(R'_{x^1} \cap \mathcal{S}^1= \emptyset) G_D(x^1- x^2) \mathbb{P}(R'_{x^2} \cap \mathcal{S}^2= \emptyset, R'_{x^2}\cap \mathcal{S}^1 \ne \emptyset ).
\end{equation*}
The main observation is that this error term should approximately be of order $\frac{n}{(\log n)^3}$. This is one $\log n$ factor less than the expected order of the main term. One still needs to determine the value of high moments of this error term; however, one no longer needs to care about the exact values. Indeed, one only needs to derive an upper bound for the high moments of this error term. Section \ref{sec:introdecomp} will justify the splitting of $\chi$ into its main and error terms, while Section \ref{sec:ProofofThmchipr} will analyze the error term. The analysis of this error term involved multiple steps; the first step was to represent the cumbersome $\mathbb{P}(R'_{x^2} \cap \mathcal{S}^2= \emptyset, R'_{x^2}\cap \mathcal{S}^1 \ne \emptyset )$  into another term that is fit for moment computation. Afterward, we had to carefully exploit a version of monotonicity for the non-intersection probability $\mathbb{P}(R'_{x^1} \cap \mathcal{S}^1= \emptyset)$ that would allow us to justify the replacement of $\mathbb{P}(R'_{x^1} \cap \mathcal{S}^1= \emptyset)$ with its expectation.  When considering a law of iterated logarithms, we see that the size on this scale will be larger by a factor of $\log \log n$. Thus, the random walk has atypical behavior and one has to be very precise with the analysis and can no longer rely on heuristics coming from analyzing typical behavior.

\section{Proof of Theorem \ref{m1} and Corollary \ref{c1}} 
In this section, we show our main results, that is, Theorem \ref{m1} and Corollary \ref{c1}. 
In the proof, we write $f (n) \lesssim g(n)$ if there exists a (deterministic) constant $c>0$ such that $f (n) \le cg(n)$  for all $n$, and $f (n) \gtrsim g(n)$ if $g(n) \lesssim f (n)$. 
$\mathcal{S}[a,b]$ means the random walk range between time $a$ and $b$. 
Let $\mathbb{P}^x$ (resp. $\mathbb{E}^x$) be the probability of the simple random walk (or the Brownian motion) starting at $x$. 
We usually write $\mathbb{P}$ (resp. $\mathbb{E}$) for $\mathbb{P}^0$ (resp. $\mathbb{E}^0$). 

\subsection{Reduction to the study of mutual capacity} \label{sec:section2}

 In order to determine the exact moderate deviation asymptotic for 
 $\ca(\mathcal{S}[1,n]) - \mathbb{E}[\ca(\mathcal{S}[1,n])]$, it suffices to derive a moderate deviation for the term $\chi$. 
 {For two random walks $\mathcal{S}^1$ and $\mathcal{S}^2$, recall the cross-term in \eqref{eq:defchi}
\begin{equation*}
\begin{aligned}
\chi(\mathcal{S}^1,\mathcal{S}^2):= & \sum_{x^1 \in \mathcal{S}^1} \sum_{x^2 \in \mathcal{S}^2} \mathbb{P}(R'_{x^1} \cap \mathcal{S}^1= \emptyset) G_D(x^1 - x^2) \mathbb{P}(R'_{x^2} \cap ( \mathcal{S}^1 \cup \mathcal{S}^2) = \emptyset)
\\&+\sum_{x^1 \in \mathcal{S}^1} \sum_{x^2 \in \mathcal{S}^2}\mathbb{P}(R'_{x^1} \cap (\mathcal{S}^1 \cup \mathcal{S}^2)= \emptyset) G_D(x^1 - x^2) \mathbb{P}(R'_{x^2} \cap  \mathcal{S}^1= \emptyset).
\end{aligned}
\end{equation*}
}
Later, 
we assume that $\mathcal{S}^1,\mathcal{S}^2$ are independent random walks of duration $n$ and $\mathcal{S}$ is also a random walk of duration $n$, that is, $\mathcal{S}[1,n]$. 
\begin{thm}\label{thm:chi}
Consider $\chi=\chi(\mathcal{S}^1,\mathcal{S}^2)$ and let $b_n= O(\log \log n)$ with $\lim_{n \to \infty }b_n = \infty$. Then, for any $\lambda>0$,
 \begin{equation}
\lim_{n \to \infty} \frac{1}{b_n} \log \mathbb{P}\left(\chi \ge \lambda \frac{nb_n}{(\log n)^2} \right)=- I_4(\lambda). 
\end{equation}
\end{thm}

We will show it in Section \ref{sec:introdecomp}. 
We will give the proof of Theorem \ref{m1} assuming the above result.

\begin{proof}[Proof of Theorem \ref{m1}]
\textit{Splitting the Walk}

For simplicity in the presentation of the argument, we will perform computations when $n$ is a  multiple of a large power of $2$. 
For a complete formalization of the argument, one can consider a continuous time random walk rather than a discrete time random walk as in \cite[Chapter 6]{BCR} to derive large deviation estimates, but the essential difference in the proofs are minimal.  

First, fix a large integer $L$; we first subdivide our random walk $\mathcal{S}$  into $2^L$  parts over various iterations. Set $m_l = n/2^l$ and let $\mathcal{S}^{(k),m_l}$ denote $\mathcal{S}[(k-1) m_l, k m_l]$; namely, it is the $k$-th portion of the random walk once divided into $2^l$ equal parts. 
With this notation in hand, we can define the cross-term, 
\begin{equation*}
 \Lambda_l = \sum_{j=1}^{2^l-1} \chi(\mathcal{S}^{(2j-1),m_l}, \mathcal{S}^{(2j),m_l}).  
 \end{equation*}
We also have the following decomposition of $\ca(\mathcal{S})$, 
\begin{align*}
\ca (\mathcal{S})
= \sum_{i=1}^{2^L} \ca (\mathcal{S}^{(i),m_L}) 
-\sum_{l=1}^L \Lambda_l + \epsilon_L. 
\end{align*}
The error $\epsilon_L$ has the moment bound $\mathbb{E}[\epsilon_L^2] = O((\log n)^2)$ from \cite[Proposition 2.3]{As3}. 
It is actually better to deal with a slightly modified cross-term.
Consider two random walks, $\mathcal{S}^1,\mathcal{S}^2$ of the same length $n$. Define, as in equation \eqref{eq:deftlprn} which will appear in the sequel, the modified cross term:
\begin{equation*}
TL(\mathcal{S}^1,\mathcal{S}^2) = \sum_{x^1 \in \mathcal{S}^1} \sum_{x^2 \in \mathcal{S}^2} \mathbb{P}(R'_{x^1} \cap \mathcal{S}^1 = \emptyset) G_D(x^1- x^2) \mathbb{P}(R'_{x^2} \cap \mathcal{S}^2 = \emptyset).
\end{equation*}
The results of Theorem \ref{thm:chipr}  show that for any $\epsilon>0$, we have that, 
\begin{equation*}
\lim_{n \to \infty} \frac{1}{b_n} \log \mathbb{P}\left( |\chi(\mathcal{S}^{(2j-1),m_l}, \mathcal{S}^{(2j),m_l}) - 2 TL(\mathcal{S}^{(2j-1),m_l}, \mathcal{S}^{(2j),m_l})| \ge \epsilon \frac{n b_n}{(\log n)^2}\right) = -\infty.
\end{equation*}
Accordingly, it is natural to consider the modified term,
\begin{equation*}
\tilde{\Lambda}_l:= 2\sum_{j=1}^{2^l-1} TL(\mathcal{S}^{(2j-1),m_l}, \mathcal{S}^{(2j),m_l}).
\end{equation*}
Furthermore, the moment bound on $\epsilon_L$ combined with Markov's inequality shows that
\begin{equation*}
\frac{1}{b_n} \log \mathbb{P}\left(\epsilon_L \ge \epsilon  \frac{n}{(\log n)^2}\right) \lesssim \frac{- \log n +  \log \log n + \log \epsilon }{b_n}.
\end{equation*}
Thus,
\begin{equation*}
\lim_{n \to \infty} \frac{1}{b_n} \log \mathbb{P}\left(\epsilon_L \ge \epsilon  \frac{n}{(\log n)^2}\right) =-\infty.
\end{equation*}


Combining these facts, we see that if we fix $L$ and take $n \to \infty$, we have that
\begin{equation*}
\begin{aligned}
&\lim_{n \to \infty} \frac{1}{b_n} \log \mathbb{P}\left(-\ca(\mathcal{S})  + \mathbb{E}[\ca(\mathcal{S})]\ge \lambda\frac{b_n n}{(\log n)^2}\right) \\&= \lim_{n \to \infty} \frac{1}{b_n} \log \mathbb{P} \left(-\sum_{i=1}^{2^L}( \ca (\mathcal{S}^{(i),m_L}) - \mathbb{E}[ \ca (\mathcal{S}^{(i),m_L})]) 
+\sum_{l=1}^L (\tilde{\Lambda}_l - \mathbb{E}[\tilde{\Lambda}_l]) \ge \lambda \frac{b_n n}{(\log n)^2}\right).
\end{aligned}
\end{equation*}
Note that in the previous expression, we used the fact that $\mathbb{E}[\epsilon_L]$ and $\mathbb{E}[|\Lambda_l - \tilde{\Lambda}_l|]$, would not contribute to the expectations.

Our goal now is to show the following:

\begin{equation} \label{eq:Simplification}
\begin{aligned}
    &\lim_{L \to \infty} \lim_{n \to \infty} \frac{1}{b_n}
     \log \mathbb{P}\left(\sum_{i=1}^{2^L} (-\ca(\mathcal{S}^{(i),m_L}) + \mathbb{E}[\ca(\mathcal{S}^{(i),m_L})]) + \sum_{l=1}^L (\tilde{\Lambda}_l - \mathbb{E}[\tilde{\Lambda}_l])  \ge  \lambda \frac{b_n n}{(\log n)^2} \right) \\& = -I_4(\lambda). 
\end{aligned}
\end{equation}

We will start with showing the upper bound of \eqref{eq:Simplification}.

\textit{Upper Bound in \eqref{eq:Simplification}:}
It is manifest that $\mathbb{E}[\tilde{\Lambda}_l]$ is a positive number. Thus, if we only care about obtaining upper bounds on the probability found in equation \eqref{eq:Simplification}, we can drop the term $-\mathbb{E}[\tilde{\Lambda}_l]$ in the computation for the upper bound. We have,
\begin{align} \label{eq:sumbnd}
&\mathbb{P}\left( \sum_{i=1}^{2^L} (-\ca(\mathcal{S}^{(i),m_L}) + \mathbb{E}[\ca(\mathcal{S}^{(i),m_L})]) + \sum_{l=1}^L \tilde{\Lambda}_l \ge \lambda \frac{b_n n}{(\log n)^2} \right)\\
\notag
\le &\mathbb{P}\left( \sum_{i=1}^{2^L} (\mathbb{E}[\ca (\mathcal{S}^{(i),m_L})] -\ca (\mathcal{S}^{(i),m_L}))
\ge \epsilon \frac{\lambda n}{(\log n)^2}b_n\right)\\
\notag
+&\sum_{l=1}^L 
\mathbb{P}\left( \tilde{\Lambda}_l \ge (1-\epsilon) 2^{-l}\frac{\lambda n}{(\log n)^2}b_n\right).
\end{align}
By using Lemma  \ref{lem:aprioribound} and  \cite[Theorem 1.2.2]{Chenbook}, we can derive that 
\begin{equation} \label{eq:trivbnd}
\limsup_{n\to \infty} \frac{1}{b_n} \log \mathbb{P}\left( \sum_{i=1}^{2^L} (\mathbb{E}[\ca (\mathcal{S}^{(i),m_L})] -\ca (\mathcal{S}^{(i),m_L}))
\ge \epsilon \frac{\lambda n}{(\log n)^2}b_n \right) \le - 2^L C \lambda \epsilon.
\end{equation}

Now recall that $\tilde{\Lambda}_l$ is a sum of i.i.d. random variables. 
We can apply our Theorem \ref{thm:TL} along with \cite[Theorem 1.2.2]{Chenbook} to assert that 
\begin{equation} \label{eq:sumtildlamb}
   \limsup_{n\to \infty} \frac{1}{b_n} \log \mathbb{P}\left(\tilde{\Lambda}_l \ge (1-\epsilon) 2^{-l} \frac{\lambda n}{(\log n)^2} b_n \right) \le - I_4(\lambda- \epsilon).
\end{equation}
If we combine equations \eqref{eq:sumtildlamb} and \eqref{eq:trivbnd} in equation \eqref{eq:sumbnd}, 
we see that,
\begin{equation*}
\begin{aligned}
& \limsup_{n \to \infty} \frac{1}{b_n} \log \mathbb{P}\left( \sum_{i=1}^{2^L} (-\ca(\mathcal{S}^{(i),m_L}) + \mathbb{E}[\ca(\mathcal{S}^{(i),m_L})] )+ \sum_{l=1}^L \tilde{\Lambda}_l \ge \frac{\lambda  n b_n}{(\log n)^2}\right) \\ &\le - \min\left( 2^L C \lambda \epsilon,  I_4(\lambda- \epsilon)\right).
\end{aligned}
\end{equation*}
If we first take $L$ to $\infty$ and then $\epsilon \to 0$, 
we derive the desired upper bound on the probability.

\textit{Lower bound in \eqref{eq:Simplification}:}

First consider the quantity $SL_n$  as in equation \eqref{def:SLN} given by,
\begin{equation*}
SL_n= \sum_{x^1 \in \mathcal{S}} \sum_{x^2 \in \mathcal{S}} \mathbb{P}(R'_{x^1} \cap  \mathcal{S} = \emptyset) G_D(x^1-x^2) \mathbb{P}(R'_{x^2} \cap  \mathcal{S} = \emptyset).
\end{equation*}
Since
\begin{equation*}
\begin{aligned}
SL_n
\le & \sum_{i=1}^{2^L} \sum_{x^1,x^2 \in \mathcal{S}^{(i),m_L}} \mathbb{P}(R'_{x^1} \cap  \mathcal{S}^{(i),m_L} = \emptyset) G_D(x^1-x^2) \mathbb{P}(R'_{x^2} \cap  \mathcal{S}^{(i),m_L} = \emptyset)\\
+&  \sum_{\substack{x^1 \in \mathcal{S}^{(i),m_L}, x^2 \in \mathcal{S}^{(j),m_L}, \\ 1\le i \neq j \le 2^L}}\mathbb{P}(R'_{x^1} \cap  \mathcal{S} = \emptyset) G_D(x^1-x^2) \mathbb{P}(R'_{x^2} \cap  \mathcal{S} = \emptyset)
\end{aligned}
\end{equation*}
and the second term in the right hand side is bounded by
\begin{equation*}
\begin{aligned}
2  \sum_{l=1}^L \sum_{j=1}^{2^l-1} \sum_{\substack{x^1 \in \mathcal{S}^{(2j-1),m_l},\\ x^2 \in \mathcal{S}^{(2j),m_l}}}
\mathbb{P}(R'_{x^1} \cap  \mathcal{S} = \emptyset) G_D(x^1-x^2) \mathbb{P}(R'_{x^2} \cap  \mathcal{S} = \emptyset)
\le \sum_{l=1}^L \tilde{\Lambda}_l, 
\end{aligned}
\end{equation*}
we have that,
\begin{equation*}
\begin{aligned}
& \sum_{i=1}^{2^L} (-\ca(\mathcal{S}^{(i),m_L}) + \mathbb{E}[\ca(\mathcal{S}^{(i),m_L})]) + \sum_{l=1}^L (\tilde{\Lambda}_l - \mathbb{E}[\tilde{\Lambda}_l]) \\
&\ge SL_n-\mathbb{E}[SL_n]   +  \sum_{i=1}^{2^L} (-\ca(\mathcal{S}^{(i),m_L}) + \mathbb{E}[\ca(\mathcal{S}^{(i),m_L})]) - \sum_{l=1}^L \mathbb{E}[\tilde{\Lambda}_l] \\& - \sum_{i=1}^{2^L} \sum_{x^1,x^2 \in \mathcal{S}^{(i),m_L}} \mathbb{P}(R'_{x^1} \cap  \mathcal{S}^{(i),m_L} = \emptyset) G_D(x^1-x^2) \mathbb{P}(R'_{x^2} \cap  \mathcal{S}^{(i),m_L} = \emptyset)+\mathbb{E}[SL_n].
\end{aligned}
\end{equation*}

Noting that $\sum_{l=1}^L \mathbb{E}[\tilde{\Lambda}_l] = O\left( \frac{n }{(\log n)^2} \right)$, the term $\sum_{l=1}^{L} \mathbb{E}[\tilde{\Lambda}_l]$ will not contribute to the large deviation statistics to the order we are concerned with. 
In addition, 
\begin{equation*}
\begin{aligned}
&  \mathbb{E}[\sum_{i=1}^{2^L} \sum_{x^1,x^2 \in \mathcal{S}^{(i),m_L}} \mathbb{P}(R'_{x^1} \cap  \mathcal{S}^{(i),m_L} = \emptyset) G_D(x^1-x^2) \mathbb{P}(R'_{x^2} \cap  \mathcal{S}^{(i),m_L} = \emptyset)] -\mathbb{E}[SL_n] \\
\le & \mathbb{E}[\sum_{i=1}^{2^L} \sum_{x^1,x^2 \in \mathcal{S}^{(i),m_L}} \mathbb{P}(R'_{x^1} \cap  \mathcal{S}^{(i),m_L} = \emptyset) G_D(x^1-x^2) \mathbb{P}(R'_{x^2} \cap  \mathcal{S}^{(i),m_L} = \emptyset)]\\
-& \mathbb{E}[\sum_{i=1}^{2^L} \sum_{x^1,x^2 \in \mathcal{S}^{(i),m_L}} \mathbb{P}(R'_{x^1} \cap  \mathcal{S} = \emptyset) G_D(x^1-x^2) \mathbb{P}(R'_{x^2} \cap  \mathcal{S} = \emptyset)]  + \frac{C n }{(\log n)^2}\\
\le & 2 \mathbb{E}[\sum_{i=1}^{2^L} \sum_{x^1,x^2 \in \mathcal{S}^{(i),m_L}}
\mathbb{P}(R'_{x^1} \cap  \mathcal{S}^{(i),m_L} = \emptyset, R'_{x^1} \cap  \mathcal{S} \neq \emptyset) G_D(x^1-x^2) \mathbb{P}(R'_{x^2} \cap  \mathcal{S}^{(i),m_L} = \emptyset)]\\
+&  \frac{C n }{(\log n)^2}
\lesssim    \frac{n }{(\log n)^2}.
\end{aligned}
\end{equation*}
The final inequality is very similar to the type of
error terms we have dealt with in Section \ref{sec:ProofofThmchipr}. 
Thus, we omit the proof. 
Thus, we have that
\begin{equation}\label{gggy}
\begin{aligned}
 & \mathbb{P}\left( \sum_{i=1}^{2^L} (-\ca(\mathcal{S}^{(i),m_L}) + \mathbb{E}[\ca(\mathcal{S}^{(i),m_L})]) + \sum_{l=1}^L \tilde{\Lambda}_l \ge \frac{\lambda  n b_n}{(\log n)^2} \right) \\ 
 &\ge \mathbb{P}\left(SL_n-\mathbb{E}[SL_n] \ge \frac{(\lambda +\epsilon)  n b_n}{(\log n)^2} \right)- \mathbb{P}\left(\sum_{i=1}^{2^L} (\ca(\mathcal{S}^{(i),m_L}) - \mathbb{E}[\ca(\mathcal{S}^{(i),m_L})])
  \ge \frac{\epsilon n  b_n }{2 (\log n)^2}\right) \\
  &- \mathbb{P} \bigg( \sum_{i=1}^{2^L} \sum_{x^1,x^2 \in \mathcal{S}^{(i),m_L}}
  \mathbb{P}(R'_{x^1} \cap  \mathcal{S}^{(i),m_L} = \emptyset) G_D(x^1-x^2) \mathbb{P}(R'_{x^2} \cap  \mathcal{S}^{(i),m_L} = \emptyset) \\
& -\mathbb{E}[\sum_{i=1}^{2^L} \sum_{x^1,x^2 \in \mathcal{S}^{(i),m_L}}
  \mathbb{P}(R'_{x^1} \cap  \mathcal{S}^{(i),m_L} = \emptyset) G_D(x^1-x^2) \mathbb{P}(R'_{x^2} \cap  \mathcal{S}^{(i),m_L} = \emptyset)] 
  \ge \frac{\epsilon n b_n}{2(\log n)^2} \bigg). 
\end{aligned}
\end{equation}

Now, we note that the negative quantities on  the right hand side are the sum of i.i.d random variables; the term on the last line are also of the form $SL_{n 2^{-L}}$. By using Lemma \ref{lem:aprioribound} and the result for $SL_n$ from Corollary \ref{cor:SLN} as well as \cite[Theorem 1.2.2]{Chenbook}, we have that the probabilities in the last two lines are bounded by $\exp[b_n(-2^L C \epsilon)]$ for some constant $C$.

Furthermore, Corollary \ref{cor:SLN} also gives us that $\lim_{n\to \infty}\frac{1}{b_n} \log P(SL_n-\mathbb{E}[SL_n] \ge \frac{(\lambda +\epsilon) n b_n}{(\log n)^2}) = -I_4(\lambda+\epsilon)$.
Given $\epsilon$, if we first choose $L$ such that $-2^L C \epsilon \ll - I_4(\lambda+\epsilon)$, we see that,
\begin{equation*}
\begin{aligned}
&\liminf_{n \to \infty} \frac{1}{b_n} \log \mathbb{P}\left( \sum_{i=1}^{2^L} (-\ca(\mathcal{S}^{(i),m_l}) + \mathbb{E}[\ca(\mathcal{S}^{(i),m_l})]) + \sum_{l=1}^L \tilde{\Lambda}_l \ge \frac{\lambda  n b_n}{(\log n)^2}\right) \\
\ge & - I_4(\lambda+\epsilon).
\end{aligned}
\end{equation*}
We can then take the limit as $L$ to $\infty$ and then $\epsilon \to 0$ to show equation \eqref{eq:Simplification}. This completes the proof of the result.

\end{proof}

 We can quickly derive our corollary for the exact constant of the LIL for the lower tail of $\ca(\mathcal{S}) - \mathbb{E}[\ca(\mathcal{S})] $.
\begin{proof}[Proof of  Corollary \ref{c1} ]
This will follow by carefully applying the Borel-Cantelli lemma. 
The large deviation estimates of Theorem \ref{m1} are used to derive the appropriate convergence or divergence conditions. 
We should take $\lambda b_n=  (1+\epsilon) I_4(1)^{-1}\log n$ and then choose the sequence $a_n=e^n$ in Theorem \ref{m1}. While we obtain the lower bound  by the first Borel-Cantelli lemma taking $\epsilon<0$, the upper bound by second taking $\epsilon>0$. 
The details are the same as those found in \cite[Theorem 8.6.2]{Chenbook}. 
\end{proof}

\subsection{A priori Estimates on $\ca(\mathcal{S})$}

In this section, we will prove the following large deviation principle on $\ca(\mathcal{S})$. 
The following lemma will give a sufficient a-priori large deviation estimate to bound the second term of the second line of \eqref{gggy}. 
\begin{lem} \label{lem:aprioribound}
Let $b_n = \text{O}(\log \log n)$ {with $\lim_{n \to \infty }b_n = \infty$}. There exists some constant $C$ such that for any $\lambda>0$,
\begin{equation} \label{eq:aprior}
\limsup_{n \to \infty} \frac{1}{b_n} \log  
\mathbb{P}\left( | \ca(\mathcal{S}) - \mathbb{E}[\ca(\mathcal{S})] | \ge \frac{\lambda n}{(\log n)^2} b_n \right) \le -C \lambda.
\end{equation}
\end{lem}

\begin{proof}
We will consider proving this when $n$ is a power of $2$. By changing the constant $C$ that appears on the right hand side of \eqref{eq:aprior}, one can use our subdivision formula of $R_n$ in order to obtain estimates on general $n$ via a binary decomposition in terms of powers of $n$.

Now, assume $n$ is power of $2$ and let $L = 4 \log( \log n) $. We can decompose $r_n$ iteratively $L$ times to notice that,
\begin{equation*}
\ca(\mathcal{S}) 
= \sum_{i=1}^{2^L} \ca (\mathcal{S}^{(i),m_L}) 
-\sum_{l=1}^L \Lambda_l + \epsilon_L,
\end{equation*}
where we use the notation from the proof of Theorem \ref{m1}. This time $\epsilon_L$ can be shown to be of $O((\log n)^{10})$. (There will at most $1+2 +4 + \ldots + 2^L = O((\log n)^4)$ many error  terms of the form $\epsilon$ in the decomposition. Each of these error terms has moment $O((\log n)^2)$.) 
By applying Chebyshev's inequality, we see that the error term $\epsilon_L$ provides no change to the probability at the scale $b_n$. 
Thus, we freely drop this error term $\epsilon_L$ in what follows.

\textit{Bounding Upper tails of $\ca(\mathcal{S}) - \mathbb{E}[\ca(\mathcal{S})] $}

If one wants to bound the probability $\mathbb{P}( \ca(\mathcal{S}) - \mathbb{E}[\ca(\mathcal{S})]  \ge \frac{\lambda n b_n}{(\log n)^2})$ from above, then since all the terms $\Lambda_l$ are positive and $\sum_{i=1}^L \mathbb{E}[\Lambda_l] = O(\frac{n}{(\log n)^2})$, it suffices to bound the probability,
\begin{equation*}
\mathbb{P}\left(\sum_{i=1}^{2^L} \ca(\mathcal{S}^{(i),m_L}) - \mathbb{E}[\ca(\mathcal{S}^{(i),m_L})] \ge \frac{\lambda n b_n}{(\log n)^2} \right).
\end{equation*}
Now, the sequence $\ca(\mathcal{S}^{(i),m_L}) - \mathbb{E}[\ca(\mathcal{S}^{(i),m_L})]$ are the sequence of i.i.d. random variables with the property that $\mathbb{
E}\exp\left[ \frac{\theta}{n}|\ca(\mathcal{S}) - \mathbb{E}[\ca(\mathcal{S})] |\right] <\infty$. (This is due the the fact that $\ca(\mathcal{S}) \le n$.)  
We can apply \cite[Lemma 4.4]{BCR} to assert that there is some constant $\theta> 0$ such that
$$
\limsup_{n \to \infty}\mathbb{E}\left[ \exp\left[ \frac{\theta}{2^{L/2}} \left|\sum_{i=1}^{2^L}  \frac{2^L}{n} (\ca(\mathcal{S}^{(i),m_L}) -\mathbb{E}[\ca(\mathcal{S}^{(i),m_L})] ) \right| \right]\right] <\infty.
$$
Since $2^{L/2} \ge (\log n)^2$ by choice, this implies that,
$$
\limsup_{n \to \infty}\mathbb{E}\left[ \exp\left[\theta \frac{(\log n)^2}{n} \left|\sum_{i=1}^{2^L}  (\ca(\mathcal{S}^{(i),m_L}) -\mathbb{E}[\ca(\mathcal{S}^{(i),m_L})] ) \right| \right]\right] <\infty.
$$
By Chebyshev's inequality, this shows that there is some constant $C$ such that,
\begin{equation} \label{eq:absval}
\limsup_{n \to \infty}\frac{1}{b_n}
\log 
\mathbb{P} \bigg( \left|\sum_{i=1}^{2^L}  (\ca(\mathcal{S}^{(i),m_L}) -\mathbb{E}[\ca(\mathcal{S}^{(i),m_L})] ) \right| \ge \frac{ \lambda n b_n}{(\log n)^2} \bigg) \le -C \lambda.
\end{equation}

\textit{Upper Bounds on the lower tail of $\ca(\mathcal{S}) - \mathbb{E}[\ca(\mathcal{S})] $}

Due to our control on $\left|\sum_{i=1}^{2^L}  (\ca(\mathcal{S}^{(i),m_L}) -\mathbb{E}[\ca(\mathcal{S}^{(i),m_L})] ) \right|$ from equation \eqref{eq:absval}. It suffices to bound $\mathbb{P}\left(\sum_{l =1}^L \Lambda_l \ge \frac{\lambda n b_n}{(\log n)^2} \right)$. 

If we define,
\begin{equation*}
\begin{aligned}
 &\alpha_{l,j}:=\\
 &\sum_{a= (2j-2) m_l+1}^{(2j-1) m_l} \sum_{b = (2j-1) m_l +1}^{(2j) m_l} \mathbb{P}(R'_{\mathcal{S}_a} \cap \mathcal{S}^{(2j-1),m_l} = \emptyset) G_D(\mathcal{S}_a - \mathcal{S}_b) \mathbb{P}(R'_{\mathcal{S}_b} \cap \mathcal{S}^{(2j),m_l} = \emptyset),
\end{aligned}
\end{equation*}
notice that we can bound,
\begin{equation*}
\chi(\mathcal{S}^{(2j-1),m_l}, \mathcal{S}^{(2j),m_l}) \le \alpha_{l,j}.
\end{equation*}

Finally, in order to show {$\limsup_{n \to \infty}\frac{1}{b_n} \log \mathbb{P}(\sum_{l=1}^L \Lambda_l \ge \frac{\lambda {n} b_n}{(\log n)^2}) \le - C \lambda$} for some constant $C$. It suffices to prove that the exponential moment,
\begin{equation} \label{eq:desiredweakmom}
\limsup_{n \to \infty}\mathbb{E}\left[ \exp\left[\frac{\theta (\log n)^2}{n} \sum_{l=1}^L \sum_{k=1}^{2^{l-1}} \alpha_{l,k} \right] \right] < \infty.
\end{equation}

By the consequence of Lemma \ref{lem:weaktostrong}, there is a parameter $\theta>0$ such that each $\alpha_{l,k}$ has the exponential moment,
\begin{equation*}
\limsup_{n \to \infty}\mathbb{E}\left[\exp\left[\theta \frac{ (\log (2^{-l} n))^2}{2^{-l}n} \alpha_{l,k} \right]\right]< \infty. 
\end{equation*}
Thus, we can follow the argument of  \cite[Theorem 5.4]{BCR} from equation (5.30) onwards to prove that the desired result \eqref{eq:desiredweakmom}. 
This completes the proof of the lemma. 
\end{proof}

\section{ Theorem \ref{thm:chi}: Large Deviations of the Cross Term} \label{sec:introdecomp}

In this section, we provide a decomposition for $\chi$ that will give us a proof of Theorem \ref{thm:chi}. 
Analyzing $\chi$ is not directly tractable due to the lack of symmetry in each individual product. 
Recall that $\mathcal{S}^1,\mathcal{S}^2$ are independent random walks of duration $n$. 
To deal with this issue, we can write this in terms of the following difference,
\begin{equation} \label{eq:decompchi}
\begin{aligned}
\chi &=\chi(\mathcal{S}^1,\mathcal{S}^2)\\
&=  2\sum_{x^1 \in \mathcal{S}^1} \sum_{x^2 \in \mathcal{S}^2} \mathbb{P}(R'_{x^1} \cap \mathcal{S}^1= \emptyset) G_D(x^1 - x^2) \mathbb{P}(R'_{x^2} \cap  \mathcal{S}^2 = \emptyset)\\
&-  \sum_{x^1 \in \mathcal{S}^1} \sum_{x^2 \in \mathcal{S}^2} \mathbb{P}(R'_{x^1}\cap \mathcal{S}^1= \emptyset) G_D(x^1- x^2) \mathbb{P}(R'_{x^2} \cap \mathcal{S}^2= \emptyset, R'_{x^2}\cap \mathcal{S}^1 \ne \emptyset )\\&-  \sum_{x^1 \in \mathcal{S}^1} \sum_{x^2 \in \mathcal{S}^2} \mathbb{P}(R'_{x^1}\cap \mathcal{S}^1= \emptyset, R'
_{x^1} \cap \mathcal{S}^2 \ne \emptyset) G_D(x^1- x^2) \mathbb{P}(R'_{x^2} \cap \mathcal{S}^2= \emptyset).
\end{aligned}
\end{equation}

Now, in order to obtain asymptotics on $\chi$, our goal is two-fold. 
\begin{enumerate}
\item  An upper bound on $\chi$ is found by merely considering the top line
\begin{equation}\label{eq:defTLn}
TL_n:=\sum_{x^1 \in \mathcal{S}^1} \sum_{x^2 \in \mathcal{S}^2} \mathbb{P}(R'_{x^1} \cap \mathcal{S}^1= \emptyset) G_D(x^1 - x^2) \mathbb{P}(R'_{x^2} \cap  \mathcal{S}^2 = \emptyset).
\end{equation}

Thus, one can obtain upper bounds on the moderate deviation statistics of $\chi$ merely from analyzing the moderate deviation statistics of $TL_n$.

\item Obtaining lower bounds on the moderate deviation statistics of $\chi$ needs more steps. First, one needs to show that the second line of \eqref{eq:decompchi}, which we denote as $\chi'$ is sub-leading relative to the first line. (The analysis of the third line would be similar to that of the second.)
\begin{equation} \label{eq:defchipr}
\chi':=\sum_{x^1 \in \mathcal{S}^1} \sum_{x^2 \in \mathcal{S}^2} \mathbb{P}(R'_{x^1} \cap \mathcal{S}^1= \emptyset) G_D(x^1- x^2) \mathbb{P}(R'_{x^2} \cap \mathcal{S}^2= \emptyset, R'_{x^2}\cap \mathcal{S}^1 \ne \emptyset ).
\end{equation}
Once this is established, lower bounds on the large deviation statistics of $\chi$ will be the same as those of $TL_n$. Furthermore, observe that only an upper bound on $\chi'$ is necessary.
\end{enumerate}

We will have two intermediate goals,
\begin{thm} \label{thm:TL}
Recall $TL_n$ as in equation \eqref{eq:defTLn}.  Fix $b_n= O(\log \log n)$ with $\lim_{n \to \infty} b_n = \infty$. We have that, for any $\lambda>0$, 
\begin{equation*}
\lim_{n \to \infty}\frac{1}{b_n} \log  \mathbb{P}\left( TL_n \ge \lambda \frac{  b_n n}{(\log n)^2} \right) = - 2 I_4(\lambda). 
\end{equation*}
\end{thm}

We remark that by following the same proof, we could obtain the following statement; this is analogous to our statement on $TL_n$ and \cite[Theorem 8.2.1]{Chenbook}, but uses the same random walk rather than two independent copies. 
\begin{cor} \label{cor:SLN}
Let $b_n = O(\log \log n)$ with $\lim_{n \to \infty} b_n = \infty$. 
Define $SL_n$ as,
\begin{equation} \label{def:SLN}
SL_n:= \sum_{x^1 \in \mathcal{S}} \sum_{x^2 \in \mathcal{S}} \mathbb{P}(R'_{x^1} \cap \mathcal{S}= \emptyset) G_D(x^1 - x^2) \mathbb{P}(R'_{x^2} \cap  \mathcal{S} = \emptyset).
\end{equation}
Then, we have that,
\begin{equation*}
\lim_{n \to \infty} \frac{1}{b_n} \log \mathbb{P} \left(  SL_n -\mathbb{E}[SL_n]  \ge \lambda \frac{b_n n}{(\log n)^2}\right) 
=- I_4(\lambda). 
\end{equation*}
\end{cor}

\begin{thm}\label{thm:chipr}
Recall $\chi'$ as in \eqref{eq:defchipr}. Fix $b_n=O(\log \log n)$ with $\lim_{n \to \infty} b_n = \infty$.  For any $\epsilon>0$, we have that,
\begin{equation*}
\lim_{n \to \infty} \frac{1}{b_n} \log\mathbb{P}\left( \chi' \ge \epsilon \frac{b_n n}{(\log n)^2} \right) =-\infty.
\end{equation*}
\end{thm}

It is clear that Theorem \ref{thm:chi} is a consequence of Theorems \ref{thm:TL} and \ref{thm:chipr} along with the decomposition \eqref{eq:decompchi}. The next few sections will be devoted to proving these theorems.

\section{Controlling The Third Order Intersections: Proof of Theorem \ref{thm:chipr}} \label{sec:ProofofThmchipr} 

For any $A$, let
$$
G_A(a,b)= \sum_{m=0}^{\infty}\mathbb{P}^a(\mathcal{S}_m =b, \mathcal{S} (0,m) \cap A = \emptyset).
$$
Here, $G_A(a,b)$ is a restricted Green's function. 
Via a path decomposition, one can see that we have the following expression for the more complicated probability term in $\chi'$,
\begin{equation} \label{eq:probdecomp}
\begin{aligned}
\mathbb{P}(R'_{x^2} \cap \mathcal{S}^2= \emptyset, R'_{x^2}\cap \mathcal{S}^1 \ne \emptyset )
= \sum_{x^1 \in \mathcal{S}^1}   G_{\mathcal{S}^2}(x^2 ,  x^1)  \mathbb{P}(R'_{x^1} \cap (\mathcal{S}^1 \cup \mathcal{S}^2) =\emptyset).
\end{aligned}
\end{equation}
It computes the total sum of all probabilities of random walk paths from $a$ to $b$ that do not intersect $A$. 

The equality in \eqref{eq:probdecomp} comes from a path decomposition. 
Namely, since $R'_{x^2} \cap \mathcal{S}^1 \ne \emptyset$, then the random walk $R'_{x^2}$ must intersect $\mathcal{S}^1$ at some final point $x^1$. After this point, the random walk starting from $x^1$ must not intersect either one of $\mathcal{S}^1$ or $\mathcal{S}^2$.
Then, we can further bound 
\begin{equation*}
\chi' \le \sum_{x^1_1,x^1_2 \in \mathcal{S}^1} \sum_{x^2 \in \mathcal{S}^2} \mathbb{P}(R'_{x^1_1} \cap \mathcal{S}^1= \emptyset) G_D(x^1_1 - x^2) G_{\mathcal{S}^2}(x^2,x^1_2)\mathbb{P}(R'_{x^1_2} \cap \mathcal{S}^1 = \emptyset).
\end{equation*}



Though we have simplified the probability term in question, we are still not ready to analyze this due to the appearance of the term $G_{\mathcal{S}^2}$. We have to introduce a more sophisticated analysis in order to deal with this term. 
First we fix parameter $\beta <1$, the specific value will be chosen later in accordance with what is appropriate for later upper bounds. 
It is important to decompose the walk $\mathcal{S}^2$ into appropriate intervals of size $n^{\beta}$.  We define,
\begin{equation*}
\mathcal{S}^2_{\beta,j} :=  \mathcal{S}^2[(j-1) n^{\beta}, j n^\beta]. 
\end{equation*}

By adding back points when necessary, and also using the fact that if $A \subset B$ then $G_{B}(x,y) \le G_{A}(x,y)$ for all points $x$ and $y$, we see that we have,
\begin{equation*}
\begin{aligned}
\chi' \le \chi'_{\beta}&:= \sum_{x^1_1, x^1_2 \in \mathcal{S}^1} \sum_{j=1}^{ n^{1-\beta} } \sum_{x^2 \in \mathcal{S}^2_{\beta,j}} \mathbb{P}(R'_{x^1_1} \cap \mathcal{S}^1 = \emptyset) G_D(x^1_1 - x^2) \\
& \hspace{2cm} \times G_{\mathcal{S}^2_{\beta,j}}(x^2, x^1_2) \mathbb{P}(R'_{x^1_2} \cap \mathcal{S}^1 = \emptyset).
\end{aligned}
\end{equation*}
(Note that $\chi'$ or $\chi'_{\beta}$ is a random variable determined by $\mathcal{S}^1$ and $\mathcal{S}^2$.)



One way to express the modified Green's function  $G_{\mathcal{S}^2_{\beta,j}}$ is as follows. Note that Green's function satisfies the following system of equations: 
for any $x \in \mathcal{S}^2_{\beta,j}$ and $m\ge 0$,
\begin{equation*}
\mathbb{P}^x(\mathcal{S}_m=x^1_2) 
=\sum_{i=0}^m \sum_{\tilde{x} \in \mathcal{S}^2_{\beta,i}} \mathbb{P}^x(\mathcal{S}_i=\tilde{x}) 
\mathbb{P}^{\tilde{x}}(\mathcal{S}(0,m-i) \cap \mathcal{S}^2_{\beta,j}=\emptyset ,\mathcal{S}_{m-i}=x^1_2)
\end{equation*}
and hence
\begin{equation*}
G_D(x- x^1_2) =\sum_{\tilde{x} \in \mathcal{S}^2_{\beta,i}} G_D(x- \tilde{x}) G_{\mathcal{S}^2_{\beta,j}}(\tilde{x}, x^1_2).
\end{equation*}


If we define the matrix with size $|\mathcal{S}^2_{\beta,j}|$
\begin{equation} \label{eq:defcalG}
[\mathcal{G}^{\mathcal{S}^2_{\beta,j}}]_{a,b} = G_D(a -b) ,\quad \text{ for } a,b \in \mathcal{S}^2_{\beta,j},
\end{equation}
we see that,
\begin{equation}  \label{eq:inveq}
\begin{bmatrix}
G_{\mathcal{S}^2_{\beta,j}}(a_1, x^1_2)\\
G_{\mathcal{S}^2_{\beta,j}}(a_2, x^1_2)\\
\vdots\\
G_{\mathcal{S}^2_{\beta,j}}(a_{|\mathcal{S}^2_{\beta,j}|}, x^1_2)
\end{bmatrix} = (\mathcal{G}^{\mathcal{S}^2_{\beta,j}})^{-1} \begin{bmatrix}
G_D(a_1, x^1_2)\\
G_D(a_2, x^1_2)\\
\vdots\\
G_D(a_{|\mathcal{S}^2_{\beta,j}|},x^1_2)
\end{bmatrix}, 
\end{equation}
where $a$ varies over all the points in $\mathcal{S}^2_{\beta,j}$.



The analysis of the matrix inverse will depend on the distance between the points in $x^1_2$  and the set $\mathcal{S}^2_{\beta,j}$. Observe that if $x^1_2$ were far away from the set $\mathcal{S}^2_{\beta,j}$, then the terms in the vector on the right hand side of equation \eqref{eq:inveq} would approximately be constant. Furthermore, it is also rather unlikely that the $x^1_2$ would be close to the set $\mathcal{S}^2_{\beta,j}$. Following this intuition, we divide the the points $x^1_2$ into two categories,
\begin{enumerate}
\item In category 1, the point $x^1_2$ is of distance at least $ \sqrt{n}^{1-\delta}$ away from all the points in $\mathcal{S}^2_{\beta,j}$.

\item In category 2, the point $x^1_2$ is not of distance at least  $ \sqrt{n}^{1-\delta}$ away from some point in $\mathcal{S}^2_{\beta,j}$

\end{enumerate} 

If we are in the first category, we have a superior analysis, as will be illustrated by the following manipulations. 
Indeed, assume that the point $x^1_2$ is of distance at least $\sqrt{n}^{1-\delta}$ from all points in $\mathcal{S}^2_{\beta,j}$. Now, let $a$ and $b$ be two points in $\mathcal{S}^2_{\beta,j}$. Then, we must have 
\begin{equation} \label{eq:diffnbd}
\begin{aligned}
&\left|G_D(a -x^1_2) - G_D(b - x^1_2)\right| 
\lesssim \left| \frac{1}{\|a-x^1_2\|^2}  - \frac{1}{\|b-x^1_2\|^2}\right| + \frac{1}{\|a-x^1_2\|^4} + \frac{1}{\|
b-x^1_2\|^4}\\
& \lesssim  \frac{\| a - b\|}{\|a - x^1_2\|^{3} } + 2 \left(\frac{1}{\sqrt{n}^{1-\delta}} \right)^4 \lesssim \frac{n^{\beta}}{ \left(\sqrt{n}\right)^{3- 3\delta}}.
\end{aligned}
\end{equation}
Note that here, we have applied the estimates of \cite[Theorem 4.3.1]{LL10}. 
Then we use the fact that $\|a - b\|$ is less than $n^{\beta}$ when both are in the neighborhood of $\mathcal{S}^2_{\beta,j}$ and our assumption that $\|x^1_2 - a\| \ge \left(\sqrt{ n}\right)^{1-\delta}$. 
Let
\begin{equation*}
\begin{aligned}
\begin{bmatrix}
E_{a_1,x^1_2}
\\
\vdots\\
E_{a_{|\mathcal{S}^2_{\beta,j}|},x^1_2}
\end{bmatrix}
:=(\mathcal{G}^{\mathcal{S}^2_{\beta,j}})^{-1} 
\begin{bmatrix}
G_D(a_1- x^1_2) - G_D(a_1 - x^1_2)  \\
\vdots\\
G_D(a_{|\mathcal{S}^2_{\beta,j}|}- x^1_2) - G_D(a_1 - x^1_2)
\end{bmatrix}
.
\end{aligned}
\end{equation*}
In this case, we can further write the inverse formula as in equation \eqref{eq:inveq} as,
\begin{equation}  \label{eq:decompose}
\begin{aligned}
&\begin{bmatrix}
G_{\mathcal{S}^2_{\beta,j}}(a_1, x^1_2)\\
\vdots\\
G_{\mathcal{S}^2_{\beta,j}}(a_{|\mathcal{S}^2_{\beta,j}|}, x^1_2)
\end{bmatrix} = (\mathcal{G}^{\mathcal{S}^2_{\beta,j}})^{-1} \begin{bmatrix}
G_D(a_1- x^1_2)\\
\vdots\\
G_D(a_{|\mathcal{S}^2_{\beta, \alpha,j}|}- x^1_2)
\end{bmatrix}\\&  =  (\mathcal{G}^{\mathcal{S}^2_{\beta,j}})^{-1} \begin{bmatrix}1 \\ \vdots \\ 1 \end{bmatrix} \times G_D(a_1 - x^1_2) +  (\mathcal{G}^{\mathcal{S}^2_{\alpha,\beta,j}})^{-1}  \begin{bmatrix}
G_D(a_1-  x^1_2) - G_D(a_1 -x^1_2)  \\
\vdots\\
G_D(a_{|\mathcal{S}^2_{\beta,j}|}-  x^1_2) - G_D(a_1 - x^1_2)
\end{bmatrix}
\\& = \begin{bmatrix} 
\mathbb{P}(R'_{a_1} \cap \mathcal{S}^2_{\beta,j} = \emptyset) \\
\vdots\\
\mathbb{P}(R'_{a_{|\mathcal{S}^2_{\beta,j}|}} \cap \mathcal{S}^2_{\beta,j} =\emptyset)
\end{bmatrix} \times G_D(a_1 - x^1_2) + 
\begin{bmatrix}
E_{a_1,x^1_2}
\\
\vdots\\
E_{a_{|\mathcal{S}^2_{\beta,j}|},x^1_2}
\end{bmatrix}
.
\end{aligned}
\end{equation}
Thus, we have the following representation of $\chi'_{\beta}$.  In what follows, we let $\mathcal{I}(y,j)$ be the indicator function of the event  
\begin{equation*}
\mathcal{I}(y,j):=   \mathbbm{1}[\text{dist}(y, \mathcal{S}^2_{\beta,j}) \ge \sqrt{n}^{1-\delta}].
\end{equation*}

For each set $\mathcal{S}^2_{\beta,j}$ choose a point $\tilde{x}^2_j$. This point will serve as the central point used in the decomposition used in equation \eqref{eq:decompose}: 
\begin{equation}\label{eq:splitchiba}
\begin{aligned}
\chi'_{\beta}
&= \sum_{x^1_1, x^1_2 \in \mathcal{S}^1} \sum_{j=1}^{n^{1-\beta}} \sum_{x^2 \in \mathcal{S}^2_{\beta,j}} \mathcal{I}(x^1_2,j) \mathcal{I}(x^1_1,j)\mathbb{P}(R'_{x^1_1} \cap \mathcal{S}^1 = \emptyset) G_D(x^1_1 - x^2) \\& \hspace{2cm} \times \mathbb{P}(R'_{x^2} \cap \mathcal{S}^2_{\beta,j} = \emptyset) G_D(\tilde{x}^2_j - x^1_2) \mathbb{P}(R'_{x^1_2} \cap \mathcal{S}^1 = \emptyset) 
\\&+ \sum_{x^1_1, x^1_2 \in \mathcal{S}^1} \sum_{j=1}^{n^{1-\beta}} \sum_{x^2 \in \mathcal{S}^2_{\beta,j}} \mathcal{I}(x^1_2,j)  \mathcal{I}(x^1_1,j)\mathbb{P}(R'_{x^1_1} \cap \mathcal{S}^1 = \emptyset) G_D(x^1_1 -  \tilde{x}^2_j) \\ & \hspace{2cm} \times E_{x^2,x^1_2} \mathbb{P}(R'_{x^1_2} \cap\mathcal{S}^1 = \emptyset)\\&+
 \sum_{x^1_1, x^1_2 \in \mathcal{S}^1} \sum_{j=1}^{ n^{1-\beta} } \sum_{x^2 \in \mathcal{S}^2_{\beta,j}} \mathcal{I}(x^1_2,j)  \mathcal{I}(x^1_1,j)\mathbb{P}(R'_{x^1_1} \cap \mathcal{S}^1 = \emptyset)
 \\ & \hspace{2cm} \times  [ G_D(x^1_1 -x^2) - G_D(x^1_1 - \tilde{x}^2_j)] E_{x^2,x^1_2} \mathbb{P}(R'_{x^1_2} \cap \mathcal{S}^1 = \emptyset)\\
&+  \sum_{x^1_1, x^1_2 \in \mathcal{S}^1} \sum_{j=1}^{ n^{1-\beta} } \sum_{x^2 \in \mathcal{S}^2_{\beta,j}} [1 -  \mathcal{I}(x^1_2,j) \mathcal{I}(x^1_1,j)]  \mathbb{P}(R'_{x^1_1} \cap \mathcal{S}^1 = \emptyset) G_D(x^1_1 - x^2) \\&  \hspace{2cm} \times G_{\mathcal{S}^2_{\beta,j}}(x^2, x^1_2) \mathbb{P}(R'_{x^1_2} \cap \mathcal{S}^1 = \emptyset)\\
=:& MT_n+\mathcal{E}_1+\mathcal{E}_2+\mathcal{E}_3.
\end{aligned}
\end{equation}
The analysis of $\chi'_{\beta}$ now devolves into the following three lemmas.

\begin{lem} \label{lem:MTn}
Fix $m \in \N$. 
There exists a constant depending only on $m$(and not on $n$) such that, if we define,
\begin{equation} \label{eq:momestmtn}
\begin{aligned}
MT_n:&= \sum_{x^1_1, x^1_2 \in \mathcal{S}^1} \sum_{j=1}^{ n^{1-\beta}} \sum_{x^2 \in \mathcal{S}^2_{\beta,j}} \mathcal{I}(x^1_2,j) \mathcal{I}(x^1_1,j)\mathbb{P}(R'_{x^1_1} \cap \mathcal{S}^1 = \emptyset) G_D(x^1_1 - x^2)  \\ & \hspace{2cm}  \times \mathbb{P}(R'_{x^2} \cap \mathcal{S}^2_{\beta,j} = \emptyset) 
  G_D(\tilde{x}^2_j - x^1_2) \mathbb{P}(R'_{x^1_2} \cap \mathcal{S}^1 = \emptyset),
\end{aligned}
\end{equation}
then,
\begin{equation*}
\mathbb{E}[(MT_n)^m] \le C_m \frac{n^m(\log \log n)^{2m}}{(\log n)^{3m}}.
\end{equation*}
\end{lem}

By Markov's inequality, we obtain the following as a consequence,
\begin{cor} \label{cor:mtn}
Recall $MT_n$ and fix any $\epsilon > 0$.  
If $b_n= O(\log \log n)$ with $\lim_{n \to \infty} b_n = \infty$, we see that we have,
\begin{equation*}
\lim_{n \to \infty }\frac{1}{b_n} \log \mathbb{P}\left( MT_n \ge \epsilon \frac{ b_n n}{(\log n)^2} \right) =-\infty. 
\end{equation*}

\end{cor}

\begin{proof}

By applying Markov's inequality to \eqref{eq:momestmtn} for some fixed power $\mathbb{E}[(MT_n)^m]$, we can derive that,
\begin{equation*}
\begin{aligned}
&\limsup_{n \to \infty}\frac{1}{b_n} \log \mathbb{P}\left( MT_n \ge \epsilon \frac{ b_n n}{(\log n)^2} \right) \\ &\le \limsup_{n \to \infty} -\frac{m}{b_n} \left[\log\epsilon - \frac{ \log C_m}{m}+ \log \log n - 2 \log \log \log n +\log b_n \right] 
\le - m \frac{\log \log n}{b_n}.
\end{aligned}
\end{equation*}
The quantity above will go to $\infty$ as one takes $m$ to $\infty$. 
    
\end{proof}

\begin{proof}[Proof of Lemma \ref{lem:MTn}]
Since the proof is similar to Claim \ref{claim:TL} and the estimate of $\mathcal{E}_2$, we explain it very briefly. 
We estimate $MT_n$ by decomposing the term of $G_D(x^1_1 - x^2)$ in $MT_n$ by $G_D(x^1_1 - x^2)-G_D(x^1_1- \tilde{x}^2_j)$ and $G_D(x^1_1- \tilde{x}^2_j)$. 
Concerning the term of $MT_n$ including $G_D(x^1_1- \tilde{x}^2_j)$, note that by \cite[Theorem 3.5.1]{LA91}, 
\begin{align}\label{interrandom}
\mathbb{P}(\mathcal{S}^1 (0,\infty) \cap (\mathcal{S}^2[0,n] \cup \mathcal{S}^3[0,n])=\emptyset)
\lesssim (\log n)^{-1},
\end{align}
where $\mathcal{S}^3$ is an independent random walk from $\mathcal{S}^1$ and $\mathcal{S}^2$. 
If one could freely replace the probability of non-intersection of the random walks appearing in the above expression of \eqref{eq:momestmtn} with $O\left( \frac{1}{\log n} \right)$, then 
the term of $MT_n$ including $G_D(x^1_1- \tilde{x}^2_j)$ is a consequence of Lemma \ref{iterated:bound} and repeating the proof of Claim \ref{claim:TL} with the aid of \eqref{interrandom}. 
Concerning the term of $MT_n$ including $G_D(x^1_1 - x^2)-G_D(x^1_1- \tilde{x}^2_j)$, notice that here we could make the replacement $G_D(x^1_1 - x^2) \approx G_D(x^1_1- \tilde{x}^2_j)$ since $\| x^2 - \tilde{x}^2_j\| \le n^{\beta}$ as they are both elements of $\mathcal{S}^2_{\beta,j}$ under $\mathcal{I}(x^1_1,j) $. Then we can estimate the term of $MT_n$ including $G_D(x^1_1 - x^2)-G_D(x^1_1- \tilde{x}^2_j)$ by a similar argument to the estimate of $\mathcal{E}_2$. 
\end{proof}

\begin{lem}\label{iterated:bound}
Consider the following quantity,
\begin{equation*}
\begin{aligned}
MT_n':&= \sum_{x^1_1, x^1_2 \in \mathcal{S}^1} \sum_{x^2 \in \mathcal{S}^2}  G_D(x^1_1 - x^2) G_D(x^2 - x^1_2).
\end{aligned}
\end{equation*}
There exists some constant $C_m$(not depending on $n$) such that
\begin{equation*}
\mathbb{E}[(MT'_n)^m] \le C_m  n^{m}(\log \log n)^{2m}.
\end{equation*}
\end{lem}
\begin{proof}
Let $\alpha > 4m$. 
First, we show that 
for any $y_i \in \Z^4$ with $\inf_{1\le i\le 2m} \|y_i\|\ge n^{1/2}(\log n)^{-\alpha}$, 
\begin{equation}\label{wed}
\begin{aligned}
\mathbb{E}\left[\sum_{k_1,\ldots ,k_{2m} =1}^n  \prod_{i=1}^{2m} G_D(\mathcal{S}^1_{k_i} - y_i) \right]
\le  C_{m,\alpha} (\log \log n)^{2m}.
\end{aligned}
\end{equation}
Let $A_i=\{\|\mathcal{S}^1_{k_{i-1}}-y_{i}\|\ge n^{1/2}(\log n)^{-2\alpha}\}$. 
Indeed, 
\begin{equation*}
\begin{aligned}
&\sum_{1\le k_1\le \ldots \le k_{2m}\le n}  \prod_{i=1}^{2m} G_D(\mathcal{S}^1_{k_i} - y_i) \\
=&\sum_{1\le k_1\le \ldots \le k_{2m}\le n} 
G_D(\mathcal{S}^1_{k_{2m}}-\mathcal{S}^1_{k_{2m-1}}+\mathcal{S}^1_{k_{2m-1}}- y_{2m}) 
\prod_{i=1}^{2m-1} G_D(\mathcal{S}^1_{k_i} - y_i) \\
=&\sum_{1\le k_1\le \ldots \le k_{2m}\le n} 
G_D(\mathcal{S}^1_{k_{2m}}-\mathcal{S}^1_{k_{2m-1}}+\mathcal{S}^1_{k_{2m-1}}- y_{2m}) \mathbbm{1}_{A_{2m}}
\prod_{i=1}^{2m-1} G_D(\mathcal{S}^1_{k_i} - y_i) \\
+&\sum_{1\le k_1\le \ldots \le k_{2m}\le n} 
G_D(\mathcal{S}^1_{k_{2m}}-\mathcal{S}^1_{k_{2m-1}}+\mathcal{S}^1_{k_{2m-1}}- y_{2m}) \mathbbm{1}_{A_{2m}^c}
\prod_{i=1}^{2m-1} G_D(\mathcal{S}^1_{k_i} - y_i). 
\end{aligned}
\end{equation*}
Note that \cite[Thm. 1.2.1]{LA91} that for any $x \in \Z^4$, $i \ge 1$,
\begin{align}\label{gre2}
  \mathbb{P}(\mathcal{S}^1_i=x) &\lesssim   i^{-2} \Big[ e^{-2\|x\|^2/i} + (\|x\|^2 \vee i)^{-1} \Big]. 
\end{align}
Hence if $\|y\|\ge n^{1/2}(\log n)^{-2\alpha}$, 
\begin{equation*}
\begin{aligned}
\mathbb{E}\left[\sum_{1\le k \le n}  G_D(\mathcal{S}^1_k - y)  \right]
\lesssim  \log \log n
\end{aligned}
\end{equation*}
and
$$
\mathbb{E}\left[\sum_{1\le k_{2m-1}\le n}\mathbbm{1}_{A_{2m}^c} G_D(\mathcal{S}^1_{k_{2m-1}} -y_{2m-1})\right] \lesssim (\log n)^{-4\alpha} \times \log n.
$$
In addition, 
\begin{equation*}
\begin{aligned}
\max_{y_1, \ldots, y_{2m}} \mathbb{E}\left[\sum_{k_1,\ldots ,k_{2m} =1}^n  
\prod_{i=1}^{2m} G_D(\mathcal{S}^1_{k_i} - y_i) \right]
\le C_{m} (\log n)^{2m}.
\end{aligned}
\end{equation*}
Then, 
\begin{equation*}
\begin{aligned}
&\mathbb{E}\left[\sum_{1\le k_1\le \ldots \le k_{2m}\le n} 
G_D(\mathcal{S}^1_{k_{2m}}-\mathcal{S}^1_{k_{2m-1}}+\mathcal{S}^1_{k_{2m-1}}- y_{2m}) \mathbbm{1}_{A_{2m}}
\prod_{i=1}^{2m-1} G_D(\mathcal{S}^1_{k_i} - y_i) \right]\\
\lesssim &  (\log \log n)
\times \mathbb{E}\left[\sum_{1\le k_1\le \ldots \le k_{2m-1}\le n}  
\prod_{i=1}^{2m-1} G_D(\mathcal{S}^1_{k_i} - y_i)\right]. 
\end{aligned}
\end{equation*}
and by Lemma \ref{lem:BridgeBound},
\begin{equation*}
\begin{aligned}
&\mathbb{E}\left[\sum_{1\le k_1\le \ldots \le k_{2m}\le n} 
G_D(\mathcal{S}^1_{k_{2m}}-\mathcal{S}^1_{k_{2m-1}}+\mathcal{S}^1_{k_{2m-1}}- y_{2m}) \mathbbm{1}_{A_{2m}^c}
\prod_{i=1}^{2m-1} G_D(\mathcal{S}^1_{k_i} - y_i) \right]\\
\lesssim &  (\log n) \mathbb{E}\left[\sum_{1\le k_1\le \ldots \le k_{2m-1}\le n}  
 \mathbbm{1}_{A_{2m}^c}\prod_{i=1}^{2m-1} G_D(\mathcal{S}^1_{k_i} - y_i)\right]\\
 \le & C_m (\log n)^{2m} (\log n)^{-4\alpha}.
\end{aligned}
\end{equation*}

Hence, we obtain \eqref{wed}. 
Moreover, if $\inf \{i_1,\ldots ,i_m\} \ge n(\log n)^{-\alpha}$,
\begin{equation*}
\begin{aligned}
\mathbb{E}[\mathbbm{1}_D]:=\mathbb{E}\left[ \cup_{i=1}^m \mathbbm{1}
\{\|\mathcal{S}^2_i\|\le n^{1/2}(\log n)^{-\alpha} \} \right]
\le C_{m} (\log n)^{-4\alpha}
\end{aligned}
\end{equation*}
and 
\begin{equation*}
\begin{aligned}
\sum_{\inf \{i_1,\ldots ,i_m\} \le n(\log n)^{-\alpha}} 1 
\le  n^m (\log n)^{-\alpha}. 
\end{aligned}
\end{equation*}
Hence, 
\begin{equation*}
\begin{aligned}
&\mathbb{E}[(MT'_n)^m] \le
\mathbb{E}\left[\sum_{ i_1,\ldots ,i_m =1}^n \sum_{k_1,\ldots ,k_{2m} =1}^n  
\prod_{j=1}^{2m} G_D(\mathcal{S}^1_{k_j} - \mathcal{S}^2_{i_{\lceil j/2 \rceil }}) \right]\\
\le & C_m n^m (\log n)^m (\log n)^{-\alpha} 
+\mathbb{E}\left[\sum_{\inf \{i_1,\ldots ,i_m\} \ge n(\log n)^{-\alpha}}  \sum_{k_1,\ldots ,k_{2m} =1}^n  
\prod_{j=1}^{2m} G_D(\mathcal{S}^1_{k_j} - \mathcal{S}^2_{i_{\lceil j/2 \rceil }}) \right] \\
\le & C_m n^m (\log n)^m (\log n)^{-\alpha} 
+\mathbb{E}\left[\sum_{\inf \{i_1,\ldots ,i_m\} \ge n(\log n)^{-\alpha}}  \sum_{k_1,\ldots ,k_{2m} =1}^n  
\mathbbm{1}_{D^c} \prod_{j=1}^{2m} G_D(\mathcal{S}^1_{k_j} - \mathcal{S}^2_{i_{\lceil j/2 \rceil }}) \right]\\
\le & C_m n^m (\log  \log n)^{2m}.
\end{aligned}
\end{equation*}
Therefore, we obtain the desired result. 
\end{proof}

The other terms of equation \eqref{eq:splitchiba} are of much smaller order.
\begin{lem} \label{lem:E23}
Consider the second summand on the right hand side of equation \eqref{eq:splitchiba}. Namely, let
\begin{equation}\label{eq:defe1}
\begin{aligned}
\mathcal{E}_1 &:=\sum_{x^1_1, x^1_2 \in \mathcal{S}^1} \sum_{j=1}^{ n^{1-\beta} } \sum_{x^2 \in \mathcal{S}^2_{\beta,j}} \mathcal{I}(x^1_2,j) \mathcal{I}(x^1_1,j)\mathbb{P}(R'_{x^1_1} \cap \mathcal{S}^1 = \emptyset) G_D(x^1_1 -  \tilde{x}^2_j)\\ & \hspace{2cm} \times E_{x^2,x^1_2} \mathbb{P}(R'_{x^1_2} \cap\mathcal{S}^1 = \emptyset).
\end{aligned}
\end{equation}
We have that,
\begin{equation} \label{eq:momeste1}
\mathbb{E}[|\mathcal{E}_1|] \lesssim n^{2 \beta +\frac{3}{2} \delta + \frac{1}{2}}.
\end{equation}

We also have a similar estimate for the third summand on the right hand side of equation \eqref{eq:splitchiba}. Namely, we have that,
\begin{equation} \label{eq:defe2}
\begin{aligned}
\mathcal{E}_2:=&
 \sum_{x^1_1, x^1_2 \in \mathcal{S}^1} \sum_{j=1}^{ n^{1-\beta}} \sum_{x^2 \in \mathcal{S}^2_{\beta,j}} \mathcal{I}(x^1_2,j) \mathcal{I}(x^1_1,j)\mathbb{P}(R'_{x^1_1} \cap \mathcal{S}^1 = \emptyset) \\ & \hspace{2cm}\times [ G_D(x^1_1 -x^2) - G_D(x^1_1 - \tilde{x}^2_j)]  E_{x^2,x^1_2} \mathbb{P}(R'_{x^1_2} \cap \mathcal{S}^1 = \emptyset),
\end{aligned}
\end{equation}
will satisfy,
\begin{equation}\label{eq:momeste2}
\mathbb{E}[|\mathcal{E}_2|] \lesssim n^{3 \beta +2 \delta }.
\end{equation}
\end{lem}
Lemma \ref{lem:E23} will be shown later in this section. 
As before, Markov's inequality will give. one can derive that,
\begin{cor} \label{cor:core12}
Recall the terms $\mathcal{E}_1$ and $\mathcal{E}_2$ from equations \eqref{eq:defe1} and \eqref{eq:defe2}. Fix any $\epsilon>0$ and set $b_n= O(\log \log n)$ with $\lim_{n \to \infty} b_n = \infty$. We have that,
\begin{equation} \label{eq:conclusione23}
\lim_{n \to \infty} \frac{1}{b_n} \log 
\mathbb{P}\left( \max(|\mathcal{E}_1|,|\mathcal{E}_2|) \ge \epsilon \frac{nb_n}{(\log n)^2}\right) = -\infty.
\end{equation}

\end{cor}
\begin{proof}
By Markov's inequality applied to \eqref{eq:momeste1} and \eqref{eq:momeste2}, we have that 
\begin{equation*}
\frac{1}{b_n} \log \mathbb{P}\left(|\mathcal{E}_1| \ge \epsilon \frac{n b_n}{(\log n)^2}\right) \le \frac{- \log \epsilon - (\frac{1}{2} -2\beta - \frac{3}{2}\delta) \log n + 2 \log \log n }{b_n}, 
\end{equation*}
and,
\begin{equation*}
\frac{1}{b_n} \log \mathbb{P}\left(|\mathcal{E}_2| \ge \epsilon \frac{nb_n}{(\log n)^2} \right) \le \frac{-\log \epsilon - (1 -3\beta - 2\delta) \log n + 2\log \log n }{b_n}.
\end{equation*}
The desired conclusion \eqref{eq:conclusione23} follows from taking the limit $n \to \infty$.
\end{proof}

Finally, the last error term, the fourth summand of \eqref{eq:splitchiba} will also be of smaller order.

\begin{lem}
Consider the last summand of \eqref{eq:splitchiba},
\begin{align}\label{{E}_3}
 \mathcal{E}_3:=
 \sum_{x^1_1, x^1_2 \in \mathcal{S}^1} \sum_{j=1}^{ n^{1-\beta}} \sum_{x^2 \in \mathcal{S}^2_{\beta,j}}
 &[1 -  \mathcal{I}(x^1_2,j) \mathcal{I}(x^1_1,j)]  \mathbb{P}(R'_{x^1_1} \cap \mathcal{S}^1 = \emptyset) \\
 \notag
 & G_D(x^1_1 - x^2) G_{\mathcal{S}^2_{\beta,j}}(x^2, x^1_2) \mathbb{P}(R'_{x^1_2} \cap \mathcal{S}^1 = \emptyset).
\end{align}
We have that, for some $\delta'>0$,
\begin{equation*}
\mathbb{E}\left[ \mathcal{E}_3 \right] \lesssim n^{1-\delta'}.
\end{equation*}
\end{lem}

\begin{proof}
By \eqref{gre2}, we remark that for any point $a$ that,
\begin{equation}\label{green**}
\mathbb{E} \left[ \sum_{x^2 \in \mathcal{S}^2}G_D(a -x^2)\right] 
\le \sum_{j=0}^n \sum_{i=j}^\infty \mathbb{P}(\mathcal{S}_i=a)
\lesssim \sum_{j=0}^n \sum_{i=j}^\infty  i_+^{-2}
\lesssim \log n,
\end{equation}
where $i_+:=1\vee i$. 
By symmetry, 
\begin{equation*}
\begin{aligned}
 \mathbb{E}[\mathcal{E}_3]
\le & \mathbb{E}\bigg[\sum_{i,j,k=0}^n  [1 - \mathbbm{1}[\|\mathcal{S}^1_i - \mathcal{S}^2_j\| \ge \sqrt{n}^{1-\delta}] \mathbbm{1}[\|\mathcal{S}^1_k - \mathcal{S}^2_j\| \ge \sqrt{n}^{1-\delta}  ] ] \\ & \hspace{2cm} \times G_D(\mathcal{S}^1_i - \mathcal{S}^2_j)  G_D(\mathcal{S}^2_j - \mathcal{S}^1_k)\bigg]\\
\le & 2 \mathbb{E}\bigg[\sum_{i,j,k=0}^n  
\mathbbm{1}[\|\mathcal{S}^1_i - \mathcal{S}^2_j\| \le \sqrt{n}^{1-\delta}]
G_D(\mathcal{S}^1_i - \mathcal{S}^2_j) 
G_D(\mathcal{S}^2_j-\mathcal{S}^1_i - (\mathcal{S}^1_k-\mathcal{S}^1_i))\bigg].
\end{aligned}
\end{equation*}
By \eqref{green**} and Markov's property, it is bound by 
\begin{equation*}
\begin{aligned}
  & C (\log n) \times  \mathbb{E}\left[\sum_{i,j=0}^n  
\mathbbm{1}[\|\mathcal{S}^1_i - \mathcal{S}^2_j\| \le \sqrt{n}^{1-\delta}]
G_D(\mathcal{S}^1_i - \mathcal{S}^2_j) \right]. 
\end{aligned}
\end{equation*}
Now, by \cite[Lemma 4.1]{DO},
\begin{equation*}
\begin{aligned}
   \mathbb{E}\left[\sum_{i,j=0}^{n^{1-\delta/2}}  
\mathbbm{1}[\|\mathcal{S}^1_i - \mathcal{S}^2_j\| \le \sqrt{n}^{1-\delta}]
G_D(\mathcal{S}^1_i - \mathcal{S}^2_j) \right]
\le \mathbb{E}\left[\sum_{i,j=0}^{n^{1-\delta/2}} 
G_D(\mathcal{S}^1_i - \mathcal{S}^2_j) \right]
 \lesssim n^{1-\delta/2}
\end{aligned}
\end{equation*}
and 
\begin{equation*}
\begin{aligned}
&\mathbb{E}\left[\sum_{i=n^{1-\delta/2}}^n \sum_{j=0}^n  
\mathbbm{1}[\|\mathcal{S}^1_i - \mathcal{S}^2_j\| \le \sqrt{n}^{1-\delta}]
G_D(\mathcal{S}^1_i - \mathcal{S}^2_j) \right]\\
\le & n  \max_{y \in \Z^4} \mathbb{E}\left[\sum_{i=n^{1-\delta/2}}^n  
\mathbbm{1}[\|\mathcal{S}^1_i - y\| \le \sqrt{n}^{1-\delta}]
G_D(\mathcal{S}^1_i - y) \right]\\
\lesssim &   n \max_{y \in \Z^4} 
\sum_{i=n^{1-\delta/2}}^n 
\sum_{\|x-y\|\le \sqrt{n}^{1-\delta}}\|x-y\|^{-2} \mathbb{P}(\mathcal{S}^1_i=x).
\end{aligned}
\end{equation*}
Then again by \eqref{gre2}, it is bound by
\begin{equation*}
\begin{aligned}
C(\log n)n^{1-\delta/2}+ C n(\log n) \max_{y \in \Z^4} 
 \sum_{i=n^{1-\delta/2}}^n 
\sum_{\|x-y\|\le \sqrt{n}^{1-\delta}}\|x-y\|^{-2} i_+^{-2}
\lesssim (\log n)n^{1-\delta/2}.
\end{aligned}
\end{equation*}
Therefore, by symmetry, it completes the proof.

\end{proof}

As before, one can show the following from Markov's inequality,
\begin{cor} \label{cor:e3}
Recall the term $\mathcal{E}_3$ from \eqref{{E}_3}. Fix any $\epsilon>0$ and set $b_n= O(\log \log n)$ with $\lim_{n \to \infty} b_n = \infty$. Then, we see that,
\begin{equation}
\lim_{n \to \infty} \frac{1}{b_n} \log \mathbb{P}\left( \mathcal{E}_3 \ge \epsilon \frac{n b_n}{n (\log n)^2}\right) = -\infty.
\end{equation}

\end{cor}
The proof is similar to that of Corollary \ref{cor:core12} and will not be shown here. 

Using the previous corollaries, one can now prove Theorem \ref{thm:chipr}.

\begin{proof}[Proof of Theorem \ref{thm:chipr}]
It is clear that,

\begin{equation*}
\begin{aligned}
&\mathbb{P}\left( \chi' \ge \epsilon \frac{n b_n}{(\log n)^2} \right) \le \mathbb{P}\left( \chi'_{\beta} \ge \epsilon \frac{n b_n}{(\log n)^2} \right)\\ & \le \mathbb{P} \left( MT_n \ge \frac{\epsilon}{4} \frac{b_n n}{(\log n)^2}\right)  
+ \sum_{i=1}^3 \mathbb{P}\left( |\mathcal{E}_i| \ge \frac{\epsilon}{4} \frac{b_n n}{(\log n)^2} \right) .
\end{aligned}
\end{equation*}

By computing $\frac{1}{b_n} \log$ to both sides, we see that the conclusion of Theorem \ref{thm:chipr} is a consequence of Corollaries \ref{cor:mtn}, \ref{cor:core12}, and \ref{cor:e3}.

\end{proof}


\subsection{Proof of Lemma \ref{lem:E23}}
The proof of this lemma requires novel techniques beyond careful computations of Green's functions due to the presence of the error terms $E$ occurring from the matrix inversion. We will present the proof here.

\begin{proof}[Proof of Lemma \ref{lem:E23}]
First, we deal with $\mathcal{E}_1$. 
We have that,
\begin{equation*}
\begin{aligned}
\mathcal{E}_1&=\sum_{x^1_1 \in \mathcal{S}^1} \sum_{j=1}^{ n^{1-\beta}}  \mathcal{I}(x^1_2,j) \mathcal{I}(x^1_1,j)\mathbb{P}(R'_{x^1_1} \cap \mathcal{S}^1 = \emptyset) G_D(x^1_1 -  \tilde{x}^2_j) \\ & \times \sum_{x^1_2 \in \mathcal{S}^1} \sum_{x^2 \in \mathcal{S}^2_{\beta,j}} E_{x^2,x^1_2} \mathbb{P}(R'_{x^1_2}\cap \mathcal{S}^1 = \emptyset).
\end{aligned}
\end{equation*}
We now consider, under the indicator function $\mathcal{I}(x^1_2,j)$,
\begin{equation*}
\begin{aligned}
&\sum_{x^1_2 \in \mathcal{S}^1} \sum_{x^2 \in \mathcal{S}^2_{\beta,j}} |E_{x^2,x^1_2}| \mathbb{P}(R'_{x^1_2}\cap \mathcal{S}^1 = \emptyset)
\\& \le \sum_{x^1_2 \in \mathcal{S}^1}  \sqrt{ |\mathcal{S}^2_{\beta,j}| \sum_{x^2 \in \mathcal{S}^2_{\beta,j}} |E_{x^2,x^1_2}|^2} 
\lesssim n \sqrt{n^{2\beta}\frac{n^{2\beta}}{n^{3 -3 \delta}}} 
\lesssim n^{2\beta +3/2\delta -1/2}.
\end{aligned}
\end{equation*}
To get the first inequality, we used the Cauchy-Schwartz inequality on the sum over $\mathcal{S}^2_{\beta,j}$. 
From Lemma \ref{clm:matinv}, the matrix $\mathcal{G}^{\mathcal{S}^2_{\beta,j}}$ is positive definite and has minimum eigenvalue greater than $1/2$. 
Recall that $a_1$ is defined in $E_{x^2,x^1_2}$ in \eqref{eq:decompose}. 
Thus, the inverse matrix has $l^2\to l^2$ operator norm less than $2$. Thus, we know that, under the indicator function $\mathcal{I}(x^1_2,j)$, 
\begin{equation*}
\sqrt{ \sum_{x^2 \in \mathcal{S}^2_{\beta,j}} |E_{x^2,x^1_2}|^2} 
\le  2 \sqrt{ \sum_{x^2 \in \mathcal{S}^2_{\beta,j}} |G_D(a_1 - x^1_2) - G_D(x^2 - x^1_2)|^2 } 
\lesssim  \sqrt{ n^{\beta}  \frac{n^{2\beta}}{n^{3-3\delta}}}.
\end{equation*}
In the final inequality, we used the deterministic bound \eqref{eq:diffnbd} to bound the differences of the Green's function in the region $\mathcal{S}^2_{\beta,j}$. 
Thus, we have that, deterministically, 
\begin{equation*}
\begin{aligned}
&\bigg|\sum_{x^1_1, x^1_2 \in \mathcal{S}^1} \sum_{j=1}^{ n^{1-\beta} } \sum_{x^2 \in \mathcal{S}^2_{\beta,j}} \mathcal{I}(x^1_2,j) \mathcal{I}(x^1_1,j)\mathbb{P}(R'_{x^1_1} \cap \mathcal{S}^1 = \emptyset) G_D(x^1_1 -  \tilde{x}^2_j) \\ & \hspace{2cm}\times E_{x^2,x^1_2} \mathbb{P}(R'_{x^1_2} \cap\mathcal{S}^1 = \emptyset)\bigg|\\
& \lesssim  \sum_{x^1_1 \in \mathcal{S}^1} \sum_{j=1}^{ n^{1-\beta} }  \mathcal{I}(x^1_2,j) \mathcal{I}(x^1_1,j)\mathbb{P}(R'_{x^1_1} \cap \mathcal{S}^1 = \emptyset) G_D(x^1_1 -  \tilde{x}^2_j)  n^{2 \beta +3/2 \delta -1/2}
\\ & \lesssim  n^{2 \beta +3/2 \delta -1/2}  \sum_{i=1}^n \sum_{j=1}^{n} G_D(\mathcal{S}^1_i - \mathcal{S}^2_j).
\end{aligned}
\end{equation*}
Therefore,
\begin{equation*}
\begin{aligned}
&\mathbb{E}\bigg[\bigg|\sum_{x^1_1, x^1_2 \in \mathcal{S}^1} \sum_{j=1}^{ n^{1-\beta}} \sum_{x^2 \in \mathcal{S}^2_{\beta,j}} \mathcal{I}(x^1_2,j) \\ & \hspace{2cm} \times\mathcal{I}(x^1_1,j)\mathbb{P}(R'_{x^1_1} \cap \mathcal{S}^1 = \emptyset) G_D(x^1_1 -  \tilde{x}^2_j) E_{x^2,x^1_2} \mathbb{P}(R'_{x^1_2} \cap\mathcal{S}^1 = \emptyset)\bigg|\bigg]\\
\lesssim &  n^{2\beta +3/2 \delta -1/2} \mathbb{E}\left[ \sum_{i=1}^n \sum_{j=1}^{n} G_D(\mathcal{S}^1_i - \mathcal{S}^2_j)\right] 
\lesssim n^{2\beta + 3/2\delta + 1/2}.
\end{aligned}
\end{equation*}
The computation of the expectation in the last line above comes from \cite[Lemma 4.1]{DO}.  The value of this last line is approximately $n^{-1/2}$ the scale of the main order term (provided $\beta,\delta$ are all chosen relatively small).

Now, the other error term involving $E$ can be dealt with in a similar way to $\mathcal{E}_3$.
To recall, the other error term is,
\begin{equation*}
\begin{aligned}
|\mathcal{E}_2| \le\bigg| &\sum_{x^1_1, x^1_2 \in \mathcal{S}^2} \sum_{j=1}^{n^{1-\beta}} \sum_{x^2 \in \mathcal{S}^2_{\beta,j}} \mathcal{I}(x^1_2,j) \mathcal{I}(x^1_1,j)\mathbb{P}(R'_{x^1_1} \cap \mathcal{S}^1 = \emptyset) \\
&\times [ G_D(x^1_1 -x^2) - G_D(x^1_1 - \tilde{x}^2_j)]  E_{x^2,x^1_2} \mathbb{P}(R'_{x^1_2} \cap \mathcal{S}^1 = \emptyset)\bigg|.
\end{aligned}
\end{equation*}
We first use the improved estimate,
\begin{equation*}
|G_D(x^1_1 - x^2) - G_D(x^1_1 - \tilde{x}^2_j)| \lesssim G_D(x^1_1 - \tilde{x}^2_j)  \frac{n^{\beta}}{n^{1/2 - \delta/2}}
\end{equation*}
under the indicator function $\mathcal{I}(x^1_1,j)$. We remark here that the factor of $G_D(x^1_1- \tilde{x}^2_j)$ is an improved error term in the case that $\|x^1_1 - \tilde{x}^2_j\|$ is relatively large. With this deterministic bound in hand, bounding this error term in $E$ reduces to the error term we just treated.
\end{proof}

\section{ The Leading Term of $\chi$: Proof of Theorem \ref{thm:TL}} \label{sec:TL}

In this section, we will consider the large deviation statistics of the following  quantity,
\begin{equation*}
TL_n:=   \sum_{x^1 \in \mathcal{S}^1} \sum_{x^2 \in \mathcal{S}^2} \mathbb{P}(R'_{x^1} \cap \mathcal{S}^1= \emptyset) G_D(x^1 - x^2) \mathbb{P}(R'_{x^2} \cap  \mathcal{S}^2 = \emptyset).
\end{equation*}
We will prove Theorem \ref{thm:TL} by separately proving lower and upper bounds for the asymptotic moments.

\subsection{Introduction of the Auxiliary $TL'_n$}

For technical reasons, $TL_n$ is not the most convenient quantity to manipulate. Instead, we consider the following auxiliary quantity. We let $\mathcal{S}^{k,i}$ denote the portion of the random walk in between the part of the random walk $\mathcal{S}^k \left[(i-1) \frac{n}{b_n}, i \frac{n}{b_n}\right]$ and 
\begin{equation}\label{eq:deftlprn}
TL'_n:= \sum_{i,j=1}^{b_n} \sum_{x^1 \in \mathcal{S}^{1,i}} \sum_{x^2 \in \mathcal{S}^{2,j}} \mathbb{P}(R'_{x^1} \cap \mathcal{S}^{1,i} = \emptyset) G_D(x^1- x^2) \mathbb{P}(R'_{x^2} \cap \mathcal{S}^{2,j} = \emptyset).
\end{equation}

We have the following relationship between $TL_n$ and $TL'_n$.
\begin{prop} \label{prop:TltoTlpr}
Let $b_n$ be a sequence satisfying $b_n = O(\log \log n)$ and $\lim_{n \to \infty} b_n = \infty$. Fix $\lambda \ge 0$. Then, we have that,
\begin{equation}
\lim_{n \to \infty}\frac{1}{b_n} \log \mathbb{P}\left( TL'_n \ge \lambda \frac{n b_n}{(\log n)^2} \right) = \lim_{n \to \infty} \frac{1}{b_n}\log \mathbb{P}\left( TL_n \ge \lambda \frac{n b_n}{(\log n)^2} \right).
\end{equation}
\end{prop}
\begin{proof}

We remark that  $TL'_n \ge TL_n$. This immediately shows that,
\begin{equation*}
\mathbb{P}\left(TL'_n \ge \lambda \frac{n b_n}{(\log n)^2} \right) \ge \mathbb{P}\left(TL_n \ge \lambda \frac{n b_n}{(\log n)^2} \right).
\end{equation*}

To derive the opposite inequality, we first observe that $TL'_n - TL_n$ can be bounded from above by,
\begin{equation} \label{eq:difftlpr}
\begin{aligned}
& TL' _n- TL_n \\
\le & \sum_{i,j=1} ^{b_n} \sum_{x^1 \in \mathcal{S}^{1,i}} \sum_{x^2 \in \mathcal{S}^{2,j}} \mathbb{P}(R'_{x^1} \cap \mathcal{S}^{1,i} =\emptyset, R'_{x^1} \cap \mathcal{S}^1 \ne \emptyset) G_D(x^1 - x^2) \mathbb{P}(R'_{x^2} \cap \mathcal{S}^{2,j} = \emptyset) \\
&+ \sum_{i,j=1}^{b_n} \sum_{x^1 \in \mathcal{S}^{1,i}} \sum_{x^2 \in \mathcal{S}^{2,j}}  \mathbb{P}(R'_{x^1} \cap \mathcal{S}^1 = \emptyset) G_D(x^1- x^2) \mathbb{P}(R'_{x^2} \cap \mathcal{S}^{2,j} = \emptyset, R'_{x^2} \cap \mathcal{S}^2 \ne \emptyset)\\
& + \sum_{i_1 \ne i_2,j=1}^{b_n} \sum_{x^1 \in \mathcal{S}^{1,i_1} \cap \mathcal{S}^{1,i_2}}  \sum_{x^2 \in \mathcal{S}^{2,j}}\mathbb{P}(R'_{x^1} \cap \mathcal{S}^1 = \emptyset) G_D(x^1 - x^2)  \mathbb{P}(R'_{x^2} \cap \mathcal{S}^2 = \emptyset)\\
& +  \sum_{i,j_1 \ne j_2=1}^{b_n} \sum_{x^2 \in \mathcal{S}^{2,j_1} \cap \mathcal{S}^{2,j_2}} \sum_{x^1 \in \mathcal{S}^{1,i}} \mathbb{P}(R'_{x^1} \cap \mathcal{S}^1 = \emptyset) G_D(x^1 - x^2)  \mathbb{P}(R'_{x^2} \cap \mathcal{S}^2 = \emptyset)\\
=:&J_1+J_2+J_3+J_4.
\end{aligned}
\end{equation}
For each line on the right hand side above, we will show that for $1\le i\le 4$
\begin{equation}\label{eq:genericrel}
\lim_{n \to \infty} \frac{1}{b_n} \log \mathbb{P}\left( J_i \ge \epsilon \frac{n b_n}{(\log n)^2} \right) = -\infty.
\end{equation}

The first two lines of the right hand side of \eqref{eq:difftlpr} are very similar to the type of error terms we have dealt with in Section \ref{sec:ProofofThmchipr}.
One can follow the analysis of said section to show the relation \eqref{eq:genericrel} for these two lines.

The last two lines will be controlled by bounding the moments and applying Markov's inequality.
We present the analysis with the term on the third line, since the term on the fourth line can be dealt with similarly.
We first bound all the probability terms on the line by $1$.

By \eqref{gre2} and \eqref{green**}, 
for any $x$ and $y$, 
\begin{equation*}
\begin{aligned}
\mathbb{E} \left[\sum_{i,j=0}^n \mathbbm{1}\{\mathcal{S}^1_i+x=\mathcal{S}^2_j+y\}\right]
= &\sum_{i,j=0}^n \mathbb{P}( \mathcal{S}^1_{i+j}=y-x)\\
\lesssim & \sum_{i,j=0}^n  (i+j)_+^{-2} \lesssim \log n.
\end{aligned}
\end{equation*}
Thus, we see that,
\begin{equation*} \label{eq:SelfInterror}
\begin{aligned}
\mathbb{E}\left[  J_3 \right] 
 &\le \mathbb{E}_{\mathcal{S}^1}\left[ \sum_{i_1 \ne i_2 = 1}^{b_n} \sum_{x^1 \in \mathcal{S}^{1,i_1} \cap \mathcal{S}^{1,i_2}}  \mathbb{E}_{\mathcal{S}^2} \left[ \sum_{x^2 \in \mathcal{S}^2} G_D(x^1- x^2)\right] \right] \\
 &\lesssim  \mathbb{E}\left[ \sum_{i_1 \ne i_2=1}^{b_n} I_{i_1, i_2} \log n \right]
\lesssim b_n^2 (\log n)^2.
\end{aligned}
\end{equation*}
On the second line $\mathbb{E}_{\mathcal{S}^i}$ is the expectation with respect to only the randomness of $\mathcal{S}^i$. $I_{i_1,i_2}$ is the number of points of intersection between $\mathcal{S}^1_{i_1}$ and $\mathcal{S}^1_{i_2}$.  We remark that $\mathcal{S}^1_{i_2} - \mathcal{S}^1_{i_2 \frac{n}{b_n}}$ and $\mathcal{S}^1_{i_1} - \mathcal{S}^1_{i_1 \frac{n}{b_n}}$ are independent random walks.
Then, $\mathbb{E}[I_{i_1, i_2}] \lesssim \log n$ by a similar computation to \cite[Proposition 4.3.1]{LA91}. 
Thus, this term will not contribute to the large deviation statistics of $TL'_n$ on the scale of $\frac{n b_n}{(\log n)^2}$.

\end{proof}

The quantity $TL'_n$ is easier to deal with since we can obtain exact moment asymptotics. Namely,
\begin{prop} \label{prop:asympmom}
Recall $TL'_n$ from equation \eqref{eq:deftlprn}. Let $b_n=O(\log \log n) $ and $\lim_{n \to \infty} b_n = \infty$. Then, for any $\theta>0$, we have the following exact moment asymptotics on $TL'_n$: 
\begin{equation} \label{eq:asympmom}
 \lim_{n \to \infty}\frac{1}{b_n} \log \sum_{m=0}^{\infty} \frac{1}{m!} \theta^m  \left(\frac{\sqrt{b_n} \log n}{\sqrt{n}}\right)^m \mathbb{E}[(TL'_n)^m]^{1/2} =\tilde{\kappa}(4,2)^{4} \frac{\pi^4 \theta^2}{8}.
\end{equation}

\end{prop}

As a consequence of the previous two propositions, one can now prove Theorem \ref{thm:TL}.

\begin{proof}[Proof of Theorem \ref{thm:TL}]
By \cite[Theorem 1.2.7]{Chenbook}, equation \eqref{eq:asympmom} would be equivalent to showing,
\begin{equation*}
\lim_{n \to \infty } \frac{1}{b_n} \log \mathbb{P}\left( TL'_n \ge \lambda \frac{n b_n}{(\log n)^2} \right) = - \frac{4}{\pi^4} \tilde{\kappa}(4,2)^{-4} \lambda.
\end{equation*}
Now, since by Proposition \ref{prop:TltoTlpr} we have that 
\begin{equation*}
\lim_{n \to \infty } \frac{1}{b_n} \log \mathbb{P}\left( TL_n \ge \lambda \frac{n b_n}{(\log n)^2} \right) =\lim_{n \to \infty } \frac{1}{b_n} \log \mathbb{P}\left( TL'_n \ge \lambda \frac{n b_n}{(\log n)^2} \right) ,
\end{equation*}
we complete the proof of the proposition.
\end{proof}

The remainder of this section is devoted to deriving upper and lower bounds to the quantity in equation \eqref{eq:asympmom}.

\subsection{Large Deviation Upper Bounds}

In this section, we establish the upper bound found in Proposition \ref{prop:asympmom}.

\begin{prop} \label{prop:asympboundupr}
Let $b_n$ be a sequence satisfying $b_n = O(\log \log n)$ and $\lim_{n \to \infty} b_n = \infty$. Then, for any $\theta>0$, we satisfy,
\begin{equation*}
 \limsup_{n \to \infty}\frac{1}{b_n} \log \sum_{m=0}^{\infty} \frac{1}{m!} \theta^m  \left(\frac{\sqrt{b_n} \log n}{\sqrt{n}}\right)^m \mathbb{E}[(TL'_n)^m]^{1/2} 
 \le \tilde{\kappa}(4,2)^{4} \frac{\pi^4 \theta^2}{8}.
\end{equation*}
\end{prop}

The proposition above is an immediate consequence of the following lemma and claim.

\begin{clm} \label{clm:Tlprbnd}
There exists some constant $C>0$ such that for all $n, m>0$, we have that,
\begin{equation}\label{eq:apriormombnd}
\mathbb{E}[(TL'_n)^m] \le C^m m! \left( \frac{n}{(\log n)^2}\right)^m.
\end{equation}
\end{clm}

The proof of the above claim will be postponed to later. We now present the second necessary lemma.

\begin{lem} \label{lem:Apriorbnd}
For any $\theta>0$,
\begin{equation*}
 \limsup_{n \to \infty}\frac{1}{b_n} \log \sum_{m=0}^{\infty} \frac{1}{m!} \theta^m  \left(\frac{\sqrt{b_n} \log n}{\sqrt{n}}\right)^m \mathbb{E}[(TL'_n)^m]^{1/2} \le \tilde{\kappa}(4,2)^{4} \frac{\pi^4 \theta^2}{8}.
\end{equation*}
\end{lem}

\begin{proof}
Let $(B^1_s)_{s\ge 0}$ and $(B^2_s)_{s\ge 0}$ be independent Brownian motions for $d=4$. 
The need for the bound in \eqref{eq:apriormombnd} and \cite[(2.3)]{BCR2} are to ensure that  one can apply dominated convergence to the terms $ \mathbb{E}\left[ \left(\frac{(\log n)^2}{n} TL'_n \right)^m\right]$ when needed, and replace them with the term: 
\begin{equation*}
\left(\frac{\pi^4}{4}\right)^m\mathbb{E}\left[\left(\int_0^1 \int_0^1 G(B^1_t - B^2_s) \text{d}t \text{d}s\right)^{m}\right]. 
\end{equation*}
The reason why this can be done is due to the fact that
\begin{equation*}
\frac{(\log n)^2}{n} TL'_n 
\stackrel{\D}{\Longrightarrow}
\frac{\pi^4}{4} \int_0^1 \int_0^1 G(B^1_t - B^2_s) \text{d}t \text{d}s
\end{equation*}
following the proof of \cite[Proposition 6.1]{As5}.

We can follow the proof of \cite[Theorem 7.2.1]{Chenbook} to derive the appropriate upper bound. Finally, by Remark \ref{ggnfi}, we see we obtain our desired constant.


\end{proof}

It is manifest that Proposition \ref{prop:asympboundupr} is a consequence of Claim \ref{clm:Tlprbnd} and Lemma \ref{lem:Apriorbnd}. We devote the rest of this subsection to deriving Claim \ref{clm:Tlprbnd}.

\subsubsection{A proof of Claim \ref{clm:Tlprbnd}}


Our first remark is that the quantity $TL'_n$ is less than,
\begin{equation} \label{eq:defTLalpha}
\begin{aligned}
TL_{n,\alpha}:= \sum_{i,j=1}^n & \mathbb{P}( R'_{\mathcal{S}^1_i} \cap  \mathcal{S}^1[i - n^{\alpha}, i + n^{\alpha}] \cap \mathcal{S}^1 = \emptyset)  G_D(\mathcal{S}^1_i - \mathcal{S}^2_j)\\& \times \mathbb{P}(R'_{\mathcal{S}^2_j} \cap \mathcal{S}^2[j- n^{\alpha}, j + n^{\alpha}]\cap \mathcal{S}^2=\emptyset).
\end{aligned}
\end{equation}

We will analyze the moments of $TL_{n,\alpha}$ via a subadditivity argument along with a careful moment analysis. Our first subadditivity argument allows us to reduce our moment analysis of $TL_{n,\alpha}$ to a slightly weaker analysis.

\begin{lem} \label{lem:weaktostrong}
If one knows that there exists some constant $C$ such that for all $n$ and $m$ that 
\begin{equation}\label{eq:weakbnd}
\mathbb{E}[(TL_{n,\alpha})^m] \le C^m (m!)^2 \left(\frac{n}{(\log n)^2}\right)^m,
\end{equation} 
then  there is some other constant $C'$ such that,
\begin{equation}
\mathbb{E}[(TL_{n,\alpha})^m] \le (C')^m (m!)  \left(\frac{n}{(\log n)^2}\right)^m.
\end{equation}

\end{lem}
\begin{proof}
We start by using a subadditivity argument. 
Recall that $G_D = \tilde{G}_D * \tilde{G}_D$. 
To match the notation of \cite[Chapter 6.1]{Chenbook}, we also define
\begin{equation*}
\mathcal{F}^{a}_{\mathcal{S}(n',n]} = \sum_{i=n'}^n \tilde{G}_D(\mathcal{S}_i-a) \mathbb{P}(R'_{\mathcal{S}_i} \cap \mathcal{S}[i - (n-n')^{\alpha}, i + (n-n')^{\alpha}] \cap \mathcal{S}(n',n] = \emptyset).
\end{equation*}
The main thing to observe about this function is that,
\begin{equation*}
TL_{n,\alpha}= \sum_{a \in \mathbb{Z}^4} \mathcal{F}^{a}_{\mathcal{S}^1[1,n]} \mathcal{F}^a_{\mathcal{S}^2[1,n]}. 
\end{equation*}
Furthermore, it is trivially true that for times $t<s$ that
\begin{equation*}
\mathcal{F}^{a}_{\mathcal{S}[1,s]} \le \mathcal{F}^{a}_{\mathcal{S}[1,t]} + \mathcal{F}^{a}_{\mathcal{S}(t,s]}
\end{equation*}
and $\mathcal{F}^{a}_{\mathcal{S}}$ has the translation symmetry,
\begin{equation*}
\mathcal{F}^{a+z}_{\mathcal{S} +z} = \mathcal{F}^{a}_{\mathcal{S}}.
\end{equation*}

For these reasons, we can apply all the results of \cite[Section 6.1]{Chenbook}.
In particular, we can apply the argument of \cite[Theorem 6.2.1]{Chenbook}.

\end{proof}

It remains to prove equation \eqref{eq:weakbnd}. 

\begin{lem}\label{lem:pfwkbnd}
 Equation \eqref{eq:weakbnd} holds. Namely, there is a constant such that for all $n$ and $m$ we have that,
\begin{equation*}
\mathbb{E}[(TL_{n,\alpha})^m] \le C^m (m!)^2 \left(\frac{n}{(\log n)^2}\right)^m.
\end{equation*} 
\end{lem} 
Before we start proving the above lemma, we will finish the proof of Claim \ref{clm:Tlprbnd}.
\begin{proof}[Proof of Claim \ref{clm:Tlprbnd}]
Since $TL_n'\le TL_{n,\alpha} $, we have
by Lemmas \ref{lem:pfwkbnd} and \ref{lem:weaktostrong} that
\begin{equation*}
\mathbb{E}[(TL_n')^m] \le \mathbb{E}[(TL_{n,\alpha})^m]  \le C^m m! \left(\frac{n}{(\log n)^2}\right)^m.
\end{equation*}
 This is exactly what was desired. 
\end{proof}

We now return to the proof of Lemma \ref{lem:pfwkbnd}.

\subsubsection{The proof of Lemma \ref{lem:pfwkbnd}}

To show Lemma \ref{lem:pfwkbnd}, we first need the following claim: 

\begin{clm} \label{claim:TL}
Define 
\begin{equation*}
\begin{aligned}
\overline{TL}^m_{n,\alpha}:=& \sum_{\substack{i_1,\ldots,i_m\\ |i_a - i_b| \ge n^{3\alpha} \forall a,b}} \sum_{\substack{j_1,\ldots,j_m\\ |j_a - j_b| \ge n^{3\alpha} \forall a,b}}  \prod_{k=1}^m\mathbb{P}( R'_{\mathcal{S}^1_{i_k}} \cap  \mathcal{S}^1[i_k - n^{\alpha}, i_k + n^{\alpha}]\cap \mathcal{S}^1 = \emptyset)\\& \times  G_D(\mathcal{S}^1_{i_k} - \mathcal{S}^2_{j_k}) \mathbb{P}(R'_{\mathcal{S}^2_{j_k}} \cap \mathcal{S}^2[j_k- n^{\alpha}, j_k + n^{\alpha}]\cap \mathcal{S}^2=\emptyset).
\end{aligned}
\end{equation*}
Then, there exists some constant $B$ such that,
\begin{equation*}
\mathbb{E}[\overline{TL}_{n,\alpha}^m] \le (m!)^2 B^m \frac{n^m}{(\log n)^{2m}}.
\end{equation*}
\end{clm}

We will show it after the proof of Lemma \ref{lem:pfwkbnd}. 

\begin{proof} [Proof of Lemma \ref{lem:pfwkbnd}]
In what follows, the constant $C$ may not remain the same from line to line. 
Since $TL_{n,\alpha} \le \sum_{i,j=1}^n G_D(\mathcal{S}^1_i -\mathcal{S}^2_j)$, it is clear that there is some constant $C$ such that 
\begin{equation*}
\mathbb{E}[(TL_{n,\alpha})^m] \le \mathbb{E}\left[ \left( \sum_{i,j=1}^n G_D(\mathcal{S}^1_i -\mathcal{S}^2_j) \right)^m\right] \le C^m m! n^m.
\end{equation*}
This latter estimate immediately follows from the large deviation statistics of 
$\sum_{i,j=1}^n G_D(\mathcal{S}^1_i -\mathcal{S}^2_j)$ from \cite[Lemma 4.1]{DO}. 
Now, observe that when $m \ge (\log n)^2$. One has that 
\begin{equation*}
m! \ge  m^m e^{-m} \ge (\log n)^{2m} e^{-m}.
\end{equation*}
Thus, for $m \ge (\log n)^2$, we have that,
\begin{equation*}
C^m m! n^m \le  (eC)^m (m!)^2  \left(\frac{n}{(\log n)^2} \right)^m.
\end{equation*}
It suffices to prove an upper bound for moments when $m \le (\log n)^2$. 

\textit{Bounding the moments when $m \le (\log n)^2$}


We will show that there exists a constant $C$ such that,
\begin{equation*}
\mathbb{E}[(TL_{n,\alpha})^m] \le (m!)^2 C^m \frac{n^m}{(\log n)^{2m}}
\end{equation*}
by induction on $m$.

Since the points $i_k$ are all spaced far apart, we are able to  use in some form that the probability terms $\mathbb{P}( R'_{\mathcal{S}^1_{i_k}} \cap  \mathcal{S}^1[i_k - n^{\alpha}, i_k + n^{\alpha}]\cap \mathcal{S}^1 = \emptyset)$ should be rather independent of each other. We will return to the proof of the claim later. Assuming the claim we have the following result; the moments of $TL_{n,\alpha}$ can be bounded from above by 
\begin{equation} \label{eq:TLalpha}
\begin{aligned}
\mathbb{E}[(TL_{n,\alpha})^m] & \le \mathbb{E}[\overline{TL}^m_{n,\alpha}] \\
& + 2m^2 \mathbb{E}\bigg[ \sum_{\substack{i_1,\ldots,i_m \\ |i_1 - i_2| \le n^{3\alpha}}} \sum_{j_1,\ldots,j_m} \prod_{k=1}^m\mathbb{P}( R'_{\mathcal{S}^1_{
i_k}} \cap  \mathcal{S}^1[i_k - n^{\alpha}, i_k + n^{\alpha}] \cap \mathcal{S}^1= \emptyset)\\& \times  G_D(\mathcal{S}^1_{i_k} - \mathcal{S}^2_{j_k}) \mathbb{P}(R'_{\mathcal{S}^2_{j_k}} \cap \mathcal{S}^2[j_k- n^{\alpha}, j_k + n^{\alpha}]\cap \mathcal{S}^2=\emptyset)\bigg].
\end{aligned}
\end{equation}
Namely, if there is a term in the $m$th moment of $TL_{n,\alpha}$ that is not already contained in the term $\overline{TL}^{m}_{n,\alpha}$, there must be some pair of points $(i_a,i_b)$ or $(j_a,j_b)$ that are of a distance closer than $n^{3\alpha}$. By symmetry, we may assume that the two points are $i_1$ and $i_2$. There are at most $2m^2$ such choices of pairs $(i_a,i_b)$ or $(j_a,j_b)$. We will now bound the moment of the second term above.

If could only sum over the terms $i_2,\ldots,i_m$ and $j_2,\ldots,j_m$, then this would be the $m-1$th moment of $(TL_{n,\alpha})$. We could then apply induction to this quantity. The main idea is that if we fix $i_2$ there are at most $2n^{3\alpha}$ choices of $i_1$. Thus, intuitively, this term should be no more than $n^{3\alpha}$ times $\mathbb{E}[(TL_{n,\alpha})^{m-1}]$. The problem is to deal with the sum over $j_1$.

Observe the following,
\begin{equation} \label{eq:divmore}
\begin{aligned}
 &\mathbb{E}\bigg[ \sum_{\substack{i_1,\ldots,i_m \\ |i_1 - i_2| \le n^{3\alpha}}} \sum_{j_1,\ldots,j_m} \prod_{k=1}^m\mathbb{P}( R'_{\mathcal{S}^1_{i_k}} \cap  \mathcal{S}^1[i_k - n^{\alpha}, i_k + n^{\alpha}]\cap \mathcal{S}^1 = \emptyset)\\& \times  G_D(\mathcal{S}^1_{i_k} - \mathcal{S}^2_{j_k}) \mathbb{P}(R'_{\mathcal{S}^2_
{j_k}} \cap \mathcal{S}^2[j_k- n^{\alpha}, j_k + n^{\alpha}]\cap \mathcal{S}^2=\emptyset)\bigg]\\
& \le \sum_{k=2}^m    \mathbb{E}\bigg[ \sum_{\substack{i_1,\ldots,i_m \\ |i_1 - i_2| \le n^{3\alpha}}} \sum_{\substack{j_1,\ldots,j_m \\ |j_1 - j_k| \le n^{3\alpha}}}\prod_{k=1}^m\mathbb{P}( R'_{\mathcal{S}^1_{i_k}} \cap  \mathcal{S}^1[i_k - n^{\alpha}, i_k + n^{\alpha}] \cap \mathcal{S}^1= \emptyset)\\& \times  G_D(\mathcal{S}^1_{i_k} - \mathcal{S}^2_{j_k}) \mathbb{P}(R'_{\mathcal{S}^2_{j_k}} \cap \mathcal{S}^2[j_k- n^{\alpha}, j_k + n^{\alpha}]\cap \mathcal{S}^2=\emptyset)\bigg] \\
& + \mathbb{E}\bigg[ \sum_{\substack{i_1,\ldots,i_m \\ |i_1 - i_2| \le n^{3\alpha}}} \sum_{\substack{j_1,\ldots,j_m \\ |j_1 - j_k| \ge n^{3\alpha} \forall k}}\prod_{k=1}^m\mathbb{P}( R'_{\mathcal{S}^1_{i_k}} \cap  \mathcal{S}^1[i_k - n^{\alpha}, i_k + n^{\alpha}] \cap \mathcal{S}^1 = \emptyset)\\& \times  G_D(\mathcal{S}^1_{i_k} - \mathcal{S}^2_{j_k}) \mathbb{P}(R'_{\mathcal{S}^2_{j_k}} \cap \mathcal{S}^2[j_k- n^{\alpha}, j_k + n^{\alpha}]\cap \mathcal{S}^2=\emptyset)\bigg].
\end{aligned}
\end{equation}
If we bound the product $$\mathbb{P}( R'_{\mathcal{S}^1_{i_1}} \cap  \mathcal{S}^1[i_1 - n^{\alpha}, i_1 + n^{\alpha}]\cap \mathcal{S}^1 = \emptyset) G_D(\mathcal{S}^1_{i_1} - \mathcal{S}^2_{j_1})  \mathbb{P}(R'_{\mathcal{S}^2_{j_1}} \cap \mathcal{S}^2[j_1- n^{\alpha}, j_1 + n^{\alpha}]\cap \mathcal{S}^2=\emptyset)$$ by $1$, we see that the first term on the right hand side above in equation \eqref{eq:divmore} can indeed be bounded by $ \lesssim m n^{6\alpha} \mathbb{E}[(TL_{n,\alpha})^{m-1}]$.

To deal with the second term, we do the following. First, fix the terms $i_2,\ldots,i_m$ and $j_2,\ldots,j_m$. Without loss of generality, we can assume that  we order $j_2 \le j_3 \le \ldots \le j_{m-1} \le j_m$, and that $  j_2 \le j_1 \le j_3$. (We can apply similar logic regardless of the relative position of  ${ j_1}$ in the ordering $j_2 \le \ldots \le j_{m}$.)
Notice that upon conditioning on the values of $\mathcal{S}^2_{ j_2 + n^{\alpha}}$ and $\mathcal{S}^2_{ j_3 - n^{\alpha}}$, the walk $\mathcal{S}^2[{j_2} + n^{\alpha}, { j_3} - n^{\alpha}]$ becomes independent of the rest of the walk. We exploit this fact by  using that
\begin{equation} \label{eq:Greenbnd}
\mathbb{E}_{\mathcal{S}^2}\left[ \sum_{ { j_2} + n^{\alpha} \le j_1 \le j_3 - n^{\alpha}} G_D(\mathcal{S}^1_{i_1}  - \mathcal{S}^2_{j_1}) \bigg| \mathcal{S}^2_{j_2 + n^{\alpha}} = x, \mathcal{S}^2_{j_3 - n^{\alpha}} = y \right] \lesssim \log n
\end{equation}
for any pairs of values $x$ and $y$ by Lemma \ref{lem:BridgeBound}.  The expectation above is only taken over the random walk $\mathcal{S}^2$.  (Note, here we are bounding the probability term $ \mathbb{P}(R'_{\mathcal{S}^2_{j_1}} \cap \mathcal{S}^2[j_1- n^{\alpha}, j_1 + n^{\alpha}]\cap \mathcal{S}^2=\emptyset)$  by $1$ to simplify further computations.)

As a consequence, we see that we have,
\begin{equation*}
\begin{aligned}
& \mathbb{E}\bigg[ \sum_{\substack{i_1,\ldots,i_m \\ |i_1 - i_2| \le n^{3\alpha}}} \sum_{\substack{j_1,\ldots,j_m \\ |j_1 - j_k| \ge n^{3\alpha} \forall k}}\prod_{k=1}^m\mathbb{P}( R'_{\mathcal{S}^1_{i_k}} \cap  \mathcal{S}^1[i_k - n^{\alpha}, i_k + n^{\alpha}]\cap \mathcal{S}^1 = \emptyset)\\& \times  G_D(\mathcal{S}^1_{i_k} - \mathcal{S}^2_{j_k}) \mathbb{P}(R'_{\mathcal{S}^2_{j_k}} \cap \mathcal{S}^2[j_k- n^{\alpha}, j_k + n^{\alpha}]\cap \mathcal{S}^2=\emptyset)\bigg]\\
& \lesssim m n^{3\alpha} (\log n)  \mathbb{E}[TL_{n,\alpha}^{m-1}] 
\lesssim m C^{m-1} n^{3 \alpha} (\log n) (m-1)!^2 \frac{n^{m-1}}{(\log n)^{2m-2}}.
\end{aligned}
\end{equation*}
The factor of $n^{3\alpha}$ comes from the possible choices of $x_1$ (given its distance from $x_2$) and the factor of $m$ comes from the fact that $j_1$ can be located in between any of the $m$ regions $[j_i, j_{i+1}]$ in the ordering $j_2 \le j_3 \ldots \le j_m$. At the final step, we applied the induction hypothesis. 

Returning to equation \eqref{eq:divmore}, we see that,
\begin{equation*}
\begin{aligned}
&\text{L.H.S. of } \eqref{eq:divmore} \\
&\lesssim  m n^{6\alpha} C^{m-1} (m-1)!^2 \frac{n^{m-1}}{(\log n)^{2m-2}} + m (\log n) (m-1)!^2 C^{m-1}\frac{n^{m-1}}{(\log n)^{2m-2}}.
\end{aligned}
\end{equation*}

Substituting this back into equation \eqref{eq:TLalpha}, we have,
\begin{equation*}
\mathbb{E}[(TL_{n,\alpha})^m] \le (m!)^2  B^m \frac{n^m}{( \log n)^m } + K m^3 n^{6 \alpha} C^{m-1} (m-1)!  \frac{n^{m-1}}{(\log n)^{2m-2}}.
\end{equation*}
Notice that the right hand side is less than $C^m  (m!)^2 \frac{n^m}{(\log n)^{2m}}$ provided,
\begin{equation*}
 1 \ge  \left(\frac{B}{C} \right)^m + K C^{-1} n^{6 \alpha -1} (\log n)^3  \ge \left(\frac{B}{C} \right)^m + K C^{-1} m n^{6 \alpha -1} (\log n)^2.
\end{equation*}
Provided $C$ is chosen large relative to $B$ and the universal constant $K$, there  is a value of $C$ such that the above inequality will be satisfied for all $n$ and $m \le (\log n)^2$. This completes the induction provided that Claim \ref{claim:TL} holds.

\end{proof}
 Now we start complete the proof of Claim \ref{claim:TL}.
\begin{proof}[Proof of Claim \ref{claim:TL}]
Without loss of generality, we may order the times as $i_1\le i_2 \le i_3 \ldots \le i_m$.  
Our first step is to condition on the values of the random walk at specific points as  $\mathcal{S}^{1}_{i_k} = x^c_k$, $\mathcal{S}^1_{i_k +n^{\alpha}}=  x^{r}_k$, $\mathcal{S}^1_{i_k - n^{\alpha}} = x^l_k$ and $\mathcal{S}^{2}_{j_k}= y^c_k$, $\mathcal{S}^2_{j_k + n^{\alpha}} = y^r_k$ , $\mathcal{S}^{2}_{j_k - n^{\alpha}} = y^l_k$.  With the endpoints of the neighborhoods $\mathcal{S}^1[i_k - n^{\alpha}, i_k + n^{\alpha}]$ specified, the neighborhoods involved in the probability terms above become independent of each other.  This is the key observation used to simplify the computations that proceed. In what follows, we let $p_t(x)$ denote the probability that a SRW transitions to the point $x$ at time $t$.

To simplify what proceeds, we introduce the following notation,
\begin{equation}
NI(\mathcal{S}, i, x^c, x^r, x^l):= \mathbb{E}[\mathbb{P}(R'_{\mathcal{S}_i} \cap \mathcal{S}[i-n^{\alpha},i+n^{\alpha}]\cap \mathcal{S}=\emptyset)|\mathcal{S}_i=x^c,\mathcal{S}_{i+n^{\alpha}} = x^r, \mathcal{S}_{i-n^{\alpha}} =x^l].
\end{equation}

This finds the expected value of the probability that an independent random walk $R'_{\mathcal{S}_i}$ starting at $\mathcal{S}_i$ does not intersect the portion of the random walk $\mathcal{S}[i- n^{\alpha},i+n^{\alpha}]$ conditioned on the random walk  being at points $x^c$ at time $i$, $x^r$ at time $i+n^{\alpha}$ and $x^l$ at time $i- n^{\alpha}$. If it is not necessary to condition $x^r$ and $x^l$, we will slightly abuse notation and denote this by dropping the appropriate argument on the left hand side. 
Let $\Pi_m$ be the collection of all permutations on $m$ points. Note that $c$ as a superscript is used as a shorthand for `center' while $r$ and $l$ are `right' and `left' respectively. 
We can write the expectation of $\mathbb{E}[\overline{TL}_{n,\alpha}^m]$ as,

\begin{equation} \label{eq:verylong}
\begin{aligned}
& (m!) \sum_{\sigma \in \Pi_m} \sum_{\substack{1 \le i_1\le i_2 - n^{3\alpha} \le  \ldots \\ \le i_m - (m-1)  n^{3^{\alpha}} \le n  -(m-1)n^{3\alpha}}} \sum_{\substack{1 \le  j_{\sigma(1)} \le j_{\sigma(2)} - n^{\alpha}\ldots \\ \le j_{\sigma(m)} - (m-1) n^{3\alpha} \le n - (m-1)n^{3\alpha}}}\sum_{\substack{x^{r}_k, x^{l}_k, x^{c}_k ,\\ y^r_k,y^l_k, y^c_k \in \mathbb{Z}^4, \forall k }} \\& p_{i_1 }(x^c_1) p_{n^{\alpha}}(x^r_1 - x^c_1) p_{i_2 - i_1-2n^{\alpha}}(x_2^l - x_1^r) NI(\mathcal{S}^1,i_1, x_1^c,x_1^r)\\& \times p_{j_\sigma(1) }(y^c_{\sigma(1)})  p_{n^{\alpha}}(y^r_{\sigma(1)} -y^c_{\sigma(1)})  p_{j_{\sigma(2)} - j_{\sigma(1)} - 2n^{\alpha}}(y^r_{\sigma(2)} - y^l_{\sigma(1)})NI(\mathcal{S}^2,j_{\sigma(1)},y^c_{\sigma(1)},y^r_{\sigma(1)}) \\
&\prod_{k=2}^{m-1} NI(\mathcal{S}^1,i_k,x^c_k,x^r_k,x^l_k) p_{n^{\alpha}}(x^c_k - x^r_k) p_{n^{\alpha}}(x^l_k - x^c_k)  p_{i_{k+1} - i_k - 2 n^{\alpha}}( x^l_{k+1} - x^r_{k})
\\& \times  NI(\mathcal{S}^2,j_{\sigma(k)},y^c_{\sigma(k)},y^r_{\sigma(k)},y^l_{\sigma(k)})p_{n^{\alpha}}(y^c_{\sigma(k)} - y^r_{\sigma(k)}) p_{n^{\alpha}}(y^l_{\sigma(k)} - y^c_{\sigma(k)}) \\ &  p_{y_{\sigma(k+1)} - y_{\sigma(k)} - 2n^{\alpha}}(y^l_{\sigma(k+1)} - y^r_{\sigma(k)})\\
&NI(\mathcal{S}^1,i_m,x_m^c,x_m^l) p_{n^{\alpha}}(x_m^c - x_m^l)
 NI(\mathcal{S}^2, j_{\sigma(m)},y_{\sigma(m)}^c, y_{\sigma(m)}^r) p_{n^{\alpha}}(y_{\sigma(m)}^c - y_{\sigma(m)}^l)
\\& \times \prod_{k=1}^m G_D(x^c_k - y^c_k).
\end{aligned}
\end{equation}
The main observation to notice now is that if we were able to freely sum over the values $x^r_k, x^l_k$, then we would have that, by \eqref{interrandom},
\begin{equation*}
\begin{aligned}
&\sum_{x^r_k , x^l_k} NI(\mathcal{S}^1,i_k,x^c_k,x^r_k,x^l_k)  p_{n^{\alpha}}(x^c_k - x^r_k)p_{n^{\alpha}}(x^l_k - x^c_k) \\
= &\mathbb{E}[\mathbb{P}(R'_{\mathcal{S}_i} \cap \mathcal{S}[i- n^{\alpha},i+n^{\alpha}] = \emptyset)] \lesssim \frac{1}{\log (\min\{n-i+2,i+1\}^{\alpha})},
\end{aligned}
\end{equation*}
because this just computes the averaged probability that an infinite random walk does not intersect a the union of two independent random walks of length $n^{\alpha}$ starting from the origin. 

The only term that prevents us from freely summing over $x^r_k$ and $x^l_k$ for all $k$ is the term $p_{i_{k+1} - i_k - 2 n^{\alpha}} (x^l_{k+1} - x^r_k)$. However, if we could bound this term from above by a constant times $p_{i_{k+1} - i_k}(x^c_{k+1} - x^c_k)$, then we would be able to freely sum over the variables $x^r_k$ and $x^l_k$ as desired. This is what we will argue now.

It is clear that $\|x^l_{k} - x^c_k\| \le n^{\alpha}$ and $\|x^r_k - x^c_k\| \le n^{\alpha}$. Provided that $\|x^{l}_{k+1} - x^{r}_k\| \le (i_{k+1} - i_k)^{1/2 + \epsilon}$ for some small $\epsilon$, we can apply the local central limit as in \cite[Theorem 2.3.12, equation (2.46)]{LL10} along with the fact that $i_{k+1} - i_k \ge n^{3 \alpha}$ to show that,
\begin{equation*}
p_{i_{k+1} - i_k - 2n^{\alpha}}(x^l_{k+1} - x^r_k) \le (1 +o(1)) p_{i_{k+1} -i_k}(x^c_{k+1} - x^c_k).
\end{equation*}

Otherwise, the probability that $\|x^l_{k+1} - x^{r}_k\| \ge (i_{k+1} - i_k)^{1/2 +\epsilon}$ is exponentially unlikely with probability at most  $\exp[-n^{6\alpha \epsilon}]$. 
Thus, we always have the bound,
\begin{align*}
&p_{i_{k+1} - i_k - 2n^{\alpha}}(x^l_{k+1} - x^r_k) \\
\le & (1 +o(1)) p_{i_{k+1} -i_k}(x^c_{k+1} - x^c_k)\\&  + \mathbbm{1}[\|x^l_{k+1} - x^{r}_k\| \ge (i_{k+1} - i_k)^{1/2 +\epsilon}] p_{i_{k+1} - i_k - 2n^{\alpha}}(x^l_{k+1} - x^r_k).
\end{align*} 
Similar statements also hold for $j$ and $y$.

Furthermore,  this term can be substituted into  equation \eqref{eq:verylong} by replacing each appearance of $p_{i_{k+1} - i_k - 2n^{\alpha}}(x^l_{k+1} - x^r_k)$ with the right hand side above.  We can expand each of these products to get a sum over $2^{4m}$ terms  (in each of these terms, $p_{i_{k+1} - i_k - 2n^{\alpha}}(x^l_{k+1} - x^r_k) $ is replaced  with either $p_{i_{k+1} -i_k}(x^c_{k+1} - x^c_k)$  or $ \mathbbm{1}[\|x^l_{k+1} - x^{r}_k\| \ge (i_{k+1} - i_k)^{1/2 +\epsilon}] p_{i_{k+1} - i_k - 2n^{\alpha}}(x^l_{k+1} - x^r_k)$). There is only one of these terms in which each $p_{i_{k+1} - i_k - 2n^{\alpha}}(x^l_{k+1} - x^r_k) $ is replaced with $p_{i_{k+1} -i_k}(x^c_{k+1} - x^c_k)$.

We remark that if even one of the $p_{i_{k+1} - i_k - 2n^{\alpha}}(x^l_{k+1} - x^r_k)$ were replaced with $  \mathbbm{1}[\|x^l_{k+1} - x^{r}_k\| \ge (i_{k+1} - i_k)^{1/2 +\epsilon}] p_{i_{k+1} - i_k - 2n^{\alpha}}(x^l_{k+1} - x^r_k)$, then such a term would be exponentially suppressed.  Indeed, we could trivially bound all the terms of the form $G_D(x- y)$ and $\mathbb{P}(R'_{\mathcal{S}_t} \cap \mathcal{S}[t - n^{\alpha}, t+ n^{\alpha}]\cap \mathcal{S}=\emptyset)$ by a constant. Performing a trivial summation shows that this term can be no more than $n^{2m} \exp[- n^{6\alpha \epsilon}] \ll (m!)^2 \frac{n^m}{(\log n)^{2m}}$ provided $m \le (\log n)^2$.  Furthermore, there are no more than $2^{4m}$ such terms. Thus, these terms are clearly negligible.

Now we consider the term in which all $p_{i_{k+1} - i_k - 2n^{\alpha}}(x^l_{k+1} - x^r_k)$ are replaced with $(1 +o(1)) p_{i_{k+1} -i_k}(x^c_{k+1} - x^c_k) $. In such a term, we can finally sum over $x^r_k,x^l_k, y^r_k,y^l_k$ for all $k$. 
Such a term will be bounded by,
\begin{equation*}
\begin{aligned}
&\left(\frac{C}{\log n}  \right)^{2m}   m! \sum_{\sigma \in \Pi_m} \sum_{i_1  \le \ldots \le i_m} \sum_{j_1 \le \ldots \le j_m} \sum_{x^c_1,\ldots,x^c_k} \sum_{y^c_1,\ldots,y^c_k} p_{i_1}(x^c_1) p_{j_{\sigma(1)}}(y^c_{\sigma(1)})\\
& \times \prod_{k=1}^{m-1} p_{i_{k+1} - i_k}(x_{k+1} - x_k) p_{j_{\sigma(k+1)} - j_{\sigma(k)}}(y_{\sigma(k+1)} - y_{\sigma(k)}) \prod_{k=1}^m G_D(x_k^c- y_k^c).
\end{aligned}
\end{equation*}
However, the last term computes the $m$-th moment of $\sum_{i=1}^n \sum_{j=1}^n G_D(\mathcal{S}^1_i - \mathcal{S}^2_j)$. 
This is bounded by $C^m m! n^m$ for some $C>0$. Thus, we can bound the line above by $$m! \left(\frac{ C n}{(\log n)^2} (1+ o(1)) \right)^m .$$ This completes the proof of the claim.

\end{proof}

\subsection{Lower Bound for the Large Deviation of $TL'_n$}

Our goal in this section is to show the following statement.



In this section, our goal is to understand the lower bound of,
\begin{equation*}
\frac{1}{b_n} \log \sum_{m=0}^{\infty} \frac{1}{m!} \theta^{m} \left(\frac{b_n (\log n)^2}{n} \right)^{m/2} \left(  \mathbb{E}[(TL'_n)^m]\right)^{1/2}.
\end{equation*}
\begin{thm}
If $b_n = \text{O}(\log \log  n)$  and satisfies $\lim_{n \to \infty} b_n = \infty$, one has that for any $\theta>0$,
\begin{equation*}
\liminf_{n \to \infty}\frac{1}{b_n} \log \sum_{m=0}^{\infty} \frac{1}{m!} \theta^{m} \left(\frac{b_n (\log n)^2}{n} \right)^{m/2} \left(  \mathbb{E}[(TL'_n)^m]\right)^{1/2} \ge \tilde{\kappa}(4,2)^{4} \frac{\pi^4 \theta^2}{8}.
\end{equation*}

\end{thm}

\begin{proof}
Recall that we let $\mathcal{S}^{k,i}=\mathcal{S}^k \left[(i-1) \frac{n}{b_n}, i \frac{n}{b_n}\right]$. 
Without loss of generality, we assume that $b_n$ is odd. 
First, notice that 
\begin{equation*}
TL'_n = \sum_{a \in \mathbb{Z}^4} \sum_{i=1}^{b_n} \sum_{x^1 \in \mathcal{S}^{1,i}} \tilde{G}_D(x^1 -a)  \mathbb{P}(R'_{x^1} \cap \mathcal{S}^{1,i} = \emptyset) \sum_{j=1}^{b_n} \sum_{x^2 \in \mathcal{S}^{2,j}} \tilde{G}_D(x^2 - a) \mathbb{P}(R'_{x^2} \cap \mathcal{S}^{2,j} = \emptyset).
\end{equation*}
If we let
\begin{equation*}
\mathcal{G}_n (a):= \mathbb{E}\left[ \sum_{i=1}^{b_n} \sum_{x^1 \in \mathcal{S}^{1,i}}  \tilde{G}_D(x^1 - a) \mathbb{P}(R'_{x^1} \cap \mathcal{S}^{1,i} = \emptyset)\right],
\end{equation*}
we see that,
\begin{equation*}
\mathbb{E}[(TL'_n)^m] = \sum_{a_1,\ldots,a_m}  \prod_{i=1}^m (\mathcal{G}_n(a_i))^2.
\end{equation*}

If we now let $f$ be any smooth function with finite support satisfying $\int_{\mathbb{R}^4}f(a)^2 \text{d}a = 1$  and $C_f^n:= \sum_{a \in \mathbb{Z}^4} f\left(\sqrt{\frac{b_n}{n}}a\right)^2$, by the Cauchy-Schwartz inequality, we see that we obtain that,
\begin{equation*}
\begin{aligned}
\mathbb{E}[(TL'_n)^m]^{1/2} &= (C_f^n )^{-m/2}\left(\sum_{a_1,\ldots,a_m}  \prod_{i=1}^m (\mathcal{G}_n(a_i))^2  \right)^{1/2} \left(\sum_{a_1,\ldots,a_m} \prod_{i=1}^m f\left( \sqrt{\frac{b_n}{n}}a_i \right)^2 \right)^{1/2}\\
& \ge (C_f^n)^{-m/2} \sum_{a_1,\ldots,a_m} \prod_{i=1}^m\mathcal{G}_n(a_i) f\left( \sqrt{\frac{b_n}{n}}a_i\right).
\end{aligned}
\end{equation*}
Defining 
\begin{equation*}
\mathcal{F}_n^i :=  (C_f^n)^{-1/2}  \sum_{x^1 \in \mathcal{S}^{1,i}}\sum_{a \in \mathbb{Z}^4 }\tilde{G}_D(x^1-a) f\left( \sqrt{\frac{b_n}{n}}a\right) \mathbb{P}(R'_{x^1} \cap \mathcal{S}^{1,i}= \emptyset),
\end{equation*}
we see that
\begin{equation*}
\mathbb{E}[(TL'_n)^m]^{1/2} \ge \mathbb{E}\left[ \left(\sum_{i=1}^{b_n} \mathcal{F}_n^i \right)^m\right].
\end{equation*}
Thus, we see that,
\begin{equation} \label{eq: TLtoF}
\sum_{m=0}^{\infty} \frac{1}{m!} \theta^{m} \left(\frac{b_n (\log n)^2}{n} \right)^{m/2} \left(  \mathbb{E}[(TL'_n)^m]\right)^{1/2} \ge \mathbb{E}\left[ \exp\left[\theta \frac{\sqrt{b_n} (\log n)}{\sqrt{n}} \sum_{i=1}^{b_n}\mathcal{F}^i_n \right] \right].
\end{equation}
Furthermore, $\mathcal{F}_n^i$ is a function of only the portion $\mathcal{S}^{1,i}$ of the random walk. Hence, notice that for any $\epsilon>0$, we have,
\begin{equation*}
\begin{aligned}
\frac{1}{b_n} \log \mathbb{E} \left[ \exp\left[  \theta \frac{\sqrt{b_n} \log n}{\sqrt{n}} \sum_{i=1}^{b_n} \mathcal{F}^i_n\right] \right] 
\ge & \frac{1}{b_n} (1+\epsilon) 
\log \mathbb{E}\left[ \exp\left[ \frac{1}{1+\epsilon}\theta \frac{\sqrt{b_n} \log n}{\sqrt{n}} \sum_{i=2}^{b_n} \mathcal{F}^i_n \right]\right]\\
& - \frac{1}{b_n} \epsilon \log \mathbb{E}\left[ \exp\left( - \frac{1+\epsilon}{\epsilon}\theta \mathcal{F}^1_n\right) \right].
\end{aligned}
\end{equation*}

Notice that for any fixed $\epsilon$, an upper bound on the large deviation statistics of $\mathcal{F}^1_n$ (which can be inherited from an upper bound large statistics on $TL'_{\frac{n}{b_n}}$
 as in equation \eqref{eq: TLtoF}) shows that as $n \to \infty$ the term on the right goes to $0$. Finally, we can take $\epsilon$ to $0$ to note that the term,
\begin{equation*}
\begin{aligned}
& \liminf_{n \to \infty}\frac{1}{b_n} \log \mathbb{E} \left[ \exp\left[  \theta \frac{\sqrt{b_n} \log n}{\sqrt{n}} \sum_{i=1}^{b_n} \mathcal{F}^i_n\right] \right] \\& \ge \lim_{\epsilon \to 0}\liminf_{n \to \infty}\frac{1}{b_n} (1+\epsilon) \log 
\mathbb{E}\left[ \exp\left[ \frac{1}{1+\epsilon}\theta \frac{\sqrt{b_n} \log n}{\sqrt{n}} \sum_{i=2}^{b_n} \mathcal{F}^i_n \right] \right].
\end{aligned}
\end{equation*}

To find the  lower bound on the term on the left hand side, it suffices to find a bound for the right hand side. 
We are now in a very similar situation to that of \cite[Theorem 7.1.2]{Chenbook}. The functions $\mathcal{F}^i_n$ are not exactly in the same format. However, one can see that $\mathcal{F}^i_n$ takes the same role as that of the term $\sum_{x \in \mathcal{S}^{1,i}} f( \sqrt{\frac{b_n}{n}}x)$. 
Indeed, we can define the operator,
$$
\begin{aligned}
\mathfrak{B}_n \xi(x) := & \mathbb{E}\bigg(  \exp \bigg\{ \frac{\sqrt{b_n} (\log n)}{\sqrt{n}} \sum_{y-x \in \mathcal{S}[1,n b_n^{-1}]}  (C^n_f)^{-1/2}\sum_{a \in \mathbb{Z}^4} \tilde{G}_D(y-a) f \left( \sqrt{\frac{b_n}{n}}a \right)\\ & \times  \mathbb{P}\left(R'_{y-x} \cap \mathcal{S}[1, b_n n^{-1}]  = \emptyset\right)  \bigg\} \xi(x+ \mathcal{S}_{n b_n^{-1}} )\bigg).
\end{aligned}
$$

 We define $\xi_n$ as the following discretization of $g$. Namely, $\xi_n(x) = \frac{1}{C_g^{1/2}}g(\frac{x}{\sqrt{\det(\Gamma)}\sqrt{n}})$ where $C_g:= \sum_{x \in \mathbb{Z}^4} g^2(\frac{x}{\sqrt{\det(\Gamma)}\sqrt{n}})$, where $\Gamma=4^{-1} I$.

This operator is a symmetric operator and following the proof of \cite[Lemma 7.1.3]{Chenbook}. We see that we can derive the following bound.  Let $g$ be a bounded function on $\mathbb{R}^4$ that is infinitely differentiable and supported on a finite box with $\int_{\mathbb{R}^4} g^2(x) \text{d}  x =1$. By the Cauchy–Schwarz inequality, there exists a constant $\delta$ depending only on $g$ (but not on $n$) that (recall that $b_n-1$ is even)
\begin{equation*}
\mathbb{E}\left[ \exp\left[ \theta \frac{\sqrt{b_n} \log n}{\sqrt{n}} \sum_{i=2}^{b_n} \mathcal{F}^i_n\right] \right] \ge \delta \langle \xi_n, \mathfrak{B}_n^{b_n -1} \xi_n \rangle \ge \delta \langle \xi_n, \mathfrak{B}_n \xi_n \rangle^{b_n -1},
\end{equation*}
where $\langle \xi_n, \mathfrak{B}_n \xi_n \rangle$ is given by
\begin{equation}\label{invariant1}
\begin{aligned}
\langle \xi_n, \mathfrak{B}_n \xi_n \rangle &= (1+ o(1)) \int_{\mathbb{R}^4} \text{d}x g(x) \\
& \times \mathbb{E}\bigg( \exp\bigg[ \theta\frac{\sqrt{b_n} \log n}{\sqrt{n}} \sum_{y \in \mathcal{S}[1, nb_n^{-1}]}(C_f^n)^{-1/2} \sum_{a \in \mathbb{Z}^4} \tilde{G}_D(y +x -a) f\left( \sqrt{\frac{b_n}{n}} a\right)\\ & \times \mathbb{P}(R'_y \cap \mathcal{S}[1, nb_{n}^{-1}] =\emptyset) \bigg] g\left(x + \sqrt{\frac{b_n}{n}} \mathcal{S}_{nb_n^{-1}} \right) \bigg)\\
& \rightarrow \int_{\mathbb{R}^4}\text{d}x g(x)  \mathbb{E}\left( \exp\left\{ \int_0^1 \frac{\pi^2}{4}(\tilde{G}*f)(x+ B(t/4))  \text{d}t \right\} g(x + B(1/4)) \right)
\end{aligned}
\end{equation}
as $n \to \infty$, where $B$ is the $4-d$ Brownian motion. 
Note that by Lemmas \ref{lem:contGeps} and \ref{lem:Hestimate}, $\tilde{G} * f $ is a bounded continuous function and $\tilde{G}_D*f$ uniformly converges to $2\tilde{G}*f$. 
Then, the invariance principle shows the convergence above. 
By \cite[(4.1.25)]{Chenbook}, if we take $\log $ to the most right hand side in \eqref{invariant1}, it is equal to
\begin{equation*}
\begin{aligned}
\sup_{h \in \mathcal{F}} 
\bigg\{ \frac{\pi^2}{4} \int_{\mathbb{R}^4} \tilde{G}*f (\Gamma^{1/2}x) h(x)^2 \text{d}x
- \frac{1}{2} \int_{\mathbb{R}^4} |\nabla h (x)|^2 \text{d}x \bigg\}, 
\end{aligned}
\end{equation*}
where $\mathcal{F}:=\{h: \int h(x)^2 dx=1, \int |\nabla h(x)|^2 dx<\infty \}$. 
Taking the supremum over f with $\int f(x)^2 dx=1$, it is larger than or equal to, 
\begin{equation*}
\sup_{h \in \mathcal{F}} 
\bigg\{ \frac{\pi^2}{4}  \bigg(\int\int_{(\mathbb{R}^4)^2} G(x-y) h(x)^2 h(y)^2 \text{d}x \text{d}y \bigg)^{1/2}
- \frac{1}{8} \int_{\mathbb{R}^4} |\nabla h (x)|^2 \text{d}x \bigg\}. 
\end{equation*}
Therefore, by the same proof as \cite[Proposition 4.1]{AO}, we obtain the desired result. 

 Let us explain some steps in the derivation in \eqref{invariant1}. First, we remark that the term inside the exponential has finite expectation.  Secondly, we also have the second moment comparison estimate 
\begin{equation}
\begin{aligned}
&\mathbb{E} \bigg[ \frac{\sqrt{b_n} \log n}{\sqrt{n}} \sum_{y \in \mathcal{S}[1, nb_n^{-1}]}(C_f^n)^{-1/2} \sum_{a \in \mathbb{Z}^4} \tilde{G}_D(y +x -a) f\left( \sqrt{\frac{b_n}{n}} a\right) \mathbb{P}(R'_y \cap \mathcal{S}[1, nb_{n}^{-1}] =\emptyset) \\
&-  \frac{\pi^2}{8} \frac{\sqrt{b_n}}{\sqrt{n}}(C_f^n)^{-1/2} \sum_{i=1}^{nb_n^{-1}} \sum_{a \in \mathbb{Z}^4} \tilde{G}_D(\mathcal{S}_i +x -a) f\left( \sqrt{\frac{b_n}{n}} a\right)  \bigg]^2 \to 0
\end{aligned}
\end{equation} 
as $n \to \infty$. 
 As before, this follows from computations similar to those found in the proof of Claim   \ref{claim:TL} to allow us to replace the term of 
 $\mathbb{P}(R'_y \cap S[1,nb_{n}^{-1}] = \emptyset)$ with $(1+o(1))\frac{\pi^2}{8 \log n}$ with the aid of \cite[Theorem 5.1]{As5}.  Combining these observations, we see that as $n\to \infty$ 
\begin{equation*}
\begin{aligned}
&\bigg|\int_{\mathbb{R}^4} \text{d}x g(x) \mathbb{E}\bigg( \exp\bigg[ \theta\frac{\sqrt{b_n} \log n}{\sqrt{n}} \sum_{y \in \mathcal{S}[1, nb_n^{-1}]}(C_f^n)^{-1/2} \sum_{a \in \mathbb{Z}^4} \tilde{G}_D(y +x -a) f\left( \sqrt{\frac{b_n}{n}} a\right)\\ & \times \mathbb{P}(R'_y \cap \mathcal{S}[1, nb_{n}^{-1}] =\emptyset) \bigg] g\left(x + \sqrt{\frac{b_n}{n}} \mathcal{S}_{nb_n^{-1}} \right) \bigg) \\
-&\int_{\mathbb{R}^4} \text{d}x g(x) \mathbb{E}\bigg( \exp\bigg[ \theta \frac{\pi^2}{8}\frac{\sqrt{b_n} }{\sqrt{n}} \sum_{i=1}^{nb_n^{-1}}(C_f^n)^{-1/2} \sum_{a \in \mathbb{Z}^4} \tilde{G}_D(\mathcal{S}_i +x -a) f\left( \sqrt{\frac{b_n}{n}} a\right) \bigg]\\ & \times  g\left(x + \sqrt{\frac{b_n}{n}} \mathcal{S}_{ nb_n^{-1}} \right) \bigg) \bigg|
\to 0.
\end{aligned}
\end{equation*}
Then we obtain the result. 

\end{proof}

\begin{appendix}

\section{Green's Function Estimates} \label{sec:weakconv}

In this section, we will establish various technical estimates necessary to show weak convergence of discrete quantities to continuum quantities. 

\subsection{The property of $\tilde{G} * f $}

In this subsection, we show that $\tilde{G} * f $ is a bounded continuous function and establish the uniform convergence of $\tilde{G}_D*f \to 2\tilde{G}*f$. 
We assume that $f$ is smooth bounded function with finite support.  

\begin{lem} \label{lem:contGeps}
There is some constant such that the following estimates hold uniformly in $a$,
\begin{equation*}
| (\tilde{G}*f)(a)| \lesssim 1, 
\quad  |(\tilde{G}*f)(a+ \kappa) - (\tilde{G}*f)(a)| \lesssim \kappa. 
\end{equation*}
\end{lem}
\begin{proof}
First see that,
$$
\begin{aligned}
(\tilde{G}*f)(a) & = \int_{\|e\| \le 1} f(a-e) \tilde{G}(e) \text{d}e + \int_{\|e\| \ge 1} f(a-e) \tilde{G}(e) \text{d}e\\
& \le \sup_{z \in \mathbb{R}^4}|f(z)| \int_{\|e\| \le 1} \tilde{G}(e) \text{d}e + \left[\int_{\mathbb{R}^4} f^2(a-e) \text{d}e\right]^{1/2} \left[\int_{\|e\| \ge 1} \tilde{G}^2(e)\text{d}e \right]^{1/2}. 
\end{aligned}
$$
By applying a similar inequality, we also have that,
$$
\begin{aligned}
& \int_{\mathbb{R}^4} \|\nabla f(a-e)\| \tilde{G}(e)\text{d}e \le \int_{\|e\| \le 1} ||\nabla f(a-e)|| \tilde{G}(e) \text{d}e + \int_{\|e\| \ge 1} \|\nabla f(a-e)\| \tilde{G}(e) \text{d}e\\
&\le \sup_{z \in \mathbb{R}^4} \| \nabla f (z) \| \int_{\|e\| \le 1} \tilde{G}(e)\text{d}e + \left[\int_{\mathbb{R}^4} ||\nabla f(a-e)||^2 \text{d}e\right]^{1/2} \left[ \int_{\|e\|\ge 1} \tilde{G}^2(e) \text{d}e \right]^{1/2}
 \lesssim 1.
\end{aligned}
$$
Thus, we see that,
    \begin{equation*}
    \begin{aligned}
& |(\tilde{G}*f)(a+\kappa) - (\tilde{G}*f)(a)| 
\le \int_{ \mathbb{R}^4} |f(a+\kappa-e) -f(a-e)| \tilde{G}(e) \text{d}e \\ &\le \int_{0}^{\kappa} \text{d}t \int_{ \mathbb{R}^4} \left|\left\langle \nabla f(a+t -e), \frac{\kappa}{||\kappa||} \right\rangle \right| \tilde{G}(e) \text{d}e 
\\&\le \int_{0}^{\kappa} \text{d}t \int_{ \mathbb{R}^4}||\nabla f(a+t-e)|| \tilde{G}(e) \text{d} e\lesssim \kappa
    \end{aligned}
    \end{equation*}
and we obtain the desired result. 

\end{proof}
To introduce the next lemma, we define,
$$
(\tilde{G}_D*f)(\sqrt{n} a) = (C_f)^{-1/2}\frac{1}{n^2} \sum_{z \in \frac{1}{\sqrt{n}}\mathbb{Z}^4} \tilde{G_D}(\sqrt{n}(a -z)) f(z), \quad a \in \frac{1}{\sqrt{n}} \mathbb{Z}^4
$$
and
$$
C_f = \frac{1}{n^2}\sum_{z \in \frac{1}{\sqrt{n}}\mathbb{Z}^4} f^2(z).
$$

\begin{lem} \label{lem:Hestimate} Uniformly in $a$,  we have that as $n\to \infty$,
\begin{equation}
\begin{aligned}
    \left|  2 \int_{\mathbb{R}^4} \tilde{G}(\lfloor a \rfloor_n -e) f(e) \text{d}e  - (C_f)^{-1/2} \frac{1}{n^{2}} \sum_{e \in \frac{1}{\sqrt{n}}\mathbb{Z}^4}  n^{3/2} \tilde{G}_{D}(\sqrt{n} (\lfloor a \rfloor_n -  e) )f(e) \right| = o(1).
\end{aligned}
\end{equation}
\end{lem}

We start with a few intermediate lemmas. 
The first lemma allows us to reduce the domain of integration of $\tilde{G}*f (a)$ from all of $\mathbb{R}^4$, to an integration over a region of finite support. 
\begin{lem} \label{lem:chi}
Fix some $\delta_2>\delta_1>0$. 
Let $\chi$ be a smooth positive function supported on $[-n^{\delta_2}, n^{\delta_2}]^4$ bounded by $1$ and such that $\chi$ is 1 on $[-n^{\delta_1}, n^{\delta_1}]^4$. 
Then, uniformly in $a \in \mathbb{R}^4$ such that the following estimate holds,
\begin{equation} \label{eq:chimainresu}
    \left|\int_{\mathbb{R}^4} \tilde{G}(a-e)f(e) \text{d}e - \int_{\mathbb{R}^4} \chi(a-e)\tilde{G} (a-e) f(e) \text{d}e \right| \lesssim  n^{- \delta_1} .
\end{equation}
\end{lem}

\begin{rem}
By similar methods to Lemma \ref{lem:chi}, we would also have a corresponding equation for $\tilde{G}_D$.
Namely, we would have,
    \begin{equation*}
    \begin{aligned}
        &\bigg|\frac{1}{n^{2}}\sum_{e \in \frac{1}{\sqrt{n}} \mathbb{Z}^4} n^{3/2}\tilde{G}_D(\sqrt{n}a-\sqrt{n}e) f(e)\\
        &- \frac{1}{n^{2}}\sum_{e \in \frac{1}{\sqrt{n}} \mathbb{Z}^4} \chi(a-e) n^{3/2}\tilde{G}_D(\sqrt{n}a-\sqrt{n}e) f(e)\bigg| \lesssim n^{-\delta_1}.
    \end{aligned}
    \end{equation*}
\end{rem}

\begin{proof}
By the Cauchy-Schwartz inequality, we see that we have,
\begin{equation*} 
\begin{aligned}
    &\left|\int_{\mathbb{R}^4} \tilde{G}(a-e)f(e) \text{d}e - \int_{\mathbb{R}^4} \chi(a-e)\tilde{G} (a-e)f(e) \text{d}e \right|\\
    &\le \int_{\mathbb{R}^4} \tilde{G}(a-e)(1-\chi(a-e)) f(e) \text{d}e \\
    & \le  \left[\int_{\mathbb{R}^4} (\tilde{G}(a-e)(1-\chi(a-e)))^2 \text{d}e \right]^{1/2} \left[\int_{\mathbb{R}^4} f^2(e)\text{d}e \right]^{1/2} \\
    & \le  \left[ \int_{\mathbb{R}^4} f^2(e) \text{d}e \right]^{1/2} \left[ \int_{\|e\| \ge  n^{\delta_1}} \tilde{G}^2(e) \text{d}e \right]^{1/2}
    \lesssim  \frac{1}{n^{\delta_1}}.
\end{aligned}
\end{equation*}
It yields the desired result. 
\end{proof}

After the reduction to a region of finite support, our next lemma allows us to replace $\tilde{G}*f$ with an appropriate discrete form closer to one found in the expression of the discrete computation.

\begin{lem} 
We have the following estimates uniform in $a$. Fix some $\delta_1 >\delta_2>0$ sufficiently small, then there is some constant such that,
\begin{equation}\label{eq:tildGdisc}
\Bigg| (\tilde{G}*f)(a) - \frac{1}{n^{2}} \sum_{\substack{z \in \frac{1}{\sqrt{n}} \mathbb{Z}^4\\ \|z\| \ge n^{-\delta_1} \\ \|\lfloor a\rfloor_n-z \| \le n^{\delta_2}}} f(\lfloor a \rfloor_n -z) \tilde{G}(z) \Bigg| \lesssim  n^{-\delta_1}. 
\end{equation}
Here, $\lfloor a \rfloor_n$ denotes the element in the lattice $\frac{1}{\sqrt{n}}\mathbb{Z}^4$ that is formed by considering

$\frac{1}{\sqrt{n}}(\lfloor \sqrt{n}a_1 \rfloor, \ldots , \lfloor \sqrt{n}a_4 \rfloor)$, where we apply the least integer function to each coordinate.

Similarly, one can show that, uniformly in $a \in \frac{1}{\sqrt{n}}\mathbb{Z}^4$, 
\begin{equation} \label{eq:tildGDdisc}
    \Bigg|  \frac{1}{n^2}\sum_{z \in \frac{1}{\sqrt{n}} \mathbb{Z}^4 }n^{3/2}\tilde{G}_{D}(\sqrt{n}(a-z)) f(z) - \frac{1}{n^{2}} \sum_{\substack{z \in \frac{1}{\sqrt{n}} \mathbb{Z}^4 \\ \|z\| \ge n^{-\delta_1}\\ \|a - z\| \le n^{\delta_2}}} f(a-z) n^{3/2} \tilde{G}_D(\sqrt{n}a) \Bigg| 
    \lesssim  n^{-\delta_1}.
\end{equation}

\end{lem}
\begin{proof}
We will only consider proving equation \eqref{eq:tildGdisc}; the proof for \eqref{eq:tildGDdisc} would be simpler.  
First, observe that
\begin{equation} \label{eq:GaussInside}
 \int_{\|z\| \le n^{-\delta_1}} f(a-z) \tilde{G}(z) \text{d}z \lesssim  \int_{\|z\| \le n^{-\delta_1}} \frac{1}{\|z\|^3} \text{d} z \lesssim  n^{-\delta_1}.
\end{equation}

Secondly, we have that, for all sufficiently large $n$, 
\begin{equation} \label{eq:GaussOutside}
\begin{aligned}
 \int_{\substack{\|z\| \ge n^{-\delta_1}\\\|a-z\| \ge n^{\delta_2}}} f(a-z) \tilde{G}(z) \text{d}z=0.
\end{aligned}
\end{equation}

Combining estimates \eqref{eq:GaussInside} and \eqref{eq:GaussOutside}, we can deduce that, 
\begin{equation*}
\left|(\tilde{G}*f)(a) - \int_{\substack{\|z\| \ge n^{-\delta_1}\\ \|z-a\| \le n^{\delta_2}}} f(a-z) \tilde{G}(z) \text{d}z\right| \lesssim  n^{-\delta_1}. 
\end{equation*}

Now, we compute the difference between the quantity on the right hand side above, and the appropriate discretization. If we let $\lfloor z\rfloor_n$ denote the point in the lattice $\frac{1}{\sqrt{n}} \mathbb{Z}^4$ that is closest to $z$, then we can observe the following,
\begin{equation*}
 |\tilde{G}(z) - \tilde{G}(\lfloor z \rfloor_n)| \lesssim  n^{\delta_1 -1/2} \tilde{G}(z), 
 \quad \forall \|z\| \ge n^{-\delta_1}.
\end{equation*}
This comes from the fact that the gradient of $\|z\|^{-3}$ is $3\|z\|^{-4}[z_1,\ldots,z_4]$ which equals $3 G(z) \|z\|^{-2}[z_1,\ldots,z_4]$  and that $\|z\|^{-1} \le n^{\delta_1}$. 

In addition, if we assume that the domain of the support of $f$ is $I$, 
\begin{equation*}
|f(a-z) - f(a-\lfloor z \rfloor_n)| \lesssim n^{-1/2} \mathbbm{1}[ a-z \in I]
\end{equation*}
since $f$ is a smooth function with a bounded derivative.
Hence, applying the triangle inequality, we ultimately see that,
\begin{equation*}
\begin{aligned}
 &\bigg|\int_{\substack{\|z\| \ge n^{-\delta_1}\\ \|a-z\| \le n^{\delta_2}}} f(a-z) \tilde{G}(z) \text{d}z - \sum_{\substack{z\in \frac{1}{\sqrt{n}} \mathbb{Z}^4\\ \|z\| \ge n^{-\delta_1}\\ \|a-z\| \le n^{\delta_2}}} f(\lfloor a \rfloor_n- z ) \tilde{G}( z ) \bigg| \\&  \lesssim \int_{\substack{\|z\| \ge n^{-\delta_1}\\ \|a-z\| \le n^{\delta_2}}} |f(a-z) \tilde{G}(z) - f(\lfloor a \rfloor_n- \lfloor z \rfloor_n) \tilde{G}(\lfloor z \rfloor_n) | \text{d}z\\& \lesssim \max[n^{\delta_1 - 1/2}, n^{- 1/2}] \int_{a-z \in I} \tilde{G}(z) \text{d}z \lesssim n^{\delta_1-1/2}.
\end{aligned}
\end{equation*}
This completes the proof of the lemma.
\end{proof}

As a corollary of the lemma, we have the following estimates.
\begin{cor} \label{cor:tildGeps}
    First, fix some $\epsilon$ not changing with $n$. Additionally, fix parameters $\delta_1 > \delta_2$ sufficiently small.  For $\|a\| \le 2 n^{\delta_2}$, we have the following estimate,
    \begin{equation*}
        \left|2 \tilde{G}*f(a) - n^{3/2} (C_f)^{-1/2}\tilde{G}_{D}*f(\lfloor \sqrt{n}a \rfloor) \right| \lesssim \frac{n^{20\delta_1 + 2\delta_2}}{n} + n^{-\delta_1}. 
    \end{equation*}
\end{cor}
\begin{proof}
By using \eqref{eq:tildGdisc} and \eqref{eq:tildGDdisc}, it suffices to estimate,
\begin{equation*}
    \Bigg| \frac{1}{n^{2}} \sum_{\substack{z \in \frac{1}{\sqrt{n}} \mathbb{Z}^4 \\ \|z\| \ge n^{-\delta_1}\\ \|a - z\| \le n^{\delta_2}}} f(a-z) \left[ n^{3/2} (C_f)^{-1/2}\tilde{G}_D(\sqrt{n}z) - 2 \tilde{G}(z) \right]\Bigg|.
\end{equation*}
From equation \eqref{eq:diffGDG}, which we can apply since if $\|a\| \le n^{\delta_2}$, then $\|z\| \le 2 n^{\delta_2}$ in the sum above, we can bound the quantity above as,
\begin{equation*}
\begin{aligned}
    &\lesssim \Bigg| \frac{1}{n^{2}} \sum_{\substack{z \in \frac{1}{\sqrt{n}} \mathbb{Z}^4\\ \|z\| \ge n^{-\delta_1} \\ \|a-z\| \le n^{\delta_2}}} f(a-z) \frac{n^{20 \delta_1}}{n\|z\|^2} \Bigg| \lesssim \Bigg| \frac{1}{n^{2}} \sum_{\substack{z \in \frac{1}{\sqrt{n}} \mathbb{Z}^4\\ \|z\| \ge n^{-\delta_1} \\ \|a-z\| \le 2n^{\delta_2}}} \frac{n^{20\delta_1}}{n\|z\|^2} \Bigg|\\
    & \lesssim n^{20\delta_1} \int_{ 3n^{\delta_2} \ge \|z\| \ge n^{-\delta_1}} \frac{1}{n\|z\|^2} \text{d}z \lesssim \frac{n^{20\delta_1 +2\delta_2}}{n}. 
\end{aligned}
\end{equation*}
In our application of equation \eqref{eq:diffGDG}, we made the choice of parameter $\epsilon = 8 \delta_1$.  
\end{proof}

We finally have all results necessary to prove Lemma \ref{lem:Hestimate}.

\begin{proof}[Proof of Lemma \ref{lem:Hestimate}]
Fix parameters $\delta_1,\delta_2,\delta_3>0$ sufficiently small satisfying $\frac{1}{400}>\delta_1 > 20\delta_2 > 20 \delta_3 >0$.  
Recalling the function $\chi$ from Lemma \ref{lem:chi}, we set $\chi$ to a be a smooth function supported on the interval $[-n^{\delta_2},n^{\delta_2}]^4$ and $1$ on $[-n^{\delta_3},n^{\delta_3}]^4$, 
\begin{equation*}
\begin{aligned}
    &\Bigg| 2 \int_{\mathbb{R}^4} \chi(a-e)\tilde{G} (a-e) f(e)\text{d}e  \\ &- \frac{1}{n^{2}}\sum_{e \in \frac{1}{\sqrt{n}} \mathbb{Z}^4} \chi(\lfloor a \rfloor_n-e) n^{3/2}\tilde{G}_{D}(\sqrt{n} \lfloor a\rfloor_n - \sqrt{n}e)  f(e) \Bigg| = o(1).
\end{aligned}
\end{equation*}
Thus, we see that it suffices to show that,
\begin{equation*}
\begin{aligned}
    & \int_{\mathbb{R}^4} \chi(a-e)\left| 2 \tilde{G}(a-e) -n^{3/2} \tilde{G}_{D}(\sqrt{n}(\lfloor a \rfloor_n -e)) \right| f(e) \text{d}e  = o(1).
\end{aligned}
\end{equation*}
By Corollary \ref{cor:tildGeps}, we can bound the difference of $\tilde{G}$ and $\tilde{G}_{D}$ in the region on which $\chi$ is not equal to $0$. 
Thus, we have,
\begin{equation*}
    \begin{aligned}
        &\int_{\mathbb{R}^4} \chi(a-e)\left| 2 \tilde{G}(a-e) -n^{3/2} \tilde{G}_{D}(\sqrt{n}(\lfloor a \rfloor_n -e)) \right| f(e) \text{d}e\\& 
        \lesssim \int_{\mathbb{R}^4} \chi(a-e) n^{-\delta_1} f(e) \text{d}e 
        \lesssim n^{- \delta_1 }. 
    \end{aligned}
\end{equation*}
We used the fact that $\chi$ is supported on $[-n^{\delta_2},n^{\delta_2}]$. 
This completes the proof of the lemma.

\end{proof}

\subsection{Additional Green's function computations}
In this subsection, we will give various useful estimates concerning Green's function. 
\begin{lem} \label{lem:posdisc}
The Green's function of the discrete random walk $G_D(x)$ has a positive convolutional square root with the following form,
\begin{equation*}
    \tilde{G}_D(z) = \sum_{n=0}^{\infty} \frac{(2n)!}{2^{2n}(n!)^2} p_n(z).
\end{equation*}
Recall that $p_n(z)$ is the transition probability that a simple random walk starting from $0$ reaches the point $z$ at time $n$. 
There is an  $L^1$ function $\hat{\tilde{G}}_D(l)$ whose Fourier transform is the function $\tilde{G}_D(x)$. 
\end{lem}
\begin{proof}

\textit{Part 1: Derivation of the form of $\tilde{G}_D$}

Consider the Taylor expansion of $(1-x)^{-1/2}$ as,
\begin{equation*}
    \frac{1}{\sqrt{1-x}} = \sum_{k=0}^{\infty} C_k x^k.
\end{equation*}
We will show that $\tilde{G}_D$ has to take the functional form,
\begin{equation*}
\tilde{G}_D(z) = \sum_{k=0}^{\infty} C_k p_k(z).
\end{equation*} 
We can check this by directly computing $\tilde{G}_D*\tilde{G}_D$. 
Thus, we have that, for any $ z \in \mathbb{Z}^4$, 
\begin{equation*}
\begin{aligned}
\tilde{G}_D *\tilde{G}_D(z)&
=  \sum_{x \in \mathbb{Z}^4}\sum_{k_1,k_2=0}^{\infty} C_{k_1} C_{k_2} p_{k_1}(x) p_{k_2}(z-x)
\\
&= \sum_{k_1,k_2=0}^{\infty} C_{k_1} C_{k_2} p_{k_1+k_2}(z)\\
&=  \sum_{k=0}^{\infty} p_k(z) \sum_{k_1=0}^{k} C_{k_1} C_{k - k_1}
= \sum_{k=0}^{\infty} p_k(z).
\end{aligned}
\end{equation*}
To get the last line, we used the fact that $$
\frac{1}{1-x} = \left(\frac{1}{\sqrt{1-x}} \right)^2 = \left(\sum_{k=0}^{\infty} C_k x^k\right)^2 = \sum_{k=0}^{\infty} x^k \sum_{k_1=0}^k C_{k_1} C_{k-k_1}.
$$
This gives the identity that $\sum_{k_1=0}^k C_{k_1} C_{k-k_1} = 1$ by comparing coefficients of the Taylor Series. 
By using similar manipulations, one can show that $\tilde{G}(x) = \int_{0}^{\infty} \frac{1}{\sqrt{\pi t} } P_t(x) \text{d}t$, 
where $P_t(x)$ is the probability density that a Brownian motion starting from zero would reach position $x$ at time $t$.

\textit{Part 2: Derivation of the Fourier Transform}

Now, we discuss the Fourier transform of $\tilde{G}_D(x)$.
Consider the following function,
$$
F(l_1,\ldots,l_4) =  \frac{1}{\sqrt{1 - \frac{1}{4}\sum_{i=1}^4 \cos(2\pi l_i)}}.
$$
We will show that,
\begin{equation*}
 \tilde{G}_D(a_1,\ldots,a_4) = \int_{(-1/2,1/2]^4} F(l_1,\ldots,l_4) \prod_{i=1}^4 \exp[-2 \pi \text{i} l_i a_i]  \text{d}l_i. 
\end{equation*}

First of all, observe that $F(l_1,\ldots,l_4)$ only has a singularity around the origin and, furthermore, around the origin, $F$ behaves like $\frac{1}{\sqrt{l_1^2+\ldots + l_4^2}}$. 
Thus, $F$ is integrable around $0$.  If we let $B_{\epsilon}(x)$ be the ball of radius $\epsilon$ around $x$ , we have,
\begin{equation} \label{eq:expandF}
\begin{aligned}
&\int_{(-1/2,1/2]^4} F(l_1,\ldots,l_4) \prod_{i=1}^4 \exp[-2 \pi \text{i} l_i a_i]  \text{d}l_i \\ &= \lim_{\epsilon \to 0 }\int_{(-1/2,1/2]^4 \setminus B_{\epsilon}(0)} \frac{1}{\sqrt{1-\frac{1}{4} \sum_{i=1}^4 \cos(2 \pi l_i)}} \prod_{i=1}^4 \exp[-2 \pi \text{i} l_i a_i]  \text{d}l_i.
\end{aligned}
\end{equation}

Now, away from the singularity at $0$, we can expand $\frac{1}{\sqrt{1- \frac{1}{4} \sum_{i=1}^4 \cos(2 \pi l_i)}}$ as,
\begin{equation*}
    \sum_{k=0}^\infty C_k \left( \frac{1}{4} \sum_{i=1}^4 \cos(2\pi l_i)\right)^k
\end{equation*}
and  observe that,
\begin{equation*}
\begin{aligned}
    &\int_{(-1/2,1/2]^4} \left( \frac{1}{4} \sum_{i=1}^4 \cos(2\pi l_i)\right)^k \prod_{i=1}^4 \text{d} l_i \\
    \le & \left[\int_{(-1/2,1/2]^4} \left( \frac{1}{4}  \sum_{i=1}^4 \cos(2\pi l_i)\right)^{2k} \prod_{i=1}^4 \text{d} l_i \right]^{1/2}
    = \sqrt{p_{2k}(0)},
    \end{aligned}
\end{equation*}
where the last equality comes from direct integration. 
By using the asymptotic that $C_k \lesssim \sqrt{k}^{-1}$ and $p_{2k}(0) \lesssim  k^{-2}$. 
We see that,
\begin{equation*}
    \sum_{k=0}^{\infty} C_k \int_{(-1/2,1/2]^4} \left| \frac{1}{4}  \sum_{i=1}^4 \cos(2 \pi l_i)\right|^k \prod_{i=1}^4 \text{d} l_i \lesssim  \sum_{k=1}^{\infty} k^{-1/2 - 1} < \infty.
\end{equation*}
This control on the absolute value of the integral allows us to freely exchange the summation of the power series, the limit as $\epsilon \to 0$, and the integration in \eqref{eq:expandF}. Thus, we have that,
\begin{equation*}
\begin{aligned}
    & \lim_{\epsilon \to 0} \int_{(-1/2,1/2]^4\setminus B_{\epsilon}(0)} \sum_{k=0}^{\infty} C_k \left( \frac{1}{4} \sum_{i=1}^4 \cos(2\pi l_i) \right)^k \prod_{i=1}^4 \exp[-2\pi \text{i} l_i a_i]  \text{d}l_i\\
    &= \sum_{k=0}^{\infty} C_k \lim_{\epsilon \to 0} \int_{(-1/2,1/2]^4 \setminus B_{\epsilon}(0)} \left( \frac{1}{4} \sum_{i=1}^4 \cos(2\pi l_i) \right)^k \prod_{i=1}^4 \exp[-2\pi \text{i} l_i a_i]  \text{d}l_i\\
    &= \sum_{k=0}^{\infty} C_k \int_{(-1/2,1/2]^4} \left( \frac{1}{4} \sum_{i=1}^4 \cos(2\pi l_i) \right)^k \prod_{i=1}^4 \exp[-2\pi \text{i} l_i a_i] \text{d}l_i = \sum_{k=0}^{\infty} C_k p_k(a_1,\ldots,a_4). 
\end{aligned}
\end{equation*}
To get the last line, observe that $\frac{1}{4} \sum_{i=1}^4 \cos(2\pi l_i)$ can be written as,
$$
\frac{1}{4} \sum_{i=1}^4 \cos(2\pi l_i) = \frac{1}{8}\left(\sum_{i=1}^4[\exp[2\pi l_i] + \exp[-2 \pi l_i]] \right).
$$
The Fourier integral in the last line thus determines the coefficient of the term $\prod_{i=1}^4 \exp[2 \pi \text{i} l_i a_i]$ in the expansion of the polynomial. This is exactly the number of ways that a simple random walk will reach the point $(a_1,\ldots,a_4)$ at time $k$.

\end{proof}

Though, this will be more useful in the sequel, we also present the following result comparing $\tilde{G}_D$ to $\tilde{G}$. 
\begin{lem}
We have the following asymptotics relating $\tilde{G}_D(z)$ with $\tilde{G}(z)$.
Fix some $\epsilon>0$ and let $ \|z\| \ge n^{-\epsilon/4}$. Then, we have the following comparison, 
\begin{equation} \label{eq:diffGDG}
|(\sqrt{n})^3\tilde{G}_D(\sqrt{n} z) - 2\tilde{G}(z)| \lesssim \frac{n^{5\epsilon/2}}{n\|z\|^2} + n^{2} \exp[-n^{\epsilon/2}].
\end{equation}
\end{lem}

\begin{rem}
The bound found in the inequality \eqref{eq:diffGDG} is most effective when $\|z\| \le n^{\epsilon}$, which will be the regime in which we will actually apply the bound in question. 
\end{rem}

\begin{proof}
\textit{Part 1: Discretization of the integral form of $\tilde{G}$}

We will begin our computation by first finding an appropriate discretization of the integral form of $\tilde{G}$. 
Recall that we can write $\tilde{G}$ as,
\begin{equation*}
    \tilde{G}(z) = \int_0^{\infty} \frac{1}{\sqrt{\pi t}} P_{t}(z) \text{d}t.  
\end{equation*}
Let $Q_z(t)$ be a shorthand for the function $Q_z(t) = \frac{1}{\sqrt{\pi t}} P_t(z)$. 
We can estimate the difference as follows: 
\begin{align}\notag
\left|\int_0^\infty Q_z(t) \text{d}t - \sum_{k \in \frac{1}{n} \mathbb{Z}^+} Q_z(k)\right|
&\le \sum_{k \in \frac{1}{n} \mathbb{Z}^+} \int_{\frac{(k-1)}{n}}^{\frac{k}{n}} \text{d} j\int_{j}^{\frac{k}{n}} |Q_z'(l)| \text{d}l \\
\label{greenint1}
& = \sum_{k \in \frac{1}{n} \mathbb{Z}^+} \int_{\frac{(k-1)}{n}}^{\frac{k}{n}} \left( l - \frac{k-1}{n}  \right) |Q_z'(l)| \text{d}l
\le \frac{1}{n} \int_0^{\infty} |Q_z'(l)| \text{d}l.
\end{align}
One can explicitly compute $|Q_z'(l)|$ as $Q_z'(l) \propto \exp[-\|z\|^2/(2l) ]\left[-5 l^{-7/2} +  \|z\|^2 l^{-9/2}  \right]$. 
By scaling, we observe that $\int_0^{\infty}|Q_z'(l)| \text{d} l = \frac{1}{\|z\|^{5}} \int_0^{\infty}|Q_{e_1}'(l)| \text{d} l $, where $e_1$ is the unit vector in the first dimension and the latter integral is finite. 
Thus, the error between $\int_0^{\infty} Q_z(t) \text{d}t$  and its discretization with lattice $\frac{1}{n}\mathbbm{Z}^+$ is of order $O\left(\frac{1}{n\|z\|^{5}}\right) $, where the implicit constant does not depend on either $\|z\|$ or $n$.

Furthermore, we claim that we can ignore the portion of the integral of $Q_z(t)$ from $t$ between $0$ and $n^{-\epsilon}$ in our regime of interest. By observing the form of the derivative of $Q_z(t)$, we notice that $Q_z(l)$ is an increasing function as long as $\|z\|^2 \ge 5 l$. For, $l \le n^{-\epsilon}$ and $\|z\| \ge n^{-\epsilon/4}$, we see that $Q_z(l)$ is increasing between $l=0$ and $l= n^{-\epsilon}$. 
Thus, 
\begin{align}\label{greenint2}
    \int_0^{n^{-\epsilon}} Q_z(l) \text{d}l \le n^{-\epsilon} Q_z(n^{-\epsilon}) \lesssim n^{3\epsilon/2} \exp[-n^{\epsilon/2}].
\end{align} 
Combining \eqref{greenint1} and \eqref{greenint2}, we see that,
\begin{equation} \label{eq:Gdiscaway}
    \left|\tilde{G}(z) -  \frac{1}{n} \sum_{k \in \frac{1}{n}\mathbb{Z}^+, k \ge n^{-\epsilon}} \frac{1}{\sqrt{\pi t}} P_t(z)\right| \lesssim \frac{1}{n\|z\|^{5}} + n^{3\epsilon/2}\exp[-n^{\epsilon/2}]. 
\end{equation}

\textit{Part 2: Estimates on $\tilde{G}_D$}

First, we will bound the contribution of $n\sum_{k=0}^{n^{1-\epsilon}} C_k p_k(\sqrt{n} z)$. 
Since $\|z\| \ge n^{-\epsilon/4}$, we have that $\sqrt{n}\|z\| \ge n^{1/2- \epsilon/4}$. 
By exponential tail estimates on discrete random walks, we know that $p_k(\sqrt{n}z) \lesssim \exp[- n \|z\|^2/k] \lesssim \exp[-n^{\epsilon/2} ]$. 
Thus, the contribution of $n^{2} \sum_{k=0}^{n^{1-\epsilon}}C_k p_k(n^{1/2}z) \lesssim n^{2} \exp[-n^{\epsilon/2}]$.  
Ultimately, we see that,
\begin{equation} \label{eq:GDawayorig}
    \left| n^{3/2} \tilde{G}_D(n^{1/2}z) - \frac{1}{n} \sum_{k \in \frac{1}{n} \mathbb{Z}^+, k \ge n^{-\epsilon}} (\sqrt{n} C_{nk}) (n^{2} p_{nk} (\sqrt{n}z) )\right| \lesssim n^{2} \exp[-n^{\epsilon/2}].
\end{equation}

By Stirling's approximation, we have,
$$
    \begin{aligned}
         &C_{nk}=\frac{(2nk)!}{2^{2nk}((nk)!)^2} \le \frac{\sqrt{2\pi(2nk)} \left(\frac{2nk}{e}\right)^{2nk} \exp[\frac{1}{12(2nk)}]}{ 2^{2nk}\left(\sqrt{2\pi(nk)} \left( \frac{nk}{e}\right)^{nk} \right)^2 \exp[\frac{2}{12nk+1}]} = \frac{1}{\sqrt{\pi nk}} \left[ 1+ \frac{O(1)}{nk} \right],\\
         &C_{nk}=\frac{(2nk)!}{2^{2nk}((nk)!)^2} \ge \frac{\sqrt{2\pi(2nk)} \left(\frac{2nk}{e}\right)^{2nk} \exp[\frac{1}{12(2nk)+1}]}{ 2^{2nk}\left(\sqrt{2\pi(nk)} \left( \frac{nk}{e}\right)^{nk} \right)^2 \exp[\frac{2}{12nk}]} = \frac{1}{\sqrt{\pi nk}} \left[ 1+ \frac{O(1)}{nk} \right].
    \end{aligned}
$$
Thus, we see that,
\begin{equation*}
    \left|\sqrt{n} C_{nk} - \frac{1}{\sqrt{\pi k}}\right| \lesssim \frac{1}{n k^{3/2}}.
\end{equation*}
By the local central limit theorem \cite[Thm 2.1.1]{LL10}, we also have that,

\begin{equation*}
|n^{2} p_{nk}(\sqrt{n}z) -  n^{2}P_{nk/4}(\sqrt{n}z)| \lesssim \frac{1}{n k^3\|z\|^2}.
\end{equation*}
Furthermore, by scaling, $n^2 P_{nk/4}(\sqrt{n}z) $ is equal to $ 16 P_{k}( 2 z)$. 
If we combine these estimates, we see that,
\begin{equation*}
\begin{aligned}
    & \left|\frac{1}{n} \sum_{k \in \frac{1}{n} \mathbb{Z}^+, k \ge n^{-\epsilon}} (\sqrt{n} C_{nk}) (n^{2} p_{nk}(\sqrt{n}z)) - \frac{1}{n} \sum_{k \in \frac{1}{n} \mathbb{Z}^+, k \ge n^{-\epsilon}} \frac{16}{\sqrt{\pi k}} P_{k}(2 z)\right| \\
    \lesssim  &\frac{1}{n} \sum_{k \in \frac{1}{n} \mathbb{Z}^+, k \ge n^{-\epsilon}} (\sqrt{n}C_{nk})|n^{2} p_{nk}(\sqrt{n}z) - 16 P_k(2 z)| \\
    + &\frac{1}{n} \sum_{k \in \frac{1}{n} \mathbb{Z}^+, k \ge n^{-\epsilon}} \left|\sqrt{n}C_{nk} - \frac{1}{\sqrt{\pi k}}\right| 16 P_{k}(2 z)\\
     \lesssim  & \frac{1}{n} \sum_{k \in \frac{1}{n} \mathbb{Z}^+, k \ge n^{-\epsilon}} \frac{1}{\sqrt{ k}} \frac{1}{n k^{3} \|z\|^2} 
    \lesssim \frac{n^{5\epsilon/2}}{n\|z\|^2}. 
\end{aligned}
\end{equation*}
In the last line, we used the estimates $\sqrt{n}C_{nk} \lesssim \frac{1}{\sqrt{k}}$ and $P_k(z) \lesssim \frac{1}{k\|z\|^2}$. This, in itself, comes from the estimate that $\exp[-\|z\|^2/k] \le k/\|z\|^2$. 
Combining this with equation \eqref{eq:GDawayorig} shows that,
$$
\left|n^{3/2} \tilde{G}_D(\sqrt{n}z)) -  \frac{1}{n} \sum_{k \in \frac{1}{n} \mathbb{Z}^+, k \ge n^{-\epsilon}} \frac{1}{\sqrt{\pi k}} 16 P_{k}(2 z)\right| \lesssim  \frac{n^{5\epsilon/2}}{n\|z\|^2} + n^{2}\exp[-n^{\epsilon/2}]. 
$$
Finally, combining this estimate with equation \eqref{eq:Gdiscaway} will give us the desired inequality in equation \eqref{eq:diffGDG}.

\end{proof}

The following lemma gives a rough estimate on sum of the Green's function over a random walk whose beginning and end are pinned to certain points.

\begin{lem}\label{lem:BridgeBound}
For any $y$ and $z\in\Z^4$, 
    \begin{equation*}
\begin{aligned}
\mathbb{E}\left[ \sum_{i=0}^n G_D(\mathcal{S}_i - z)\bigg| \mathcal{S}_n = y \right]
\lesssim \log n.
\end{aligned}
 \end{equation*}
\end{lem}
\begin{proof}
We let $B_{\sqrt{n}}(z)$ be the ball of radius $\sqrt{n}$ around the bound $z$.
 Let $\tau_1$ be the (random variable) that is the first time that the random bridge touches a point in $B_{\sqrt{n}}(z)$. Let $\tau_2$ be the last time that the random bridge touches a point in $B_{\sqrt{n}}(z)$.

 We see that,
 \begin{equation*}
\begin{aligned}
&\mathbb{E}\left[ \sum_{i=0}^n G_D(\mathcal{S}_i - z)\bigg| \mathcal{S}_n = y \right]\\
&=\mathbb{E}\bigg[ \sum_{ 0 \le k_1 \le k_2 \le n} \sum_{a_1,a_2 \in B_{\sqrt{n}}(z)}  \mathbbm{1}[\tau_1 = k_1, \tau_2 = k_2, \mathcal{S}_{\tau_1} =a_1, S_{\tau_2} = a_2]\\ &  \hspace{ 1 cm} \times\mathbb{E}\left[ \sum_{i=0}^n G_D(\mathcal{S}_i - z) \bigg| \tau_1 = k_1, \tau_2 = k_2,  \mathcal{S}_{\tau_1} = a_1, \mathcal{S}_{\tau_2}= a_2, \mathcal{S}_n = y\right]\bigg| \mathcal{S}_n = y \bigg] \\
&= \mathbb{E}\bigg[ \sum_{ 0 \le k_1 \le k_2 \le n} \sum_{a_1,a_2 \in B_{\sqrt{n}}(z)}  \mathbbm{1}[\tau_1 = k_1, \tau_2 = k_2, \mathcal{S}_{\tau_1} =a_1, \mathcal{S}_{\tau_2} = a_2]\\& \hspace{1 cm}\times \mathbb{E}\left[ \sum_{i=k_1}^{k_2} G_D(\mathcal{S}_i - z)\bigg| \tau_1 = k_1, \tau_2 = k_2,  \mathcal{S}_{\tau_1} = a_1, \mathcal{S}_{\tau_2}= a_2,\mathcal{S}_n = y\right] \bigg|\mathcal{S}_n = y \bigg] \\
&+\mathbb{E}\bigg[ \sum_{ 0 \le k_1 \le k_2 \le n} \sum_{a_1,a_2 \in B_{\sqrt{n}}(z)}  \mathbbm{1}[\tau_1 = k_1, \tau_2 = k_2, \mathcal{S}_{\tau_1} =a_1, \mathcal{S}_{\tau_2} = a_2,\mathcal{S}_n =y]\\ & \hspace{1 cm}\times\mathbb{E}\left[ \sum_{i=0}^{k_1} G_D(\mathcal{S}_i - z)\bigg| \tau_1 = k_1, \tau_2 = k_2,  \mathcal{S}_{\tau_1} = a_1, \mathcal{S}_{\tau_2}= a_2,\mathcal{S}_n = y \right] \bigg|\mathcal{S}_n = y \bigg]\\
&+ \mathbb{E}\bigg[ \sum_{ 0 \le k_1 \le k_2 \le n} \sum_{a_1,a_2 \in B_{\sqrt{n}}(z)}  \mathbbm{1}[\tau_1 = k_1, \tau_2 = k_2, \mathcal{S}_{\tau_1} =a_1, \mathcal{S}_{\tau_2} = a_2] \\ & \hspace{1cm} \times\mathbb{E}\left[ \sum_{i=k_2}^{n} G_D(\mathcal{S}_i - z)| \tau_1 = k_1, \tau_2 = k_2,  \mathcal{S}_{\tau_1} = a_1, \mathcal{S}_{\tau_2}= a_2,\mathcal{S}_n = y\right] \bigg|\mathcal{S}_n = y \bigg].
\end{aligned}
 \end{equation*}

 For the last two summands, we can make the following observation.
Since we have that $G_D(\mathcal{S}_i-z) \lesssim \frac{1}{n}$ for $i \le \tau_1$ and $i \ge \tau_2$. Thus, $$\sum_{i=0}^{\tau_1} G_D(\mathcal{S}_i-z) \lesssim n \frac{1}{n}=1, $$ and $$\sum_{i=\tau_2}^n G_D(\mathcal{S}_i-z) \lesssim 1.$$
Hence,
\begin{equation*}
    \begin{aligned}
&\mathbb{E}\bigg[ \sum_{ 0 \le k_1 \le k_2 \le n} \sum_{a_1,a_2 \in B_{\sqrt{n}}(z)}  \mathbbm{1}[\tau_1 = k_1, \tau_2 = k_2, \mathcal{S}_{\tau_1} =a_1, \mathcal{S}_{\tau_2} = a_2]\\ & \hspace{1 cm}\times\mathbb{E}\left[ \sum_{i=0}^{k_1} G_D(\mathcal{S}_i - z)| \tau_1 = k_1, \tau_2 = k_2,  \mathcal{S}_{\tau_1} = a_1, \mathcal{S}_{\tau_2}= a_2,\mathcal{S}_n = y\right] \bigg|\mathcal{S}_n = y \bigg]\\
& \le \mathbb{E}\bigg[ \sum_{ 0 \le k_1 \le k_2 \le n} \sum_{a_1,a_2 \in B_{\sqrt{n}}(z)}  \mathbbm{1}[\tau_1 = k_1, \tau_2 = k_2, \mathcal{S}_{\tau_1} =a_1, \mathcal{S}_{\tau_2} = a_2] \bigg| \mathcal{S}_n = y\bigg] \lesssim 1.
    \end{aligned}
\end{equation*}

Now, all that is left to check is that,
\begin{equation*}
\begin{aligned}
& \mathbb{E}\left[ \sum_{i=k_1}^{k_2} G_D(\mathcal{S}_i - z)\bigg| \tau_1 = k_1, \tau_2 = k_2,  \mathcal{S}_{\tau_1} = a_1, \mathcal{S}_{\tau_2}= a_2,\mathcal{S}_n = y\right]\\
& \le \mathbb{E}\left[ \sum_{i=k_1}^{k_2} G_D(\mathcal{S}_i - z)\bigg|   \mathcal{S}_{k_1} = a_1, \mathcal{S}_{k_2}= a_2\right] \lesssim \log n.
\end{aligned}
\end{equation*}
It suffices to find a bound on the following for general $T$ and a random walk $\mathcal{S}$:
\begin{equation*}
\mathbb{E}\left[  \sum_{i=1}^T G_D(\mathcal{S}_i) 
\bigg| \mathcal{S}_0 =x,  \mathcal{S}_T=y \right]\lesssim \log T.
\end{equation*}
Recall \cite[Thm. 1.2.1]{LA91} yields that for some $C$ finite and any $\|x\| \le i^{1/2}$,
\begin{align*}
\mathbb{P}(\mathcal{S}_i=x)\gtrsim i^{-2}. 
\end{align*}
Then, with \eqref{gre2}, if $\|x-y\|\le T^{1/2}$, 
\begin{equation*}
\mathbb{P}( \mathcal{S}_0 =x,  \mathcal{S}_T=y )
\gtrsim T^{-2}\exp(-\frac{\|x-y\|^2}{T}) \gtrsim T^{-2}
\end{equation*}
and 
\begin{align*}
&\mathbb{E}\left[  \sum_{i=1}^T G_D(\mathcal{S}_i) 
\mathbbm{1} \{ \mathcal{S}_0 =x, \mathcal{S}_T=y \}
\right]\\
= & \sum_{i=0}^{T/2} \sum_{z \in \Z^4}
G_D(z) \mathbb{P}^x(\mathcal{S}_i=z)\mathbb{P}^z(\mathcal{S}_{T-i}=y)
+\sum_{i=T/2}^T \sum_{z \in \Z^4}
G_D(z) \mathbb{P}^z(\mathcal{S}_i=x)\mathbb{P}^y(\mathcal{S}_{T-i}=z)\\
\lesssim & \sum_{i=0}^{T/2} \sum_{z \in \Z^4}
G_D(z) \mathbb{P}^x(\mathcal{S}_i=z)  (T-i)_+^{-2}
+\sum_{i=T/2}^{T} \sum_{z \in \Z^4}
G_D(z) \mathbb{P}^y(\mathcal{S}_{T-i}=z)  i_+^{-2}\\
= & \sum_{i=0}^{T/2} \mathbb{E}[G_D(\mathcal{S}_i)]  (T-i)_+^{-2}
+\sum_{i=T/2}^T \mathbb{E}[G_D(\mathcal{S}_{T-i})]  i_+^{-2}\\
\lesssim & \sum_{i=0}^{T/2}  i_+^{-1} (T-i)_+^{-2} 
+\sum_{i=T/2}^T  (T-i)_+^{-1} i_+^{-2} 
\lesssim \frac{\log T}{T^2}.
\end{align*}
Therefore, we have the result. 
\end{proof}

\begin{lem} \label{clm:matinv}
Recall the matrix $\mathcal{G}^{S^2_{\beta,\alpha,j}}$ from equation \eqref{eq:defcalG}. This matrix $\mathcal{G}^{S^2_{\beta,\alpha,j}}$ is positive definite and has minimum eigenvalue greater than $\frac{1}{2}$. 
\end{lem}

\begin{proof}
We will show this proposition for any general matrix of the form,
$$
[\mathcal{G}]_{i,j} = G_D(a_i- a_j),
$$
where $\{a_i\}$ is a collection of $n$ distinct points.
Note that we have the Fourier transformation,
$$
G_D(x)= \int_{[0,1]^4} \frac{1}{1- \frac{1}{4}\sum_{i=1}^4 \cos(2\pi k_i)} \exp[2 \pi \text{i} \langle k, x \rangle] \text{d} k,
$$
where $\langle k, x\rangle$ is the inner product between the vector $k$ and $x$.
Let $(v_1,\ldots,v_n)$ be any vector with $l^2$ norm $1$.
Thus, we have,
\begin{equation*}
\begin{aligned}
\sum_{i,j} v_i [\mathcal{G}]_{i,j} \overline{v_j} & =  \int_{[0,1]^4} \frac{\left| \sum_{i=1}^n v_i \exp[2 \pi \text{i} \langle k, a_i \rangle] \right|^2}{1 -\frac{1}{4}\sum_{i=1}^4 \cos(2\pi k_i) } \text{d}k
\\ &\ge \frac{1}{2} \int_{[0,1]^4} \left| \sum_{i=1}^n v_i \exp[2 \pi \text{i} \langle k, a_i \rangle] \right|^2 \text{d}k = \frac{1}{2} \|v\|^2.
\end{aligned}
\end{equation*}
This shows that any matrix of the form $\mathcal{G}$.
\end{proof}



\subsection{Generalized Gagliardo-Nirenberg constant}

 In our previous manuscript \cite{AO}, we showed that the large deviation constant associated to the quantity $\int_0^1 \int_0^1 G(B^1_t- B^2_s) \text{d}s \text{d}t$ can be associated to the optimal constant of the generalized Gagliardo-Nirenberg inequality. Namely,
\begin{rem}\label{ggnfi}
We have
\begin{equation*}
\lim_{T \to \infty}
T^{-1}\log \mathbb{P}\left(\int_0^1 \int_0^1 G(B^1_t - B^2_s) \text{d}t\text{d}s  \ge T \right)
=-\tilde{\kappa}^{-4}(4,2).
\end{equation*}
\end{rem}

We remark that this large deviation constant was also obtained by Bass-Chen-Rosen in \cite[(1.10)]{BCR2}. Their result is not presented in the same manner, since they do not identify the generalized Gagliardo-Nirenberg inequality. Some manipulations, based on  Section 4 of \cite{AO} and Section 7 of \cite{BCR2}, can demonstrate the link between these constants. We remark that in order to adapt the results of \cite{BCR2} to the case of the Brownian motion, one has to adjust the Fourier transform appearing in \cite[equation (1.1)]{BCR2} by a factor of $1/2$. 
\end{appendix}
%
%

\begin{acks}[Acknowledgments]
The authors would like to thank Amir Dembo for his useful suggestions. 
The authors are also grateful to Makoto Nakamura for his helpful comments.
\end{acks}
\begin{funding}
 The first author was supported by NSF grant DMS 2102842.

 The second author was supported by JSPS KAKENHI Grant-in-Aid  for Early-Career Scientists (No.~JP20K14329) 
\end{funding}

\end{document}